\documentclass{article}
\usepackage[utf8]{inputenc}
\usepackage{amsmath}
\usepackage{amsthm}
\usepackage{amssymb,amsfonts}
\usepackage{array}
\usepackage{mathrsfs}
\usepackage{hyperref}
\usepackage[capitalize]{cleveref}
\usepackage{latexsym}
\usepackage{todonotes}
\usepackage{float}
\usepackage{nicefrac}
\usepackage{epsfig}
\usepackage{stmaryrd}
\usepackage{blkarray}
\usepackage{setspace}
\usepackage{tikz-cd}
\usepackage{enumerate}
\usepackage[all]{xypic}
\usepackage{bbm,ifpdf,tikz}
\ifpdf
\usepackage{pdfsync}
\fi
\usetikzlibrary{positioning,matrix,arrows,decorations.pathmorphing}
\usepackage{enumitem}
\usepackage{extarrows}
\counterwithin{equation}{section}
\usepackage{leftindex}
\usepackage{subcaption}
\usepackage{adjustbox}
\usepackage{booktabs}
\usepackage[title]{appendix}

\usetikzlibrary{positioning,arrows}
\usetikzlibrary{decorations.markings}

\makeatletter
\newtheorem*{rep@thm}{\rep@title}
\newcommand{\newrepthm}[2]{%
\newenvironment{rep#1}[1]{%
 \def\rep@title{#2 \ref{##1}}%
 \begin{rep@thm}}%
 {\end{rep@thm}}}
\makeatother
\newrepthm{thm}{Theorem}
\newrepthm{prop}{Proposition}
\newtheorem{thm}{Theorem}[section]
\newtheorem{prop}[thm]{Proposition}

\newtheorem{cor}[thm]{Corollary}
\newtheorem{lemma}[thm]{Lemma}

\theoremstyle{theorem}
\newtheorem{maintheorem}{Theorem}	

\crefname{maintheorem}{theorem}{theorems}

\theoremstyle{definition}
\newtheorem{ex}[thm]{Example}
\newtheorem{definition}[thm]{Definition}
\newtheorem{qst}[thm]{Question}

\newtheorem{rmk}[thm]{Remark}
\crefname{rmk}{remark}{remarks}

\makeatletter
\newenvironment{subdefinition}[1]{%
  \def\subdefinitioncounter{#1}%
  \refstepcounter{#1}%
  \protected@edef\theparentnumber{\csname the#1\endcsname}%
  \setcounter{parentnumber}{\value{#1}}%
  \setcounter{#1}{0}%
  \expandafter\def\csname the#1\endcsname{\theparentnumber.\Alph{#1}}%
  \ignorespaces
}{%
  \setcounter{\subdefinitioncounter}{\value{parentnumber}}%
  \ignorespacesafterend
}
\makeatother
\newcounter{parentnumber}

\setenumerate{label=(\arabic*)}

\newcommand{\rk}{\operatorname{rk}}
\newcommand{\cl}{\operatorname{cl}}

\renewcommand{\dim}{\operatorname{dim}}
\newcommand{\pol}{\Pi^\uparrow}
\newcommand{\proj}{\Pi^\downarrow}

\newcommand{\vp}{\varphi}

\newcommand{\td}[1]{\,\widetilde{#1}\,}
\newcommand{\ol}[1]{\overline{#1}}

\newcommand{\into}{\hookrightarrow}
\newcommand{\onto}{\twoheadrightarrow}

\newcommand{\Rbar}{\overline{\mathbb{R}}}
\newcommand{\C}{\mathbb{C}}

\newcommand{\Tr}{\mathrm{Tr}}
\newcommand{\stcap}{\cap_{st}}
\newcommand{\Trop}{\operatorname{Trop}}
\newcommand{\trop}{\operatorname{trop}}
\newcommand{\val}{\operatorname{val}}
\newcommand{\vol}{\operatorname{vol}}
\newcommand{\Vol}{\operatorname{Vol}}
\newcommand{\nk}[2]{\binom{[#1]}{#2}}
\newcommand{\supp}{\operatorname{supp}}
\newcommand{\inv}{\operatorname{inv}}

\newcommand{\rec}{\operatorname{rec}}

\newcommand{\cone}{{\operatorname{cone}}}
\newcommand{\sym}{{\operatorname{sym}}}
\renewcommand{\deg}{\operatorname{deg}}

\newcommand{\Z}{\mathbb{Z}}

\newcommand{\N}{\mathbb{N}}
\newcommand{\R}{\mathbb{R}}

\newcommand{\bF}{\mathbb{F}}
\newcommand{\K}{\mathbb{K}}

\newcommand{\PP}{\mathbb{P}}
\newcommand{\PT}[1]{\mathbb{P}\mathbb{T}^{#1}}

\newcommand{\cHH}{\mathcal{H}}
\renewcommand{\cR}{\mathcal{R}}
\newcommand{\cP}{\mathcal{P}}
\newcommand{\cG}{\mathcal{G}}
\newcommand{\cF}{\mathcal{F}}

\newcommand{\cJ}{\mathcal{J}}

\newcommand{\cA}{\mathcal{A}}
\newcommand{\cS}{\mathcal{S}}

\newcommand{\cB}{\mathcal{B}}

\newcommand{\Dr}{\mathbf{Dr}}
\newcommand{\cLL}{\mathcal{L}}

\newcommand{\sF}{\mathscr{F}}

\newcommand{\tD}{\mathtt{D}}

\newcommand{\Gr}{\mathbf{Gr}}
\newcommand{\Qt}{\mathbf{Qt}}
\newcommand{\bC}{\mathbf{C}}
\newcommand{\be}{\mathbf e}
\newcommand{\bv}{\mathbf v}
\newcommand{\bV}{\mathbf V}
\newcommand{\bu}{\mathbf u}
\newcommand{\ba}{\mathbf a}
\newcommand{\bb}{\mathbf b}
\newcommand{\bc}{\mathbf c}
\newcommand{\bx}{\mathbf x}
\newcommand{\by}{\mathbf y}
\newcommand{\bw}{\mathbf w}

\newcommand{\bh}{\mathbf h}
\newcommand{\br}{\mathbf r}

\newcommand{\rmL}{\mathrm{L}}
\newcommand{\rmS}{\mathrm{S}}

\newcommand{\rmV}{\mathrm{V}}

\usepackage[style=numeric, maxbibnames=9, giveninits=true]{biblatex}
\makeatletter
\newcommand{\address}[1]{\gdef\@address{#1}}
\newcommand{\email}[1]{\gdef\@email{\url{#1}}}
\newcommand{\@endstuff}{\par\vspace{\baselineskip}\noindent\small
\begin{tabular}{@{}l}\scshape\@address\\\textit{E-mail address:} \@email\end{tabular}}
\AtEndDocument{\@endstuff}
\makeatother
\author{Jidong Wang}
\address{Department of Mathematics, University of Michigan, Ann Arbor, Michigan, 48109}
\email{jidongw@umich.edu}
\title{Lorentzian polynomials and the incidence geometry of tropical linear spaces}
\date{}

\begin{document}

\maketitle

\begin{abstract}
We introduce a notion of Lorentzian proper position in close analogy to proper position of stable polynomials. Using this notion, we give a new characterization of elementary quotients of M-convex function that parallels the Lorentzian characterization of M-convex functions. We thereby use Lorentzian proper position to study the incidence geometry of tropical linear spaces, and vice versa. In particular, we prove new structural results on the moduli space of codimension-1 tropical linear subspaces of a given tropical linear space.

Applying these results, we show that some properties of classical linear incidence geometry fail for tropical linear spaces. For instance, we show that the poset of all matroids on $[n]$, partially ordered by matroid quotient, is not submodular when $n\geq 8$. On the other hand, we introduce a notion of adjoints for tropical linear spaces, generalizing adjoints of matroids, and show that certain incidence properties expected from classical geometry hold for tropical linear spaces that have adjoints.
\end{abstract}

\section{Introduction}
\subsection{Lorentzian proper position}
A polynomial $f\in\R_{\geq 0}[w_1,...,w_n]$ with nonnegative coefficients is \textit{stable} if for all $\bu \in \R^n_{\geq0}$ and $\bv\in \R^n$, the univariate polynomial $f(t \bu +\bv)$ is real-rooted. The theory of stable polynomials features a key notion called \textit{proper position}: given two stable polynomials $f$ and $g$ in $\R_{\geq0}[w_1,...,w_n]$, $f$ is in \textit{proper position} with respect to $g$, denoted $f\ll g$, if $g+w_{n+1}f\in \R_{\geq0}[w_1,...,w_{n+1}]$ is stable. Homogeneous stable polynomials with nonnegative coefficients are Lorentzian. Here we introduce a Lorentzian analog of proper position.

\begin{definition}\label{def:lorentzian-proper-position}
    Let $f$ and $g$ be nonzero Lorentzian polynomials of degree $d$ and $d+1$, respectively. Say $f$ is in \textit{Lorentzian proper position} with respect to $g$, denoted $f\ll_L g$, if $g+w_{n+1}f$ is Lorentzian.
\end{definition}

Proper position of stable polynomials has two convexity properties that are pivotal in the whole theory. 
\begin{prop}\label{prop:convexity-stable-proper-position} \cite[Lemma 2.6]{borcea2010multivariate}
    Let $f$ be a nonzero stable polynomial. Then the subsets
\begin{center} \vspace{5pt}
    $\big\{g \text{ stable}\mid f\ll g\big\}$ and $\big\{h \text{ stable}\mid h\ll f\big\}$
\vspace{5pt} \end{center} 
are closed convex cones.
\end{prop}
This has a partial Lorentzian analog which has not appeared in the literature, but is known to experts. 
\begin{prop}\label{prop:A}
    Let $f$ be a nonzero Lorentzian polynomial. Then the subset $\big\{g \text{ Lorentzian}\mid f\ll_L g\big\}\cup \{0\}$ is a closed convex cone.
\end{prop}
Our first main result provides a framework for applying
\Cref{prop:A} to the geometry of tropical linear spaces and their tropical linear subspaces, which we explain now. Let $\Rbar=\R\cup\{\infty\}$, $[n]=\{1,...,n\}$, and $\mathbf{1}=(1,...,1)\in \R^n$. Let $\PT{[n]}$ denote the $(n-1)$-dimensional \textit{tropical projective space} and $T_{[n]}$ denote the $(n-1)$-dimensional \textit{tropical projective torus}:
\begin{equation}
    \begin{split}
        \PT{[n]}:=\left(\Rbar^n\backslash\{(\infty,...,\infty)\}\right)/\R\mathbf{1},\quad T_{[n]}:=\R^n/\R\mathbf{1}.
    \end{split}
\end{equation}
A valuated matroid $\mu $ of rank $d$ on $[n]$ is a function $\nk{n}{d}\to \Rbar$ that satisfies the \textit{tropical Pl\"{u}cker relations}. Analogous to the Pl\"{u}cker coordinates of linear spaces in classical geometry, $\mu$ represents a tropical linear space $\Trop \mu\subset\PT{[n]}$ of dimension $d-1$. Let $[\mu]$ be the class of $\mu$ modulo scalar addition. Then the correspondence $[\mu]\leftrightarrow \Trop\mu$ is bijective between equivalence classes of valuated matroids and tropical linear spaces. A valuated matroid $\theta$ of rank $s$ on $[n]$ is a \textit{quotient} of $\mu$, denoted $\mu\onto\theta$, if they satisfy the \textit{tropical incidence-Pl\"{u}cker relations}, and $\Trop\theta\subset\Trop\mu$ if and only if $\mu\onto \theta$. A quotient is \textit{elementary} if $s=d-1$. Elementary quotients of $\mu$ correspond to codimension-1 tropical linear subspaces of $\Trop \mu$. See \Cref{subsec:prelim-tropical-linear-spaces} for a review on these notions.

Associate to each function $\mu:\nk{n}{d}\to \Rbar$ the polynomial 
\begin{equation}
f_q^\mu(w_1,...,w_n) = \sum_{|B|=d} q^{\mu(B)}w^B 
\end{equation}
where $w^B = \prod_{i\in B}w_i$ and $0<q\leq 1$ is a parameter. 

\begin{maintheorem}\label{main:A}
    Let $\mu$ and $\theta$ be valuated matroids of rank $d$ and $d-1$ on $[n]$, respectively. Then 

\begin{center} \vspace{5pt}
    $\mu\onto\theta$ \quad  if and only if \quad  $f_q^\theta \ll_L f_q^\mu$ for all $0<q\leq 1$.
\vspace{5pt} \end{center} 
\end{maintheorem}

By tropicalization and valuated matroid duality, \Cref{main:A} and \Cref{prop:A} together recover the following result by by Fink and Moci \cite[Corollary 6.6]{Fink_2019}.

\begin{cor}\label{cor:A}
    Let $\mu$ be a valuated matroid, then
    \begin{equation}
     \Dr^1(\mu) := \big\{\text{codimension-1 tropical linear subspaces of }\Trop\mu\big\}
\end{equation}
is tropically convex. Namely, if $\Trop\theta_1$ and $\Trop\theta_2$ are codimension-1 tropical linear subspaces of $\Trop\mu$ and $\theta$ is any tropical linear combination of $\theta_1$ and $\theta_2$, then $\theta$ is a valuated matroid, and $\Trop\theta\subset \Trop\mu$.
\end{cor}

In particular, $\Dr^1(\mu)$ is connected. For an interesting consequence of \Cref{cor:A} for matroids, see \Cref{cor:union-basis-families}. In \Cref{thm:dressian-cut-out-by-three-term}, we improve \Cref{cor:A} and describe $\Dr^1(\mu)$ as a tropical prevariety cut out by a nice set of tropical hyperplanes.

Our next main result uses \Cref{main:A} in the other direction. By studying the geometry of $\Dr^1(\mu)$, we obtain a partial analog for Lorentzian proper position of the convexity of $\{h\text{ stable }\mid h\ll f\}$. We also provide numerical examples in \Cref{ex:L1-not-convex} showing the full Lorentzian analog is false.

\begin{maintheorem}\label{main:B}
    Let $\mu$ be a valuated matroid. Then for each $0<q\leq 1$, $\{h \text{ Lorentzian}\mid h\ll_L f^\mu_q\}\cup \{0\}$ contains the convex cone spanned by $\{f^\theta_q\mid [\theta]\in \Dr^1(\mu)\}$.
\end{maintheorem}


\subsection{Tropical linear incidence geometry}
The proof of \Cref{main:B} relies on understanding the incidence geometry of tropical linear spaces. We study the tropical analogs of the following properties of the incidence geometry in classical projective spaces. Let $V\subset \mathbb{P}^{n-1}$ be a linear subspace of dimension $d-1$:
\begin{enumerate}[label=\textbf{(P\arabic*)},leftmargin=0.5in]
    \item\label{itm:IG-intersection} If $W_1$ and $W_2$ are codimension-1 linear subspaces of $V$, then $W_1\cap W_2$ contains a codimension-2 linear subspace of $V$.
    \item\label{itm:IG-interpolation} Any set of $d-1$ points in $V$ is contained in a subspace $W\subset V$ of codimension at most 1.
    \item\label{itm:IG-flag} If $V_k$ is a subspace of codimension $k$ in $V_0=V$, then the partial flag $V_k\subset    V_0$ can be completed, i.e., there is a chain of subspaces $V_k\subset V_{k-1}\subset\cdots \subset V_0$ such that $\dim V_i - \dim V_{i+1}=1$ for all $i$.
\end{enumerate}

\begin{maintheorem}\label{main:C}
    Let $\Trop\mu$ be a tropical plane. Then any two tropical lines in $\Trop\mu$ intersect.
\end{maintheorem}

This provides an affirmative case to the tropical analog of \ref{itm:IG-intersection} when $d=3$. This theorem, together with the technique of common interlacers developed in \cite{chudnovsky2007roots}, allows us to deduce \Cref{main:B} from \Cref{prop:convexity-stable-proper-position}. However, we construct tropical linear spaces of each dimension $d\geq 4$ for which the analog of \ref{itm:IG-intersection} fails. See \Cref{ex:counter-to-submodular} and \Cref{prop:intersection-fail-at-every-d-larger-than-4}. These examples also show that the poset of all matroids on $[n]$, partially ordered by matroid quotient, is not upper semimodular when $n\geq 8$.

We investigate the tropical analog to \ref{itm:IG-interpolation} as follows. Let $M$ be a matroid of rank $d$ on $[n]$, $\cLL(M)$ be its lattice of flats, and $\cLL^1(M)$ be its hyperplanes, i.e., corank-1 flats. An \textit{adjoint} of $M$ is a matroid $W$ on $\cLL^1(M)$ of rank $d$ such that $\cLL(M)$ embeds in the order dual of $\cLL(W)$. We give a new characterization of adjoints in terms of the geometry of $\Dr^1(\mu)$. Let $\Trop M$ be the Bergman fan of $M$ and put
\begin{equation}
     \Dr^1(M) := \big\{\text{codimension-1 tropical linear subspaces of }\Trop M\big\}. 
\end{equation}
Then an adjoint of $M$ corresponds to a tropical linear space of dimension $d-1$ inside $\Dr^1(M)$. See \Cref{prop:characterization-of-adjoint} for the precise statement. Using this idea, we generalize adjoints to valuated matroids in Definition \ref{def:valuated-adjoint} and use this notion to formulate \Cref{main:D}. 

\begin{maintheorem}\label{main:D}
    Let $\mu$ be a valuated matroid of rank $d$ on $[n]$. If $\mu$ has an adjoint, then any set of $d-1$ points in $\Trop\mu$ is contained in a codimension-1 tropical linear subspace $\Trop\theta\subset \Trop\mu$.
\end{maintheorem}

Since all rank-3 matroids have adjoints, \Cref{main:D} implies 

\begin{cor}
    Let $M$ be a rank-3 matroid. Then through any two points on $\Trop M$ there is a tropical line that is contained in $\Trop M$.
\end{cor}

However, not all valuated matroids have adjoints, and we show that the tropical analog of \ref{itm:IG-interpolation} fails for $\Trop M$ if $M$ does not have the \textit{Levi intersection property}. This property is described by \textit{linear subclasses}, which are subsets $\cHH$ of $\cLL^1(M)$ satisfying the following property:
\begin{center} \vspace{5pt}
    If $H_i,H_j\in \cHH$ and they cover $H_i\cap H_j$, then all hyperplanes covering $H_i\cap H_j$ are in $\cHH$.
\vspace{5pt} \end{center} 
The whole set $\cLL^1(M)$ is the \textit{trivial linear subclass}. A matroid $M$ has the \textit{Levi intersection property} if any $d-1$ hyperplanes are contained in a nontrivial linear subclass. For instance, the V\'{a}mos matroid $V_8$ does not have the Levi intersection property \cite{cheung1974adjoints}. 
\begin{maintheorem}\label{main:E}
    Let $M$ be a matroid of rank $d$ on $[n]$. If $M$ does not have the Levi intersection property, then there is a set of $d-1$ points in $\Trop M\cap T_{[n]}$ not contained in any codimension-1 tropical linear subspace of $\Trop M$.
\end{maintheorem}
The tropical analog of \ref{itm:IG-flag} is known to hold for Bergman fans of matroids and for flags $\emptyset \subset \mu$ and $\mu\subset \PT{[n]}$. We add one more affirmative case.
 \begin{maintheorem}\label{main:F}
 Let $\mu$ be a valuated matroid. For any point $[\bw]\in \Trop\mu$, the partial flag $\{[\bw]\}\subset \Trop\mu$ can be completed.
 \end{maintheorem}

  No counter-example to the tropical analog of \ref{itm:IG-flag} is known. However, we show that there are flags of tropical linear spaces that cannot be completed with a \textit{prescribed} flag of underlying matroids. Recall that a flag of tropical linear spaces $\Trop \mu_k\subset \Trop \mu_{k-1}\subset\cdots\subset \Trop\mu_0$ induces a flag of matroids $M_0\onto M_2\onto \cdots \onto M_k$, where $M_i$ is the underling matroid of $\mu_i$ for each $i$. 
   \begin{maintheorem}\label{main:G}
   Let $p$ be any prime power and $M$ be the matroid of the finite projective plane over $\mathbb{F}_q$. Let $\Trop\theta$ be a generic tropical line whose underlying matroid is the uniform matroid $U_{2,q^3-1}$. Then $U_{2,q^3-1}$ is a quotient of $M$, but there is no tropical plane containing $\Trop\theta$ with underlying matroid $M$.
  \end{maintheorem}

\subsection{History}\label{subsec:history}

After this paper was finished, we learned from Petter Br\"{a}nd\'{e}n that \Cref{prop:A} is known to experts. He communicated to us a simple proof which we include in \Cref{rmk:branden-proof}. The introduction of this paper differs slightly from its first arXiv version, as we were informed of more relevant literature which we didn't give enough credit to: The tropical convexity of $\Dr^1(\mu)$ was first stated in Frenk's thesis \cite[Theorem 4.2.9]{frenk2013tropical}, where the proof was incomplete. It was restated in Fink and Moci \cite[Corollary 6.6]{Fink_2019}. Later, a special case where $\mu$ was uniform was proved in \cite{joswig2023generalized}. We wrote the arXiv version knowing about only \cite{joswig2023generalized}.

\Cref{prop:A} and \Cref{main:B} concern the local structure of spaces of Lorentzian polynomials. They can be compared to \cite{branden2021spaces}, which shows that spaces of Lorentzian polynomials of given degrees are homeomorphic to Euclidean balls, and to \cite{bakerhuh}, which shows that the logarithmic image of the set of Lorentzian polynomials supported on any matroid is a manifold with boundary. The failure of the convexity of $\{h \text{ Lorentzian}\mid h\ll_L f\}$ deserves an explanation. Recall that stability of polynomials is preserved under the inversion operation analogous to matroid duality:
\begin{equation}
    f(w_1,...,w_n)\mapsto w_1^{d_1}\cdots w_n^{d_n}f(w_1^{-1},...,w_n^{-1}),
\end{equation}
where for each $i$, the degree of $f$ in $w_i$ is $d_i$. Therefore, the convexity property in \Cref{prop:convexity-stable-proper-position} is symmetric. The asymmetry in \Cref{prop:A} implies that inversion does not preserve the Lorentzian property. We give explicit examples in \Cref{ex:L1-not-convex}.

The Dressian $\Dr(d,n)$ that parametrizes all tropical linear spaces of dimension $d-1$ in $\PT{n}$ is notoriously hard to study. One of our major discoveries is that $\Dr^1(\mu)$, in contrast, has much simpler structure. We call these moduli spaces that parametrize tropical linear subspaces in a given tropical linear space the \textit{relative Dressians}, a name suggested by June Huh. \Cref{cor:A} is the tropical analog of the following observation in classical geometry. Let $V$ be any $d$-dimensional vector subspace in $E\cong \bF^n$ for some field $\bF$. The inclusion $V\hookrightarrow E$ induces inclusion $\Gr(d-1,V)\hookrightarrow \Gr(d-1,E)$ and inclusion $\bigwedge^{d-1}V\hookrightarrow \bigwedge^{d-1}E$. The following diagram commutes
 \begin{center} \vspace{5pt}
      \begin{tikzcd}
     \Gr(d-1,V) \ar[r,hook] \ar[d,hook,"\text{Pl\"{u}cker}"] & \Gr(d-1,E) \ar[d,hook,"\text{Pl\"{u}cker}"] \\
     \mathbb{P}(\bigwedge^{d-1}V) \ar[r,hook] & \mathbb{P}(\bigwedge^{d-1}E)
 \end{tikzcd}
 \vspace{5pt} \end{center} 
because the Pl\"{u}cker map on $E$ restricts to the Pl\"{u}cker map on $V$. Since $\Gr(d-1,V)\cong \mathbb{P}(\bigwedge^{d-1}V)\cong \mathbb{P}^{d-1}$, we see that $\Gr(d-1,V)$ is a copy of $\mathbb{P}^{d-1}$ inside the Pl\"{u}cker embedding of $\Gr(d-1,E)$. 
The linearity of $\Gr(d-1,V)$ inside $\mathbb{P}(\wedge^{d-1}E)$ is proved using coordinates in \cite{jell2022moduli}. There, the authors study the space of \textit{tropicalized} codimension-1 linear subspaces of a \textit{tropicalized} linear space.
\Cref{thm:dressian-cut-out-by-three-term}, as a strengthening of \Cref{cor:A}, extends \cite[Theorem 3.10]{joswig2023generalized}, which treats the case when $\mu$ is uniform.

The notion of adjoints of valuated matroids is new, but the underlying geometry in the classical setting is simple: there is an embedding
\begin{equation}
    \Gr(d,E) \hookrightarrow \Gr(d,\wedge^{d-1}E),\quad V\mapsto \wedge^{d-1}V.
\end{equation}
Each point in the image of this embedding is a linear subspace of $\PP(\wedge^{d-1}E)$ parametrizing hyperplanes of some $(d-1)$-dimensional linear subspace $\PP V$ of $\PP E$. It turns out there are determinantal identities expressing the Pl\"{u}cker coordinate of $\PP(\wedge^{d-1}V)$ in terms of the Pl\"{u}cker coordinate of $\PP V$. We did not find this in the literature, so we include this material in \Cref{subsec:cofactor-grassman}. Many of our results on adjoints of valuated matroids are tropical analogs of this phenomenon.

Dress and Wenzel's original definition of valuated matroids was motivated by \textit{oriented matroids}. Therefore, it is instructive to compare our results with what is known for oriented matroids.
\begin{itemize}
    \item For an oriented matroid $M^{o}$, the analog of the relative Dressian $\Dr^1(\mu)$ is the \textit{oriented elementary quotient lattice} $\widehat{\Qt}_{OM}^1(M^o)$, which parametrizes elementary \textit{oriented} matroid quotients of $M^o$. In contrast to the connectivity of $\Dr^1(\mu)$, the order complex of $\widehat{\Qt}_{OM}^1(M^o)$ may be disconnected \cite{mnev1993two,liu2020counterexample}. In \Cref{thm:dressian-is-order-complex}, we show that the order complex of the (non-oriented) elementary quotient lattice $\widehat{\Qt}^1(M)$ equals $\Dr^1(M)$. This suggests that the disconnectivity of $\Qt_{OM}^1(M^o)$ comes from orientations.
    \item The existence of adjoints and the Levi intersection properties are the \textit{intersection properties} extensively studied for oriented matroids, as an effort to generalize polar duality of polytopes to oriented matroids. See \cite[Section 7.5]{bjorner1999oriented} for a full account. Here for valuated matroids, these properties naturally describe the incidence properties of tropical linear spaces.
\end{itemize} 

Our study on the incidence geometry of tropical linear spaces was motivated by two recent advancements.
The first is the theory of matroids over hyperstructures \cite{baker2019matroids,baker2020foundations,baker2023foundations}. This theory predicts that the set of all quotients of a given matroid, partially ordered by the quotient relation, might resemble a geometric lattice. See \Cref{sec:incidence-problems} for an explanation on how our choice of the incidence problems is related to this prediction. Our results provide both evidence and counter-evidence. Among the counter-evidence, the failure of the upper semimodularity of the poset of matroids seems most essential, as many nice properties of $\cLL(M)$ stem from its semimodularity.

The second relevant advancement is the study of linear series on tropical curves, which has led to new results and new proofs of classical theorems about the geometry of algebraic curves \cite{jensen2014tropical,jensen2016tropical,farkas2020kodaira,farkas2024nonabelian}. In developing a theory of tropical linear series that is fully grounded in tropical geometry, it seems necessary to understand tropical linear incidence geometry. Murota's constructions \cite{murota1997matroid} of truncation and elongation address the tropical analog of \ref{itm:IG-flag} for flags $\emptyset\subset\Trop\mu$ and $\Trop\mu\subset \PT{[n]}$. Other than that, little was known about those incidence problems before our attempt. This paper is only the start of a program in this direction. There are other incidence problems we do not address within the scope of this paper. For instance, do a codimension-1 tropical linear subspace and a tropical line in $\Trop\mu$ always intersect?

\subsection{Organization}\label{subsection:organization}
Here is a roadmap to each section.

\Cref{sec:prelim}: background materials on matroid quotients, tropical linear spaces, M-convex functions, and Lorentzian polynomials.

\Cref{sec:Lorentzian-proper-position}: properties of Lorentzian proper position, such as preservation under linear operations and the relation with quotients of M-convex functions; proofs of \Cref{prop:A} and \Cref{main:A}; sources of Lorentzian proper position, such as pullback of intersection product of nef divisors and intrinsic volumes of convex bodies.

\Cref{sec:dressian-structure}: the structural result \Cref{thm:dressian-cut-out-by-three-term} on $\Dr^1(\mu)$; descriptions of $\Dr^1(M)$ as the order complex of the elementary quotient lattice $\widehat{\Qt}^1(M)$ and as a tropical polytope; a geometric characterization of adjoints of ordinary matroids. 

\Cref{sec:valuated-adjoint}: generalization of adjoints to valuated matroids; the important \Cref{lem:Sigma-basis-valuation} which relates the value of $\mu$ and the value of its adjoint; proof of \Cref{thm:projective-plane}, which implies \Cref{main:G};

\Cref{sec:incidence-problems}: the tropical analogs of \ref{itm:IG-intersection}, \ref{itm:IG-flag} and \ref{itm:IG-flag}, and the corresponding  incidence problems of matroids; proofs of \Cref{main:C,main:D,main:E,main:F} and counterexamples.

\Cref{sec:from-incidence-to-lorentzian}: application of the incidence results to Lorentzian polynomials; proof of \Cref{main:B}.

Appendix \ref{subsec:cofactor-grassman}: the classical geometry behind adjoints of valuated matroids.

\subsection{Notations}\label{subsection:notations}

Unless otherwise stated, $M$ denotes a \textit{simple} matroid of rank $d$ on $[n]$, and $\mu$ denotes a valuated matroid of rank $d$ on $[n]$ with underlying matroid $M$. The degree of the zero polynomial is taken to be 0, so any polynomial of positive degree is nonzero. The following notations will be used throughout.

\vspace{1em}
   \bgroup
   \def\arraystretch{1.2}
   \begin{tabular}{cl}
    \multicolumn{2}{l}{Sets and tuples} \\\hline 
    $\ba,\bb,...$ &  tuples in $\R^n,\N^n,$ or $\Rbar^n$ \\
        $\ba(i)$ & the $i$-th component of $\ba$ \\ 
        $[\ba]$ & the class of $\ba$ modulo scalar addition \\
       $|\ba|_1$ &  $\sum_{i=1}^n|\ba(i)|$\\
        $\be_i$ & the $i$-th standard basis vector of $\R^n$\\
       $Aijk$  & $A\cup\{i,j,k\}$ \\
       $A+i$ & $A\cup \{i\}$ \\
       $A-i$ & $A\backslash\{i\}$ \\\hline
       \multicolumn{2}{l}{Polynomials} \\\hline
       $f,g,h,...$ & polynomials \\
        $w^\ba$   & $w_1^{\ba(1)}\cdots w_n^{\ba(n)}$ \\ 
        $w^B$ & $\prod_{i\in B}w_i$ \\
        $\partial_i$ & the partial derivative with respect to the $i$-th variable \\
        $\partial_{\bv}$ & the directional derivative in the $\bv$ direction \\
        $\partial^\ba$ & $\partial_1^{\ba(1)}\cdots \partial_n^{\ba(n)}$\\
        $\partial^B$ & $\prod_{i\in B}\partial_i$
       \\\hline        \multicolumn{2}{l}{Valuated matroids}  \\\hline
       $\mu,\theta,\nu,\Sigma$  & valuated matroids \\
       $[\mu]$ & the class of $\mu$ modulo scalar addition \\
       $\mu\backslash I$ & the deletion of $\mu$ by $I$ \\
       $\mu/I$ & the contraction of $\mu$ by $I$ 
    \end{tabular}
    \vspace{1em}
    \egroup
    
We also adopt the following standard notations in matroid theory. 

\begin{center} \vspace{5pt}
\begin{tabular}{c|c}
   $\cB(M)$  &  the set of bases of $M$ \\
    $\cLL(M)$ & the lattice of flats of $M$ \\
     $\cLL^k(M)$  &  the corank-$k$ flats of $M$ \\
    $\cLL_k(M)$ & the rank-$k$ flats of $M$ \\
     $\cl_M$ & the closure operator of $M$ \\
     $\Tr M$ & the truncation of $M$
   
\end{tabular}\hspace{0.2in}
\begin{tabular}{c|c}
 
    $\rk_M$ & the rank function of $M$ \\
    $\rk M$ & the rank  of $M$   \\
     $M/A$ & contraction of $M$ by $A$ \\
     $M\backslash A$ &  deletion of $M$ by $A$\\
     $M\oplus N$ & the direct sum of $M$ and $N$ \\
     $U_{r,n}$ & the uniform matroid of rank $r$ on $[n]$
    
\end{tabular}
\vspace{5pt} \end{center} 

\noindent\textbf{Acknowledgments.}
We thank Chris Eur, Alex Fink, June Huh, Dhruv Ranganathan, Botong Wang, and Wang Yao for helpful conversations, and Sam Payne for suggestions on the organization of this paper. We thank Petter Br\"{a}nd\'{e}n for helpful feedback on a draft of this paper. The author received support from NSF grants DMS-2113468, DMS-2302475, and DMS-2053261.

\section{Preliminaries}\label{sec:prelim}

We assume that the reader is familiar with basic notions about matroids and posets. \Cref{subsec:prelim-matroid-quotient} provides background on matroid quotients and \Cref{subsec:prelim-tropical-linear-spaces} on tropical linear spaces. They are not needed until \Cref{sec:dressian-structure}. Readers with the assumed background may skip these sections and come back later.

\subsection{Matroid quotient and the elementary quotient lattice}\label{subsec:prelim-matroid-quotient}

We review the characterization of elementary quotients by linear subclasses and by modular cuts. This material can be found in \cite{cheung1974compatibility} and \cite{white1986theory}.

Given two matroids $M$ and $N$ on $[n]$, $N$ is a \textit{quotient} of $M$, denoted $M\onto N$, if $\cLL(N)\subset \cLL(M)$; the \textit{corank} of the quotient is the number $\rk M-\rk N$; an \textit{elementary quotient} is a quotient of corank 1. A \textit{linear subclass} of $M$ is a subset $\cHH\subset \cLL^1(M)$ that satisfies the following property: 
\begin{center} 
\vspace{5pt}
    If $H_i,H_j\in \cHH$ and they cover $H_i\cap H_j\in \cLL^2(M)$, then all hyperplanes covering $H_i\cap H_j$ are in $\cHH$.
\vspace{5pt} \end{center} 
The whole set $\cLL^1(M)$ is the \textit{trivial linear subclass}. If $Q$ is an elementary quotient of $M$, then $\cLL^1(Q)\cap \cLL^1(M)$ is a nontrivial linear subclass of $M$. Conversely, given a nontrivial linear subclass $\cHH$, there is a unique elementary quotient $Q$ such that $\cLL^1(Q)\cap \cLL^1(M)=\cHH$. Put 
\begin{equation}
    \widehat{\Qt}^1(M)=\Qt^1(M) \cup \{M\},
\end{equation}
where
\begin{equation}
    \Qt^1(M):=\big\{\text{ $Q$ a matroid} \mid  Q\text{ is an elementary quotient of }M\big \}.
\end{equation}
Then the map
\begin{equation}\label{eqn:elementary-quotient-to-linear-subclass}
    \widehat{\Qt}^1(M) \to \big\{\text{ linear subclasses of }M\big\},\quad  Q\mapsto \cLL^1(Q)\cap\cLL^1(M)
\end{equation}
is bijective. See \Cref{ex:linear-subclass-and-elementary-quotient} for some examples of linear subclasses.

\begin{ex}\label{ex:linear-subclass-and-elementary-quotient}
Let $M=U_{3,4}$. The empty set is a linear subclass; the subset \{12,13\} is \textit{not} a linear subclass, because $1$ is covered by 12, 13 and 14, but 14 is not in $\{12,13\}$. The order filter $\{12,34,1234\}\subset \cLL(U_{3,4})$ is a modular cut; the order filter $\{12,13,1234\}$ is \textit{not} a modular cut, because the flats 12 and 13 satisfy \eqref{eqn:modular-pair}, but their intersection is not in $\{12,13,1234\}$. \Cref{tab:linear-subclass-and-quotient} shows three elementary quotients of $U_{3,4}$, the corresponding linear subclasses, and the corresponding modular cuts. 
\end{ex}

\begin{table}[h]
    \centering
    \small 
    \begin{tabular}{cccc} 
        \toprule
        \small \begin{tabular}{c}
            \textbf{The affine}  \\
                \textbf{diagram of $Q$}
        \end{tabular} & 
        \small \begin{tabular}{c}
             \textbf{$\cLL(Q)$, boxed} \\ 
             \textbf{inside $\cLL(U_{3,4})$}
        \end{tabular}   & 
         \small\begin{tabular}{c}
            \textbf{The corresponding}  \\
                \textbf{linear subclass $\cHH$}
        \end{tabular}
        & 
        \small\begin{tabular}{c}
            \textbf{The corresponding}  \\
                \textbf{modular cut $\cF$}\end{tabular}  \\ 
        \midrule
        
        \begin{tabular}{c}
             \includegraphics[width=0.8in]{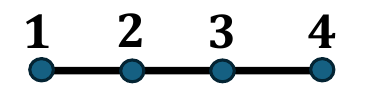} \\
        \end{tabular} & 
         \begin{tabular}{c}
             \\
             \includegraphics[width=1.5in]{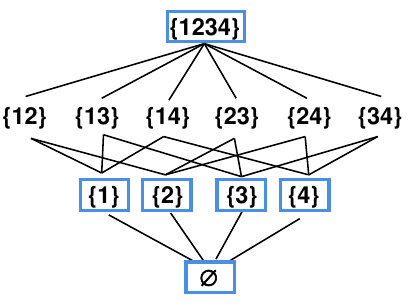} 
             \\
        \end{tabular} & 
         \begin{tabular}{c}
             $\emptyset$ \\
        \end{tabular} & 
        $\{1234\}$ \\ 
        \addlinespace 
        
        \begin{tabular}{c}
             \includegraphics[width=0.6in]{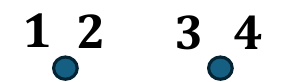}\\
             \\
        \end{tabular} & 
        \begin{tabular}{c}
             \\
             \includegraphics[width=1.5in]{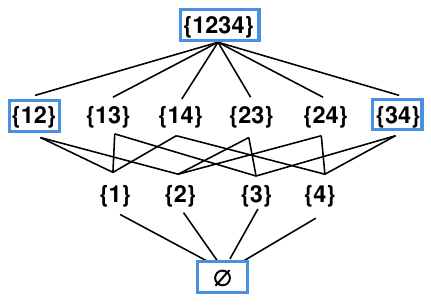} 
             \\
        \end{tabular} & 
            \begin{tabular}{c}
             $\{12,34\}$ \\
        \end{tabular}  & 
        $\{12,34,1234\}$ \\ 
        \addlinespace
        
        \begin{tabular}{c}
             \includegraphics[width=1in]{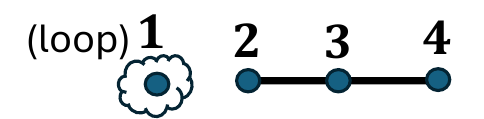}\\
             \\
        \end{tabular}
        & 
               \begin{tabular}{c}
             \\
             \includegraphics[width=1.5in]{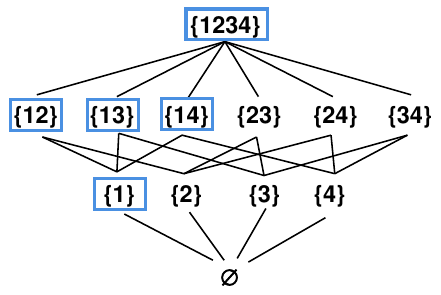} 
             \\
        \end{tabular}& 
        \begin{tabular}{c}
             $\{12,13,14\}$ \\
        \end{tabular}   & 
        upper interval of 1 \\ 
        \bottomrule
    \end{tabular}
    \caption{Three elementary quotients of $U_{3,4}$}
    \label{tab:linear-subclass-and-quotient}
\end{table}

By definition, an intersection of linear subclasses is again a linear subclass. Under \textit{reverse} inclusion, the set of linear subclasses of $M$ is a lattice. The bijection \eqref{eqn:elementary-quotient-to-linear-subclass} then turns $\widehat{\Qt}^1(M)$ into a lattice. The minimal element of $\widehat{\Qt}^1(M)$ is $M$, corresponding to the trivial linear subclass $\cLL^1(M)$; the maximal element is the \textit{truncation} $\Tr M$ of $M$, which corresponds to the linear subclass $\emptyset$. This lattice is known as the \textit{elementary quotient lattice} \cite[Page 150]{white1986theory}.

\Cref{fig:elementary-quotient-lattice-u34} shows the lattice of linear subclasses and the elementary quotient lattice of $U_{3,4}$. It is isomorphic to the order dual of the lattice of flats of the graphical matroid $M_{K_4}$. We warn the reader that this belongs to a special case. See \Cref{prop:valuated-adjoint-tropicalization} (2). In general, elementary quotient lattice is very far from being the order dual of the lattice of flats of any matroid. See \Cref{ex:elementary-quotient-lattice}.
  \begin{figure}
      \centering
\includegraphics[width=3in]{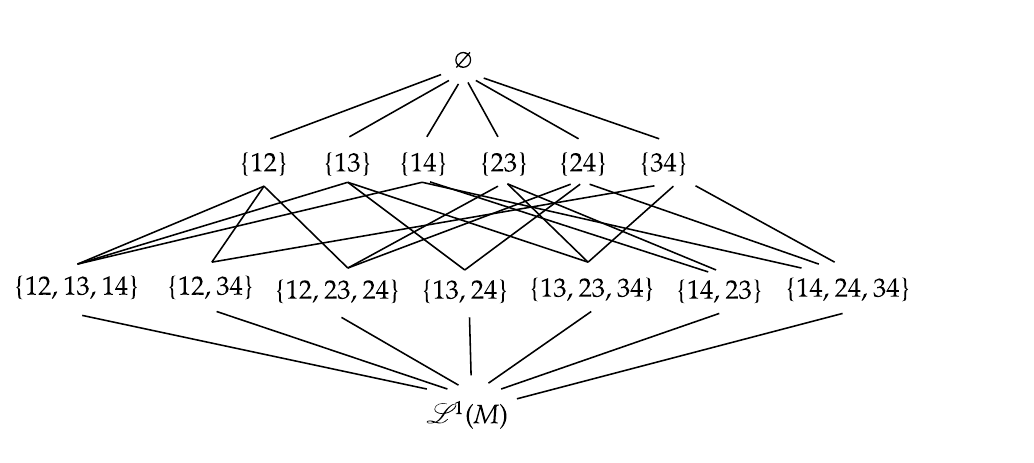} 
    \includegraphics[width=3in]{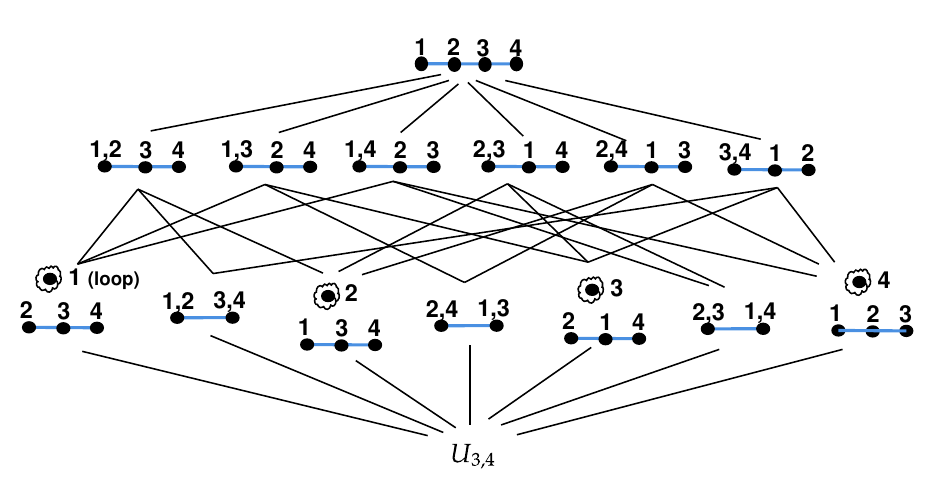}
      \caption{The lattice of linear subclasses of $U_{3,4}$ and the elementary quotient lattice $\widehat{\Qt}^1(U_{3,4})$.}
      \label{fig:elementary-quotient-lattice-u34}
  \end{figure}
  
  A nonempty order filter $\cF\subset\cLL(M)$ is called a \textit{modular cut} if whenever $F_1,F_2\in \cF$ satisfy 
\begin{equation}\label{eqn:modular-pair}
    \rk_M(F_1) + \rk_M(F_2) = \rk_M(F_1\wedge F_2) + \rk_M(F_1\vee F_2),
\end{equation}
we have $F_1\wedge F_2\in \cF$. The \textit{collar} $\cR$ of a modular cut $\cF$ is the set of flats not in $\cF$ but are covered by some element in $\cF$. The full set $\cLL(M)$ is the \textit{trivial modular cut}. If $\cF$ is a nontrivial modular cut, then the subset $\cLL(M)\backslash \cR$ is the lattice of flats of an elementary quotient $Q$ of $M$. Conversely, if $Q$ is an elementary quotient of $M$, then $\cLL(M)\backslash\cLL(Q)$ is the collar of some modular cut. If $\cF$ is a modular cut, then $\cF\cap \cLL^1(M)$ is a linear subclass; conversely, the set of all joins of elements in a linear subclass is a modular cut. See \Cref{ex:linear-subclass-and-elementary-quotient} for examples of modular cuts.

\subsection{Tropical linear spaces}\label{subsec:prelim-tropical-linear-spaces}

 We review notions on valuated matroids and tropical linear spaces. Our references are \cite{brandt2021tropical} and \cite{murota2001circuit}. 

Recall the \textit{tropical projective spaces} $\PT{[n]}$ and \textit{tropical projective tori} $T_{[n]}$ defined in the introduction:
\begin{equation}
    \PT{[n]}:=\left(\Rbar^n\backslash\{(\infty,...,\infty)\}\right)/\R\mathbf{1},\quad T_{[n]}:=\R^n/\R\mathbf{1}.
\end{equation}
For each tuple $\bw\in \Rbar^n$, the \textit{support} of $\bw$, denoted $\supp(\bw)$, is the subset $\{i\in [n]\mid \bw(i)<\infty\}$. Given a finite collection of values $a_1,...,a_m\in \Rbar$, we say $\min_i\{a_i\}$ \textit{vanishes tropically} if the minimum is achieved at least twice. Each $\bw\in \Rbar^n$ with nonempty support defines a tropical equation on $\Rbar^n$, the tropical vanishing of
\begin{equation}
    \min_{i}\{\bx(i)+\bw(i)\}.
\end{equation}
Its vanishing set always has the lineality space $\R\mathbf{1}$. The image of this vanishing set in $\PT{[n]}$ is a \textit{tropical hyperplane}.

A \textit{valuated matroid} of rank $d$ on $[n]$ is a function $\mu:\nk{n}{d}\to \Rbar$ that satisfies the
 \begin{itemize}
     \item \textbf{Tropical Pl\"{u}cker relations:} For any $I\in \nk{n}{d-2}$ and $J\in \nk{n}{d}$, the following vanishes tropically
    \begin{equation}
        \min_{i\in J\backslash I}\{\theta(I+i)+\theta(J-i)\}.
    \end{equation}
 \end{itemize}
 We will call each of the above relations the \textit{Pl\"{u}cker relation indexed by} $(I,J)$. 
 
  For any basis $B$ of $\mu$ and any $i\notin B$, there is a unique circuit contained in $B+i$, called the \textit{basic circuit} $(B,i)$. A basic circuit $(B,i)$ produces a vector $\bC_{(B,i)}\in \Rbar^n$ by 
\begin{equation}
    \bC_{(B,i)}(j)=\mu(B-j+i).
\end{equation}
Up to scalar addition, these are all the \textit{valuated circuits} of $\mu$. Each valuated circuit defines a tropical hyperplane. The intersection of all such hyperplanes is a tropical linear space denoted by $\Trop\mu$. Namely, $\Trop\mu$ is the intersection of
\begin{equation}\label{eqn:valuated-circuit-equation}
    H_\bC =\left\{[\bx]\in \PT{[n]}\mid \min_{i}\{\bx(i)+\bC(i)\}\text{ vanishes tropically}\right\},
\end{equation}
where $\bC$ ranges all valuated circuits of $\mu$. See \Cref{ex:hyperplanes} and \Cref{ex:points} for two examples of tropical linear spaces: hyperplanes and points.

Tropical linear spaces are parametrized by the \textit{Dressians}:
\begin{equation}
    \Dr(d,n) := \left\{ \text{ tropical linear spaces in $\PT{[n]}$ of dimension $d-1$ } \right\},
\end{equation}
which we identify with the set of equivalence classes of valuated matroids of rank $d$ on $[n]$ modulo scalar addition. This makes $\Dr(d,n)$ into a tropical prevariety in $\PT{\nk{n}{d}}$ cut out by all the Pl\"{u}cker relations. Put 
\begin{equation}
        \Dr^k(\mu) := \big\{\text{codimension-$k$ tropical linear subspaces of }\Trop\mu\big\},
\end{equation}
which is a subset of $\Dr(d-1,n)$. We call these spaces the \textit{relative Dressians}. We realize $\Dr^1(\mu)$ as a tropical prevariety in $\PT{\nk{n}{d-1}}$. A tropical linear space $\Trop\theta$ of dimension $d-2$ is contained in $\Trop\mu$ if and only if $\theta$ is an \textit{elementary quotient} of $\mu$, i.e., $\theta$ satisfies the:

\begin{itemize}
    \item \textbf{Tropical incidence-Pl\"{u}cker relations:} For any $I\in \nk{n}{d-2}$ and $J\in \nk{n}{d+1}$, the following vanishes tropically
    \begin{equation}\label{eqn:incidence-plucker}
        \min_{i\in J\backslash I}\{\theta(I+i)+\mu(J-i)\}.
    \end{equation}
\end{itemize}
We will call each of such relations the \textit{incidence relation} \textit{indexed by} $(I,J)$.  

A valuated matroid $\mu$ has an \textit{underlying matroid} $M$ whose bases are the support of $\mu$. We also say $\mu$ is a \textit{basis valuation} on $M$. Conversely, every matroid has a \textit{trivial valuation} that sends the bases of $M$ to 0, and other $d$-sets to $\infty$. The \textit{Bergman fan} $\Trop M$ of a matroid $M$ is the tropical linear space given by the trivial valuation on $M$. Put\footnote{Not to be confused with the \textit{local Dressian} $\Dr_M$, which in some sources denotes the set of all basis valuations on $M$ \cite{olarte2019local}.}
\begin{equation}
    \Dr^1(M):=\big\{\;\text{codimension-1 tropical linear subspaces of }\Trop M\;\big\}.
\end{equation}
Note the slight abuse of language here: in most literature, the Bergman fan is $\Trop M\cap T_{[n]}$. In this paper, $\Trop M$ is the closure of $\Trop M\cap T_{[n]}$ in $\PT{[n]}$, which is not a fan. This abuse of language will be harmless, but we need to adapt the definition of the recession fan to our setting.

\begin{definition}\label{def:recession-space}
    Let $X$ be a tropical prevariety in $\PT{[n]}$. The \textit{recession space} of $X$ is
    \begin{equation}
        \rec(X):=\Big\{[\bw]\in \PT{[n]}\mid \text{there is some } [\bx]\in X\cap T_{[n]}
        \text{ such that }[\bx+t\bw]\in X\text{ for all }t>0\Big\}.
    \end{equation}
\end{definition}
With this definition, the intersection of $\rec(X)$ with $T_{[n]}$ is the usual recession fan of $X\cap T_{[n]}$. Taking the recession space has the following property: given tropical prevarieties $X_1,...,X_m$,
\begin{equation}\label{eqn:recession-space-intersection}
    \rec(\bigcap_{i=1}^m X_i) \subset \bigcap_{i=1}^m \rec(X_i).
\end{equation}
We also have $\rec(\Trop\mu)=\Trop M$. 

\begin{ex}\label{ex:hyperplanes}
If the corank of $\mu$ is 1, then $\mu$ takes the form
\begin{equation}
    \mu([n]-i) = -\bw(i)
\end{equation}
for some $\bw\in \Rbar^n\backslash\{(\infty,...,\infty)\}$. In this case, $\Trop\mu$ is a tropical hyperplane. If $\bw\in \R^n$, then we say $\Trop\mu$ is the hyperplane \textit{centered} at $[\bw]$, which is defined by the tropical vanishing of 
    \begin{equation}
    \min_{i}\{\bx(i)-\bw(i)\}
\end{equation}
\end{ex}

\begin{ex}\label{ex:points}
    If the rank of $\mu$ is 1, then $\mu$ is simply a point in $\Rbar\backslash \{(\infty,...,\infty)\}$.
\end{ex}

\subsection{M-convex functions and sets}\label{subsec:prelim-m-convexity}

We review M-convex sets and M-convex functions. Our main reference is \cite{murota2003DCA}.

  A subset $S\subset \N^n$ is an \textit{M-convex set} if it satisfies the following exchange property: for any $\bx,\by\in S$ and for any $i$ such that $\bx(i)>\by(i)$, there is an index $j$ such that
\begin{center} \vspace{5pt}
    $\by(j)>\bx(j)$ and $\bx-\be_i+\be_j\in S$ and $\by-\be_j+\be_i\in S$.
\vspace{5pt} \end{center} 
 There are two special M-convex sets. The \textit{discrete simplex} 
\begin{center} \vspace{5pt}
    $\Delta^d_n:=\{\ba\in \N^n\mid |\ba|_1=d\}$
\vspace{5pt} \end{center} 
and the set $\nk{n}{d}$ of subsets of $[n]$ of size $d$, regarded as the 0-1 vectors in $\Delta^d_n$. 

The \textit{support} of a function $\vp\colon \N^n\to \Rbar$ is 
\begin{equation}
    \supp(\vp):=\{\ba\in \N^n\mid \vp(\ba)<\infty\}.
\end{equation}
We say $\vp$ is an \textit{M-convex function} if $\supp(\vp)\neq \emptyset$ and it satisfies the following exchange property: for any $\bx,\by\in \supp(\vp)$ and any $i$ such that $\bx(i)>\by(i)$, there is $j$ such that
\begin{center} \vspace{5pt}
    $\bx(j)<\by(j)$ and $\vp(\bx)+\vp(\by)\geq \vp(x-\be_i+\be_j)+\vp(\by-\be_j+\be_i)$.
\vspace{5pt} \end{center} 
If $\vp$ is M-convex, then $\supp(\vp)$ is an M-convex set. Any M-convex set is contained in some $\Delta^d_n$, and $d$ is called the \textit{rank} of $\vp$. From now on, we will use $\Delta^d_n$ as the domain of an M-convex function. On the other hand, $S\subset \N^n$ is M-convex if and only if its \textit{indicator function}
\begin{equation}
    \delta_S(\bx) = \begin{cases}
        0, & \text{ if }\bx\in S \\
        \infty, & \text{ if }\bx\notin S
    \end{cases}
\end{equation}
is M-convex.

 M-convex functions generalize valuated matroids and matroids: a valuated matroid is an M-convex function whose domain is $\nk{n}{d}$; the support of a valuated matroid is the set of bases of its underlying matroid. The indicator function of $\cB(M)$ is the trivial valuation on $M$. 

\subsection{Lorentzian polynomials}\label{subsec:prelim-lorentzian}

We review the definition of Lorentzian polynomials and the relation of the Lorentzian property with M-convexity. For more advanced materials on this topic, we refer to the original paper \cite{branden2020lorentzian}.

  Each polynomial $f\in \R_{\geq 0}[w_1,...,w_n]$ can be written uniquely as $\sum_{\ba\in \N^n}c_\ba w^\ba$. The \textit{support} of $f$ is
\begin{center} \vspace{5pt}
    $\supp f:=\{\ba\in \N^n\mid c_\ba\neq 0\}$.
\vspace{5pt} \end{center} 
A homogeneous polynomial $f\in \R_{\geq 0}[w_1,...,w_n]$ is \textit{Lorentzian} if it satisfies the following conditions:

\begin{itemize}
    \item Support condition: $\supp f$ is M-convex;
    \item Signature condition: for any $\ba\in \N^n$ and $|\ba|_1=d-2$, $\partial^\ba f$ is Lorentzian. That is, the quadratic form $\partial^\ba f$ has at most one positive eigenvalue.
\end{itemize}
If $f$ is Lorentzian, then all its partial derivatives are Lorentzian.

A polynomial $f\in \R_{\geq0}[w_1,...,w_n]$ is \textit{stable} if for all $\bu \in \R^n_{\geq0}$ and all $\bv\in \R^n$, the univariate polynomial $f(t \bu +\bv)$ is real-rooted. Put
\begin{equation}
\begin{split}
    &\rmS^d_n:=\big\{f\in \R_{\geq 0 }[w_1,...,w_n]\mid f\text{ is homogeneous stable of degree }d\big\},\\
    &\rmL^d_n:=\big\{f\in \R_{\geq 0 }[w_1,...,w_n]\mid f\text{ is Lorentzian of degree }d\big\}.
\end{split}
\end{equation}
For instance, $\rmS^1_n=\rmL^1_n$ are linear forms with nonnegative coefficients, $\rmS^2_n=\rmL^2_n$, and in general, $\rmS^d_n\subset \rmL^d_n$. 

For any function $\vp\colon \Delta^d_n\to \Rbar$ and a parameter $0<q\leq 1$, set
\begin{equation}\label{eqn:basis-generating}
    f^\vp_q(w) = \sum_{\ba\in \supp(\vp)}\frac{q^{\vp(\ba)}}{\ba!}w^\ba,\quad \text{ and }\quad g^\vp_q(w) = \sum_{\ba\in \supp(\vp)}\binom{d}{\ba}q^{\vp(\ba)}w^\ba,
\end{equation}
where $\ba!=\prod_{i=1}^n\ba(i)!$, and $\binom{d}{\ba}=\prod_{i=1}^n\binom{d}{\ba(i)}$. We call $f^\vp_q$ the \textit{basis generating polynomial} of $\vp$. When $\vp$ is the trivial valuation on a matroid $M$, we get the usual basis generating polynomial $f_M$ for $M$,
\begin{equation}
    f_M(w) = \sum_{B\in \cB(M)}w^B.
\end{equation}
Recall the following Lorentzian characterization of M-convex functions \cite[Theorem 3.14]{branden2020lorentzian}: 
\begin{center} \vspace{5pt}
    $\vp$ is M-convex if and only if $f^\vp_q\in \rmL^d_n$ for all $0<q\leq 1$.
\vspace{5pt} \end{center} 
This characterization has a tropical version. Let $\K=\R\{t\}$ be the field of \textit{generalized Puiseux series} over $\R$. Elements of $\K$ are formal series of the form
\begin{equation}
    s(t)=\sum_{k\in A}c_kt^k,
\end{equation}
where $A$ is well-ordered and $c_k\in \R$. An element $s\in \K$ is positive if $c_k>0$ for $k=\min A$. Moreover, $\K$ has a valuation,
\begin{equation}
    \val\colon\K \to \Rbar, \quad 0\mapsto \infty,\quad \sum_{k\in A}c_kt^k\mapsto \min A.
\end{equation}
 Lorentzian polynomials can be defined over $\K$. A polynomial over $\K$ is Lorentzian if and only if it is Lorentzian as a real polynomial for all sufficiently small parameter $t$. For more details on changing real coefficients to $\K$, we refer to \cite{branden2010discrete} and the references therein. Put
\begin{equation}
    \rmL^d_n(\K):=\Big\{f\in \K_{\geq 0 }[w_1,...,w_n]\mid f\text{ is Lorentzian of degree }d\Big\}.
\end{equation} 
Given $f=\sum_{\ba\in \Delta^d_n}s_{\ba}(t)w^\ba$, the \textit{tropicalization} of $f$ is a discrete function given by
\begin{equation}
    \trop(f)\colon \Delta^d_n \to \Rbar,\quad \ba\mapsto \val(s_\ba(t)).
\end{equation}
Then the set of M-convex functions on $\Delta^d_n$ is exactly the tropicalization of $\rmL^d_n(\K)$ \cite[Theorem 3.20]{branden2020lorentzian}. 

\section{Lorentzian proper position}\label{sec:Lorentzian-proper-position}

\subsection{Basic properties}

In this subsection, we summarize some properties of Lorentzian proper position. The highlight is that many linear operations preserving the Lorentzian property also preserve Lorentzian proper position.

\begin{prop}\label{prop:proper-position-properties}
    Let $f\in \rmL^d_n$ and $h\in \rmL^{d-1}_n$.
    \begin{enumerate}
        \item If $h\ll_L f$, then $\partial_\bv h \ll_L \partial_\bv f + t h$ for any $\bv\in \R^n_{\geq 0}$ such that $\partial_\bv h\neq 0$ and any $t\geq 0$.
        \item For any linear form $\ell\in \rmL^1_n(\R)$, $f\ll_L \ell f$.
        \item If $f\in \rmL^2_n$, then for any linear form $\ell\in \rmL^1_n(\R)$, $\ell\ll_L f$ if and only if $\ell \ll f$.
        \item For any $\bv\in \R^n_{\geq 0}$, $\partial_\bv f\ll_L f$ whenever $\partial_\bv f\neq 0$.
        \item If $h\ll_L f$, then $th\ll_L f$ for any $t>0$. 
    \end{enumerate}
\end{prop}

\begin{proof}
    If $f+w_{n+1}h\in \rmL^d_{n+1}$, then for $\bv'=(\bv,t)\in \R_{\geq 0}^{n+1}$, $\partial_{\bv'}(f+w_{n+1}h)=\partial_\bv f + t h + w_{n+1}\partial_\bv h\in \rmL^{d-1}_{n+1}$. This proves the first statement. As a product of Lorentzian polynomials, $(\ell+w_{n+1})f=\ell f + w_{n+1}f$ is Lorentzian, proving the second statement. Since $\rmS^2_{n+1}=\rmL^2_{n+1}$, the third statement follows. The linear operator 
    \begin{equation}
        \R[w_1,...,w_n]\mapsto \R[w_1,...,w_n,w_{n+1}],\quad f\mapsto f+w_{n+1}\partial_\bv f
    \end{equation}
    is homogeneous, It preserves stability and polynomials with nonnegative coefficients, so by \cite[Theorem 3.4]{branden2020lorentzian}, $f+w_{n+1}\partial_\bv f$ is Lorentzian, proving the forth. The fifth statement is because scaling the variable $w_{n+1}$ preserves the Lorentzian property.
\end{proof}

For any tuple $\bh\in \N^n$, set 
\begin{equation}
    \R_\bh[w_1,...,w_n] = \Big\{f\in \R[w_1,...,w_n]\mid \text{ the degree of $f$ is at most $\bh(i)$ in $w_i$ for every $i$}\Big\}.
\end{equation}

\begin{definition}\label{def:preservation-proper-position}
    Let $T\colon\R_\bh[w_1,...,w_n]\to \R_\br[z_1,...,z_m]$ be a linear operator. We say $T$ \textit{preserves Lorentzian proper position}, if $T$ preserves the Lorentzian property and $T(f)\ll_LT(g)$ whenever $f\ll_L g$.
\end{definition}

A linear operator $T\colon\R_\bh[w_1,...,w_n] \to \R_\br[z_1,...,z_m]$ is \textit{homogeneous} if for some $s\in \Z$,
\begin{center} \vspace{5pt}
    $\deg T(w^\ba) = \deg w^\ba + s$ whenever $T(w^\ba)\neq 0$.
\vspace{5pt} \end{center} 
In \cite[Theorem 4.3]{borcea2010multivariate}, it is shown that a finite-order differential operator preserving stability also preserves proper position. It is also shown there that such an operator preserves stability if and only if its symbol is stable. On the other hand, a homogeneous linear operator with a Lorentzian symbol preserves the Lorentzian property \cite[Theorem 3.2]{branden2020lorentzian}. The following \Cref{prop:Lorentzian-symbol} can be regarded as the Lorentzian analog of \cite[Theorem 4.3]{borcea2010multivariate}.

In the following proofs, for $\ba,\bh\in \N^n$,
\begin{center} \vspace{5pt}
    $\ba \leq \bh$ means $\ba(i)\leq \bh(i)$ for all $i=1,...,n$;\quad  $\binom{\bh}{\ba}=\prod_{i=1}^n\binom{\bh(i)}{\ba(i)}$.
\vspace{5pt} \end{center} 
\begin{prop}\label{prop:Lorentzian-symbol}
    If $T$ is a homogeneous linear operator with a Lorentzian symbol, then $T$ preserves Lorentzian proper position.
\end{prop}

\begin{proof}
    Let $T\colon\R_\bh[w_1,...,w_n]\to \R_\br [z_1,...,z_m]$ be a linear operator satisfying the hypothesis. Define another linear operator $\td{T}\colon\R_{\td{\bh}}[w_0,w_1,...,w_n]\to \R_{\td{\br}} [z_0,z_1,...,z_m]$, where 
    \begin{equation}
        \td{\bh}(i) = \begin{cases}
            1,& \text{ if }i=0 \\
            \bh(i),&\text{ else}
        \end{cases}, \quad  \td{\br}(i) = \begin{cases}
            1,& \text{ if }i=0 \\
            \br(i),&\text{ else}
        \end{cases},
    \end{equation}
    by the formula $\td{T}(w_0^iw^\ba)=z_0^iT(w^\ba)$ and extending linearly. Then $\td{T}$ is homogeneous because $T$ is. Note that for $f,g\in \R_\bh[w_1,...,w_n]$, $\td{T}(g+w_0f) = T(g) + z_0T(f)$, so it suffices to show that $\td{T}$ preserves the Lorentzian property. We compute its symbol
    \begin{equation}
        \begin{split}
            \sym_{\td{T}}(u_0,u_1,...,u_n,z_0,z_1,...,z_m) & = \sum_{\ba\leq \bh}\binom{\bh}{\ba}u_0T(w^\ba)u^{\bh-\ba} + \sum_{\ba\leq \bh}\binom{\bh}{\ba}T(w^\ba)u^{\bh-\ba}z_0 \\
             & = \sum_{\ba\leq \bh}\binom{\bh}{\ba}T(w^\ba)u^{\bh-\ba}(u_0+z_0) \\
            & = \sym_T(u,z)(u_0+z_0).
        \end{split}
    \end{equation}
Since $\sym_T(u,z)$ and $u_0+z_0$ are both Lorentzian and the product of Lorentzian polynomials is Lorentzian, we have that $\sym_{\td{T}}$ is Lorentzian.
\end{proof}

Abbreviate $\R_\bh[w_1,...,w_n]$ to $\R_\bh[w_i]$ and put
\begin{equation}
    \R_\bh^\text{a}[w_{ij}] = \big\{\text{multi-affine polynomials in }\R[w_{ij}]_{1\leq i\leq n,1\leq j\leq \bh(i)}\big\}.
\end{equation}
The \textit{polarization} operator $\pol_\bh\colon\R_\bh[w_i]\to\R_\bh^\text{a}[w_{ij}]$ is a linear map that sends $w^\ba$ to 
\begin{equation}
    \frac{1}{\binom{\bh}{\ba}}\prod_{i=1}^n\left(\text{elementary symmetric polynomial of degree $\ba(i)$ in the variables $\{w_{ij}\}_{1\leq j\leq \bh(i)}$}\right).
\end{equation}
The \textit{projection} operator $\proj_\bh\colon\R_\bh^\text{a}[w_{ij}]\to \R_\bh[w_i]$ is a linear map that sets $w_{ij}=w_i$. Polarization and projection preserve Lorentzian property \cite[Proposition 3.1]{branden2020lorentzian}. They also preserve Lorentzian proper position.

\begin{prop}\label{prop:polarization-projection}
 Polarization and projection preserve Lorentzian proper position. Let $f,g\in \R_\bh[w_i]$ be Lorentzian of degree $d$ and $d-1$, respectively. Then $g\ll_L f$ if and only if $\pol g\ll_L \pol f$.
\end{prop}

\begin{proof}
    The preservation of Lorentzian property follows from applying \cite[Proposition 3.1]{branden2020lorentzian} to $\R_{\td{\bh}}[w_i]$ and $\R_{\td{\bh}}[w_{ij}]$, where $\td{\bh}$ is as in the proof of \Cref{prop:Lorentzian-symbol}. The rest of the statement follows from $\proj_\bh\circ \pol_\bh(f)=f$.
\end{proof}

\subsection{Lorentzian characterization of elementary quotients of M-convex functions}

In this subsection, we relate Lorentzian proper position to \textit{quotients of M-convex functions} studied in \cite{brandenburg2024quotients}. There, the authors give four candidate definitions in different strength of quotients of M-convex functions.
We take the following one.
\begin{definition}\label{def:quotient-M-convex-fn}
    Given two M-convex functions $\vp\colon \Delta^d_n\to \Rbar$ and $\psi\colon \Delta^s_n\to\Rbar$ where $s\leq d$, $\psi$ is called a \textit{quotient} of $\vp$, denoted $\vp\onto \psi$, if the following exchange property is satisfied: for any $\bx\in \supp(\psi)$, $\by\in\supp(\vp)$, and $i$ such that $\bx(i)>\by(i)$, there is some $j$ such that
\begin{center} \vspace{5pt}
    $\bx(j)<\by(j)$\quad  and \quad $\psi(\bx)+\vp(\by)\geq \psi(\bx-\be_i+\be_j)+\vp(\by-\be_j+\be_i)$.
\vspace{5pt} \end{center} 
\end{definition}

Quotients of M-convex functions generalize quotients of valuated matroids and quotients of ordinary matroids. Hence, we will use the notation  $\onto$ universally for quotients and use the default letters to distinguish the objects:
\begin{center} \vspace{5pt}
\begin{tabular}{cl}
    $\vp\onto \psi$ & a quotient of M-convex function;  \\
    $\mu\onto \theta$ &  a quotient of valuated matroids; \\
    $M\onto N$ &  a quotient of matroids.
\end{tabular}
\vspace{5pt} \end{center} 
When $s=d-1$, $\psi$ is called an \textit{elementary quotient} of $\vp$. We will need the following characterization of elementary quotients, which follows from \cite[Theorem 1.3]{brandenburg2024quotients}.

\begin{lemma}\label{lem:elementary-quotient-equiv-description}
    Let $\vp\colon \Delta^d_n\to \Rbar$ and $\psi\colon \Delta^{d-1}_n$ be two M-convex functions. Then $\vp\onto \psi$ if and only if the function $\hat{\vp}\colon\Delta^d_{n+1}\to \Rbar$ defined by 
        \begin{equation}
        \hat{\vp}(\bx,x_{n+1}) = \begin{cases}
            \vp(\bx),& \text{ if }x_{n+1}=0 \\
            \psi(\bx),& \text{ if }x_{n+1}=1 \\
            \infty,&\text{ else }
        \end{cases}
    \end{equation}
        is M-convex.
\end{lemma}

\begin{rmk}
    If $\vp\onto \psi$ and $\vp$ is a quotient of M-convex functions, and $\vp$ is a valuated matroid, then $\psi$ is automatically also a valuated matroid.
\end{rmk}

By applying the definition of quotients of M-convex functions to indicator functions of M-convex sets, we get the definition of \textit{quotients of M-convex sets}. Given two M-convex sets $S_1\subset \Delta^d_n$ and $S_2\subset \Delta^s_n$ where $s\leq d$, $S_1$ is a quotient of $S_2$, denoted $S_1\onto S_2$, if the following exchange property is satisfied: for any $\bx\in S_2$, $\by\in S_1$, and $i$ such that $\bx(i)>\by(i)$, there is some $j$ such that 
\begin{equation}\label{eqn:quotient-M-convex-set}
    \bx(j)<\by(j),\quad  \bx - \be_i+\be_j\in S_2,\quad \text{and}\quad \by-\be_i+\be_j\in S_1.
\end{equation} \Cref{lem:elementary-quotient-equiv-description} then implies that if $s=d-1$, then $S_1\onto S_2$ is an elementary quotient if and only if $S_1\times \{0\}\cup S_2\times \{1\}\subset \Delta^{d}_{n+1}$ is M-convex.

\begin{definition}\label{def:factorization}
    Let $S_0\subset \Delta^d_n$ and $S_k\subset \Delta^{d-k}_n$ be M-convex and $S_0\onto S_k$. A \textit{factorization} of $S_0\onto S_k$ is a collection of M-convex sets $S_1,...,S_{k-1}$ such that $S_i\onto S_{i+1}$ is an elementary quotient for $i=0,...,k-1$. If a factorization $S_0\onto \cdots \onto S_k$ satisfies the condition:
\begin{center} \vspace{5pt}
    for any factorization $S_1',...,S_{k-1}'$, $S_i'\subset S_i$ for all $i=1,...,k-1$,
\vspace{5pt} \end{center} 
then $S_0\onto S_1\onto \cdots \onto S_{k-1}\onto S_k$ is called a \textit{Higgs factorization}.
\end{definition}
 Higgs factorizations exist for M-convex sets and, in  particular, matroids. By definition, the Higgs factorization is unique. For more details, see \cite[Lemma 2.36]{brandenburg2024quotients}. We will use the following fact: let $S_0$ and $S_k$ be nonempty M-convex sets. Then $S_0\onto S_1 \onto \cdots\onto S_k$ is a Higgs factorization if and only if 
 \begin{equation}
     \td{S} = \bigcup_{i=0}^k S_i\times \{i\} \subset \Delta^d_{n+k}
 \end{equation}
 is M-convex.
 
\begin{definition}
    Let $f$ be any polynomial in $\R[w_0,w_1,...,w_n]$ whose degree in $w_0$ is $d$. Write
\begin{equation}\label{eqn:single-out-a-variable}
    f = f_0 + w_0f_1 + \cdots + w_0^d f_d,
\end{equation}
where $f_i\in \R[w_1,...,w_n]$. Given two integers $i,j$ such that $0\leq i < j \leq d$, the \textit{$[i,j]$ segment} of $f$, denoted $f_{[i,j]}$, is the polynomial $w_0^if_i + w_0^{i+1}f_{i+1}+\cdots + w_0^jf_j$.
\end{definition}
   The following statement shows taking any segment preserves the Lorentzian property. As a corollary, from any Lorentzian polynomial one can extract a sequence of Lorentzian polynomials where each consecutive pair is in proper position.

\begin{prop}\label{prop:interval-preserves-lorentzian}
    Let $f$ given as in \eqref{eqn:single-out-a-variable} be Lorentzian. Suppose $f_0\neq 0$. Then 
    \begin{enumerate}
        \item $f_{[i,j]}$ is Lorentzian for all $0\leq i <j \leq d$.
        \item $\supp f_0\onto \supp f_1\onto \cdots \onto \supp f_d$ is a Higgs factorization.
    \end{enumerate}
\end{prop}

\begin{proof}
    We reuse the notation $\td{\bh}$ from the proof of \Cref{prop:Lorentzian-symbol}. The map $f\mapsto f_{[i,j]}$ is the linear operator 
    \begin{equation}
        T\colon\R_{\td{\bh}}[w_0,...,w_n]\to \R_{\td{\bh}}[w_0,...,w_n],\quad  w_0^mw^\ba\mapsto \begin{cases}
            0,& \text{ if }m\notin [i,j] \\
            w_0^mw^\ba,&\text{ else }
        \end{cases} 
    \end{equation}
    which is homogeneous of degree 1. We compute its symbol:
    \begin{equation}
        \begin{split}
            \sym_T(u_0,...,u_n,w_0,...,w_n) & = \sum_{i\leq m\leq j}\sum_{\ba \leq \bh}\binom{d}{m}\binom{\bh}{\ba} w_0^mu_0^{d-m}w^\ba u^{\bh-\ba}  \\
            & = \left(\sum_{i\leq m\leq j}\binom{d}{m}w_0^mu_0^{d-m}\right)\prod_{i=1}^n \left(\sum_{\ba(i)\leq \bh(i)}\binom{\bh(i)}{\ba(i)} w_i^{\ba(i)}u^{\bh(i)-\ba(i)} \right)
        \end{split}
    \end{equation}
    This a product of bivariate polynomials. Each of them is Lorentzian, due to the ultra log-concavity of the binomial coefficients \cite[Example 2.26]{branden2020lorentzian}. Hence, $\sym_T$ is Lorentzian. This proves the first statement. The second follows from the M-convexity of $\supp f$.
\end{proof}

\begin{cor}\label{cor:proper-position-from-single-poly}
    Let $f\in \rmL^d_{n+1}$ have the form as in \eqref{eqn:single-out-a-variable}. If $f_i$ and $f_{i+1}$ are nonzero, then $f_{i+1}\ll_L f_i$.
\end{cor}
\begin{proof}
    By \Cref{prop:interval-preserves-lorentzian}, $w_0^if_i + w_0^{i+1}f_{i+1}$ is Lorentzian, so $i!f_i+(i+1)!w_0f_{i+1}$ is Lorentzian. The conclusion follows from \Cref{prop:proper-position-properties} (5).
\end{proof}

The next two statements are the Lorentzian characterization of elementary quotients of M-convex functions that parallel the Lorentzian characterization of M-convex functions, from which \Cref{main:A} follows.

\begin{prop}\label{prop:M-convex-quotient-Lorentzian}
    Let $\vp\colon \Delta^d_n\to \Rbar$ and $\psi\colon \Delta^{d-1}_n\to \Rbar$ be M-convex functions. The following are equivalent.
    \begin{enumerate}
        \item $\vp\onto \psi$ is an elementary quotient of M-convex functions.
    \item $f^\psi_q\ll_L f^\vp_q $ for all $0<q\leq 1$.\end{enumerate}
\end{prop}

\begin{proof}
 By \Cref{lem:elementary-quotient-equiv-description}, $\vp\onto\psi$ if and only if the auxiliary function 
    \begin{equation}
        \hat{\vp}(\bx,x_{n+1}) = \begin{cases}
            \vp(\bx),& \text{ if }x_{n+1}=0 \\
            \psi(\bx),& \text{ if }x_{n+1}=1 \\
            \infty,&\text{ else }
        \end{cases}
    \end{equation}
    is M-convex on $\Delta^d_{n+1}$. Note that $f^\vp_q + w_{n+1} f^\psi_q=f^{\hat{\vp}}_q$. Now the equivalence follows from \cite[Theorem 3.14]{branden2020lorentzian}.
\end{proof}

We state a tropical version of this characterization. Define Lorentzian proper position for polynomials over $\K$ analogous to the real case: for $f\in \rmL_n^d(\K)$ and $g\in \rmL_n^{d+1}(\K)$,
\begin{center} \vspace{5pt}
    $f\ll_L g$ \quad if\quad  $g+w_{n+1}f\in \rmL^{d+1}_{n+1}(\K)$. 
\vspace{5pt} \end{center} 
\begin{prop}\label{prop:quotient-characterization-puiseux}
If $f\in \rmL^d_n(\K)$, $h\in \rmL^{d-1}_n(\K)$ and $h\ll_L f$, then $\trop (f) \onto \trop (h)$. Conversely, if $\vp\colon  \Delta^d_n\to \Rbar$, $\psi\colon \Delta^{d-1}_n\to\Rbar$, and $ \vp\onto \psi$ is an elementary quotient of M-convex functions, then there are $f\in \rmL^d_n(\K)$ and $h\in \rmL^{d-1}_n(\K)$ such that $h\ll_L f$, $\trop(h)=\psi$ and $\trop(f)=\vp$.
\end{prop}

\begin{proof}
    This is an immediate consequence of \Cref{lem:elementary-quotient-equiv-description} and \cite[Theorem 3.20]{branden2020lorentzian}.
\end{proof}

\subsection{The convexity property of Lorentzian proper position}

 Let $f\in\rmL^d_n$. Consider the following subsets of polynomials induced by Lorentzian proper position.
\begin{equation}
\begin{split}
    \rmL[f\ll_L]:=\big\{g\in\rmL_n^{d+1} \mid f\ll_L g\big\}\cup\{0\}, \\
    \rmL[\ll_Lf]:=\big\{h\in\rmL^{d-1}_n \mid h\ll_L f\big\}\cup\{0\}.
\end{split}
\end{equation}

We are ready to prove \cref{prop:A}, which states the convexity of $\rmL[f\ll_L]$. 

\begin{lemma}\label{lem:lorentzian-property-lemma}
    Let $f,g\in \R_{\geq 0}[w_1,...,w_n]$ be homogeneous of degree $d$ and $d+1$, respectively. Then 
    $f\ll_L g$ if and only if $g+w_{n+1}f$ has M-convex support, $f\in \rmL^d_n$, and $\partial_{\bv_1}\cdots \partial_{\bv_{d-1}}f \ll \partial_{\bv_1}\cdots \partial_{\bv_{d-1}} g$ for any $\bv_1,...,\bv_{d-1}\in \R^n_{>0}$.
\end{lemma}

\begin{proof}
    Note that $\partial_{\bv_1}\cdots \partial_{\bv_{d-1}}f\neq 0$ is a linear form for any $\bv_1,...,\bv_{d-1}\in \R^n_{>0}$. The `only if' direction follows from \Cref{prop:proper-position-properties} (1) and (3). Since the Lorentzian property is a closed condition, sufficiency follows by letting $\bv_1,...,\bv_{d-1}$ approach to standard basis vectors. 
\end{proof}

\begin{lemma}\label{lem:support-of-lorentzian-above}
    Let $f\in\rmL^d_n$ and $f\ll_L g$. Let $\ell\in \rmL^1_n$ be a linear form whose coefficients are all positive. Then
\begin{equation}
 \supp g \subset  \supp \ell f.
\end{equation}
\end{lemma}

\begin{proof}
    It suffices to show that for each $\by\in \supp g$,
    \begin{equation}
        \min_{\bx\in \supp f} |\bx-\by|_1 = 1.
    \end{equation}
    By \Cref{prop:interval-preserves-lorentzian}, $\supp g\onto \supp f$. Take any $\by\in \supp g$ and $\bx\in \supp f$. Since $|\by|_1=d+1$ and $|\bx|_1=d$, if $|\bx-\by|_1> 1$, there must be some $i$ such that $\bx(i)>\by(i)$. Using the exchange axiom, we get $\bx'=\bx-\be_i+\be_j\in \supp f$ for some $j$ such that $\by(j)>\bx(j)$. Now $|\bx'-\by|_1=|\bx-\by|_1-2$. The proof is finished by induction.
\end{proof}

\begin{repprop}{prop:A}
    Let $f\in \R[w_1,...,w_n]$ be Lorentzian of degree $d$. Then $\rmL[f\ll_L]$ is a closed convex cone.
\end{repprop}

\begin{proof}
    Since the Lorentzian property is a closed condition and $\rmL[f\ll_L]$ is defined by the Lorentzian property of $g+w_{n+1}f$, we conclude that $\rmL[f\ll_L]$ is closed. Let $g_1,g_2\in \rmL[f\ll_L]$. By \Cref{prop:proper-position-properties} (5), we only need to show $g_1+g_2\in \rmL[f\ll_L]$. Let $g_0=\ell f$ where $\ell\in \rmL^1_n$ has all-positive coefficients. By \Cref{lem:support-of-lorentzian-above}, whenever $0<t\leq 1$, the support of $tg_0 + (1-t)(g_1+g_2)+w_{n+1}f$ is the support of $(\ell+w_{n+1}) f$, which is M-convex. By \Cref{lem:lorentzian-property-lemma}, for any $\bv_1,...,\bv_{d-1}\in \R^n_{>0}$,
    \begin{equation}
        \partial_{\bv_1}\cdots \partial_{\bv_{d-1}} f\ll \partial_{\bv_1}\cdots \partial_{\bv_{d-1}} g_i, \quad i=0,1,2.
    \end{equation}
    By \Cref{prop:convexity-stable-proper-position}, 
    \begin{equation}
        \partial_{\bv_1}\cdots \partial_{\bv_{d-1}} f\ll \partial_{\bv_1}\cdots \partial_{\bv_{d-1}}(tg_0 + (1-t)(g_1+g_2)).
    \end{equation}
    This shows $tg_0 + (1-t)(g_1+g_2)\in \rmL[f\ll_L]$ for $0<t\leq 1$. Letting $t\to 0$ finishes the proof.
\end{proof}

\begin{rmk}\label{rmk:branden-proof}
The following simple argument for \Cref{prop:A} was communicated to us by Petter Br\"{a}nd\'{e}n. Suppose $f\ll_L g_1,g_2$. Let $c>0$, $\ell=w_1+w_2+\cdots+ w_n$, and put
\begin{equation}
    \begin{split}
       g_1'=g_1(w_1 + c\ell,w_2+c\ell,...,w_n+c\ell),\\
       g_2'=g_2(w_1 + c\ell,w_2+c\ell,...,w_n+c\ell), \\
       f' = f(w_1 + c\ell,w_2+c\ell,...,w_n+c\ell). 
    \end{split}
\end{equation}
Then $g_1',g_2'$ and $f'$ all have full support, so $g_1'+g_2'$ has M-convex support. Then by taking derivatives, it follows from \Cref{prop:convexity-stable-proper-position} that $f\ll_L g_1'+g_2'$. \Cref{prop:A} follows from letting $c\to 0$.
\end{rmk}

One can also talk about convex cones over $\K$: for a vector space $V$ over $\K$, $S\subset V$ is a convex cone if $c_1s_1+c_2s_2\in S$ for any $s_1,s_2\in S$ and nonnegative scalars $c_1,c_2\in \K$. Similar to the real case, for $f\in\rmL^d_n(\K)$ consider
\begin{equation}
\rmL[f\ll_L]_\K:=\big\{g\in \rmL^{d+1}_n(\K) \mid f\ll_L g\big\}\cup\{0\}.
\end{equation}

\begin{thm}\label{thm:Lorentzian-convexity-puiseux}
    Let $f\in\K[w_1,...,w_n]$ be Lorentzian of degree $d$. Then $\rmL[f\ll_L]_\K$ is a convex cone.
\end{thm}

\begin{proof}
    The convexity of $\rmL[f\ll_L]$ is a first-order statement. Since $\K$ is real closed, $\rmL[f\ll_L]_\K$ is a convex cone by Tarski's principle.
\end{proof}

We now deduce \Cref{main:A} as part of \Cref{cor:codim-1-tropical-convex} from \Cref{thm:Lorentzian-convexity-puiseux}. Given any M-convex function $\vp\colon \Delta^d_n\to \Rbar$, consider the following two sets of M-convex functions.
\begin{equation}
\begin{split}
    \sF^1(\vp):=\big\{\psi\colon \Delta^{d-1}_n\to \Rbar \text{ M-convex}\mid \vp\onto \psi\big\},\\
     \leftindex^1{\sF}(\vp):=\big\{\rho\colon \Delta^{d+1}_n\to \Rbar \text{ M-convex}\mid \rho\onto \vp\big\}.
\end{split}
\end{equation}
Likewise, put:
\begin{equation}\label{eqn:incidence-dressian}
     \leftindex^1{\Dr}(\mu):=\big\{[\nu]\in \Dr(d+1,n)\mid \Trop\mu\subset \Trop\nu \big\}.
\end{equation}

\begin{cor}\label{cor:codim-1-tropical-convex}
    The tropicalization map
    \begin{equation}
        \Trop\colon \rmL[f\ll_L]_\K \to \leftindex^1{\sF}(\trop f),\quad g\mapsto \trop g
    \end{equation}
    is surjective and it respects convex combinations. In particular, $\sF^1(\vp)$, $\leftindex^1{\sF}(\vp)$, $\Dr^1(\mu)$ and $\leftindex^1{\Dr}(\mu)$ are tropically convex.
\end{cor}

\begin{proof}
    The well-definedness of $\Trop$ and its surjectivity follow from \Cref{prop:quotient-characterization-puiseux}. $\Trop$ respects convex combinations, because for nonnegative elements $t_1,t_2,s_1,s_2\in\K$, $\val(t_1s_1+t_2s_2)=\min(\val(t_1)+\val(s_1),\val(t_2)+\val(s_2))$. This proves the tropical convexity of $\leftindex^1{\sF}(\vp)$. Let $D\geq d$ be an integer. Let $\mathbf{D}$ be the $n$-tuple $(D,...,D)$. Consider the operation $\vp\mapsto \td{\vp}$, where $\td{\vp}(\bx)=\vp(\mathbf{D}-\bx)$. It is easy to check that 
    \begin{center} \vspace{5pt}
        $\td{\vp}$ is M-convex if and only if $\vp$ is M-convex
    \vspace{5pt} \end{center} 
    and 
    \begin{center} \vspace{5pt}
        $\psi\onto\vp$ if and only if $\td{\vp}\onto \td{\psi}$.
    \vspace{5pt} \end{center} 
     Hence, $\sF^1(\vp)$ is also tropically convex. The rest of the statement follows by considering functions whose supports are contained in the unit cube $\{0,1\}^n\subset\N^n$.
\end{proof}

 \Cref{cor:codim-1-tropical-convex} has the following consequence, which \textit{a priori} is not obvious. Let $\psi_1,\psi_2\colon \Delta^{d-1}_n\to \Rbar$ be M-convex functions. If they are elementary quotients of some M-convex function $\vp\colon \Delta^{d}_n\to \Rbar$, then $\psi_3=\min\{\psi_1,\psi_2\}$ is also M-convex and $\vp\onto\psi_3$. This further implies a combinatorial statement about \textit{matroid perspectivity} \cite[Chapter 7]{white1986theory}.

\begin{definition}\label{def:matroid-perspectivity}
    A family of matroids $M_1,...,M_k$ are \textit{perspective} if they have a common elementary quotient $Q$; they are \textit{coperspective} if they are all elementary quotients of some matroid $L$. 
\end{definition}

In general, the union of the sets of bases of matroids is rarely the set of bases of any matroid, whereas this happens when the matroids are perspective or coperspective.

\begin{cor}\label{cor:union-basis-families}
    If $M_1,...,M_k$ are perspective (coperspective, respectively), then $\cB(M_1)\cup\cdots\cup\cB(M_k)$ is the set of bases of another matroid $M_{k+1}$, and $M_1,...,M_{k+1}$ are perspective (coperspective, respectively).
\end{cor}
\begin{proof}
    For any functions $\mu_1,\mu_2\colon \Delta^d_n \colon \Rbar$, $\supp (\min\{\mu_1,\mu_2\}) = \supp (\mu_1)\cup \supp (\mu_2)$. Apply \Cref{cor:codim-1-tropical-convex} to the trivial valuations on the matroids.
\end{proof}

We end this subsection with an example showing $\rmL[\ll_L f]$ may not be convex.
\begin{ex}\label{ex:L1-not-convex}
    Consider $M=U_{3,4}$. Polynomials in $\rmL[\ll_Lf_M]$ are multi-affine quadratic forms $\frac{1}{2}w^TAw$ for some 4-by-4 matrix $A$. Let $h_1=\frac{1}{2}w^TA_1w$ and $h_2=\frac{1}{2}w^TA_2w$ where
    \begin{equation}
        A_1 = \begin{bmatrix}
            0 & 3.9 & 1 & 1 \\
            3.9 & 0 & 1 & 1 \\
            1 & 1 & 0 & 1 \\
            1 & 1 & 1 & 0
        \end{bmatrix},\quad A_2 = \begin{bmatrix}
            0 & 1 & 1 & 1 \\
            1 & 0 & 1 & 1 \\
            1 & 1 & 0 & 3.9 \\
            1 & 1 & 3.9 & 0
        \end{bmatrix}
    \end{equation}
The leading principal minors of $A_1$ have signs $0,-,+,-$. By Cauchy's interlacing theorem, $h_1$ is Lorentzian. Similar calculation shows $h_2$ is also Lorentzian and $h_1,h_2\ll_L f_M$. However, the determinant of $A_1+A_2$ is positive, so $h_1+h_2$ cannot be Lorentzian. 

From this one sees that the Lorentzian property is not preserved under inversion. For any polynomial $f$, consider the operator
\begin{equation}
    \inv: f \mapsto f^{\inv}:= w_1^{d_1}\cdots w_n^{d_n}f(w_1^{-1},...,w_n^{-1})
\end{equation}
where $d_i$ is the degree of $f$ in $w_i$. Let $g_1=f_M + w_5h_1$ and $g_2=f_M+w_5h_2$, which are Lorentzian. Then either $g_1^{\inv}=w_5f_M^{\inv} + h_1^{\inv}$ or $g_2^{\inv}=w_5f_M^{\inv} + h_2^{\inv}$ is not Lorentzian. Otherwise, by \Cref{prop:A}, $h_1^{\inv}+h_2^{\inv}=h_2 + h_1$ is Lorentzian, a contradiction. Indeed, $g_2^{\inv}$ as a quadratic form is given by the matrix
    \begin{equation}
        A_3 = \begin{bmatrix}
            0 & 3.9 & 1 & 1 & 1 \\
            3.9 & 0 & 1 & 1 & 1 \\
            1 & 1 & 0 & 1 & 1 \\
            1 & 1 & 1 & 0 & 1 \\
            1 & 1 & 1 & 1 & 0
        \end{bmatrix}
    \end{equation}
    whose determinant is negative. Hence, $g_2^{\inv}$ is not Lorentzian.
\end{ex}

\subsection{Lorentzian proper position in context}
\subsubsection{Pullback of intersection product}
 Let $X$ be a $d$-dimensional irreducible projective variety over an algebraically closed field $\mathbb{F}$. Let $H_1,...,H_n$ be a collection of nef divisors on $X$. Then the volume polynomial given by the $d$-fold intersection, i.e.,
    \begin{equation}\label{eqn:volume-poly}
        \vol_X(w_1,...,w_n) = (w_1H_1+\cdots +w_nH_n)^d
    \end{equation}
is a Lorentzian polynomial \cite[Theorem 4.6]{branden2020lorentzian}. Let $Y$ be another nef divisor on $X$. Consider the intersection product 
\begin{equation}\label{eqn:volume-poly-sub}
    \vol_Y(w_1,...,w_n) = (w_1H_1+\cdots +w_nH_n)^{d-1}\cdot Y.
\end{equation}
When $Y$ is the class of a subvariety, $\vol_Y$ is the $(d-1)$-fold intersection of the pullback of $w_1H_1+\cdots + w_nH_n$ on $Y$.
\begin{thm}\label{thm:proper-position-AG}
    Let $\vol_X$ be given by \eqref{eqn:volume-poly} and $\vol_Y$ be given by \eqref{eqn:volume-poly-sub}, then $\vol_Y\ll_L \vol_X$ if both are nonzero.
\end{thm}

\begin{proof}
The polynomial     \begin{equation}
        \begin{split}
            f(w_0,w_1,...,w_n) & = (w_0Y + w_1H_1 + \cdots + w_nH_n )^d \\
           & = \vol_X + w_0\vol_Y  + w^2_0(w_1H_1+\cdots +w_nH_n)\cdot Y^2 + ... + w_0^d Y^d \end{split}
    \end{equation}
    is Lorentzian. By \Cref{cor:proper-position-from-single-poly}, $\vol_X+w_0\vol_Y$ is Lorentzian.
\end{proof}

\subsubsection{Intrinsic volumes of convex bodies}
Let $K_1,...,K_n,C\subset\R^d$ be convex bodies. Let $\Vol(K)$ be the $d$-dimensional volume of the convex body $K$. The polynomial 
\begin{equation}
\begin{split}    \vol\colon\R^{n+1}_{\geq 0} \to \R,\quad w\mapsto \Vol(w_0 C + w_1K_1+\cdots w_nK_n).
\end{split}
\end{equation}
is Lorentzian by \cite[Theorem 4.1]{branden2020lorentzian}. As a special case, take $C$ to be the unit $d$-ball $B^d$. The Steiner formula gives the following expansion \cite[Section 4]{schneider2013convex}:
\begin{equation}
    \vol(w_0,w_1,...,w_n) = \sum_{k=0}^d c_{d-k}w_0^{d-k}\rmV_k(w_1,...,w_n),
\end{equation}
where $c_i$ is the $i$-dimensional volume of the unit $i$-ball. The polynomial $\rmV_k(w_1,...,w_n)$ is the $k$-th \textit{intrinsic volume} of the convex body $w_1K_1+\cdots +w_nK_n$. For instance, $\rmV_d(K)=\Vol(K)$ and $\rmV_{d-1}(K)$ is proportional to the surface area of $K$. By \Cref{cor:proper-position-from-single-poly}, we have the following.

\begin{thm}\label{thm:proper-position-convex-geometry}
The $k$-th intrinsic volume $\rmV_k(w_1,...,w_n)$ is Lorentzian for all $k=0,...,d$ and $\rmV_k\ll_L \rmV_{k+1}$ whenever they are nonzero.
\end{thm}

\subsubsection{Determinantal polynomials}
Let $V\subset\C^n$ be a complex linear subspace of dimension $d$. With the standard basis of $\C^n$, we represent $V$ by a $d$-by-$n$ matrix $A$ over $\C$ by choosing a basis of $V$. Let $A^*$ be the Hermitian conjugate of $A$. Let $Z$ be the $n$-by-$n$ diagonal matrix with diagonal entries the variables $z_1,...,z_n$. Then the polynomial
\begin{equation}\label{eqn:determinant-poly}
    f_A(z_1,...,z_n)=\det(AZA^*)
\end{equation}
is homogeneous stable with nonnegative coefficients \cite[Theorem 8.1]{choe2004homogeneous}, hence Lorentzian. 
\begin{prop}\label{prop:proper-position-determinant}
    Let $W\subset V$ are linear subspaces in $\C^n$ of dimension $d-1$ and $d$ represented by matrices $A_1$ and $A_2$, respectively. Then $f_{A_1}\ll_L f_{A_2}$.
\end{prop}

\begin{proof}
Choosing a different representing matrix only changes the polynomial \eqref{eqn:determinant-poly} by multiplying by a positive scalar. By \Cref{prop:proper-position-properties} (5), we only need to prove the statement for one choice of the matrices.

Let $A_1$ be any $d-1$-by-$n$ matrix representing $W$. Extending $A_1$ to a $d$-by-$n$ matrix $A_2$ representing $V$. Extend $A_2$ to a $d$-by-$(n+1)$ matrix $A_3$ as follows
\begin{equation}
    A_3 = \begin{pmatrix}
        A_2 & \be^T_1
    \end{pmatrix}
\end{equation}
where $\be_1=(1,0,...,0)\in \C^d$. Let $Z'$ be the $(n+1)$-by-$(n+1)$ diagonal matrix with diagonal entries $z_1,...,z_n,z_{n+1}$. Then 
\begin{equation}
    \det(A_3Z'A_3^*) = f_{A_2}+z_{n+1}f_{A_1}
\end{equation}
is stable. Hence, $f_{A_1}\ll f_{A_2}$. In particular, $f_{A_1}\ll_L f_{A_2}$.
\end{proof}

\subsection{Lorentzian polynomials supported on given matroids}\label{subset:tutte-group}

In this subsection, we consider Lorentzian polynomials $f$ supported on a given matroid $M$, i.e., $\supp f$ is the set of bases of $M$. A more comprehensive study on this topic is \cite{bakerhuh}, while our focus is how $M$ affects $\rmL[f\ll_L]$, and $\rmL[\ll_Lf]$.

Put
\begin{equation}
    \rmL_M:=\big\{f \text{ Lorentzian }\mid f\text{ is supported on }M\big\}.
\end{equation}
A crucial tool for understanding $\rmL_M$ is the \textit{Tutte group} of a matroid, studied by Dress and Wenzel \cite{dress1989geometric,wenzel1989group}. For the definition of the Tutte group, we refer the reader to the original papers. What is relevant to us is that the Tutte group of a matroid $M$ reflects the following two types of information about $M$: 
\begin{itemize} 
    \item The `number' of basis valuations on $M$;
    \item the space $\rmL_M$ of Lorentzian polynomials supported on $M$. 
\end{itemize}

All matroids have basis valuations induced by a weight vector: for each $\bw\in \R^n$, the assignment  
\begin{equation}
    B\mapsto \bw\cdot \be_B
\end{equation}
is a valuation, where 
\begin{equation}\label{eqn:indicator-vec}
    \be_B = \sum_{i\in B}\be_i
\end{equation}
and $\bw\cdot\be_B$ is the usual dot product. A matroid is called \textit{rigid} if all of its basis valuations arise this way \cite[Definition 2.2]{dress1992valuated}. In other words, if $M$ is rigid, then any tropical linear space whose underlying matroid is $M$ is a translation of $\Trop M$. By \cite[Corollary 5.8]{dress1992valuated}, if the \textit{inner Tutte group} of $M$, a quotient group of the Tutte group, is torsion, then $M$ is rigid. In particular, binary matroids and finite projective spaces are rigid \cite[Proposition 5.9,5.10]{dress1992valuated}. This will be useful in \Cref{sec:incidence-problems} when we study incidence problems of tropical linear spaces.

The relation of the Tutte group and the space of stable polynomials is implicit in \cite{branden2007polynomials} and is studied in \cite{branden2010half}. We extend the results therein to Lorentzian polynomials. Let $A\in\nk{n}{d-2}$ and $i,j,k,l$ be distinct elements not in $A$. The set $\{Aij,Ajk,Akl,Ail\}$ is a \textit{degenerate quadrangle} of the matroid $M$ if $\{Aij,Ajk,Akl,Ail\}\subset \cB(M)$ and at least one of $Aik$ and $Ajl$ is not a basis.

\begin{lemma}\label{lem:degenerate-quadrangle}
    Let $f=\sum_{\cB(M)}a_Bw^B$ be Lorentzian. If $\{Aij,Ajk,Akl,Ail\}$ is a degenerate quadrangle, then 
    \begin{equation}\label{eqn:degenerate-quadrangle}
        a_{Aij}a_{Akl} = a_{Ajk}a_{Ail}.
    \end{equation}
\end{lemma}

\begin{proof}
    The statement is true for stable polynomials \cite[Lemma 6.1]{branden2007polynomials}. If $\{Aij,Ajk,Akl,Ail\}$ is a degenerate quadrangle of $M$, then $\{ij,jk,kl,il\}$ is a degenerate quadrangle of the contraction $M/A$. The conclusion follows from the stability of $\partial^Af$.
\end{proof}

Let $\R^{\cB(M)}$ be the obvious coordinate subspace of $\R^{\nk{n}{d}}$. Let $V_M$ be the subspace of $\R^{\cB(M)}$ cut out by the coordinate-wise logarithm of \eqref{eqn:degenerate-quadrangle}: $\bc\in V_M$ if and only if
\begin{equation}
    \bc(Aij) + \bc(Akl) = \bc(Ajk) + \bc(Ail)
\end{equation}
for each degenerate quadrangle $\{Aij,Ajk,Akl,Ail\}$. There is an obvious subspace of $U_M\subset V_M$ where 
\begin{equation}
    U_M = \bigg\{\bu\in \R^{\cB(M)}\mid \bu(B) = \sum_{i\in B} \ba(i)\text{ for some }\ba\in\R^n \bigg\}.
\end{equation}

Let $\log \rmL_M$ be the logarithmic image of $\rmL_M$ in $\R^{\cB(M)}$. We have 
\begin{equation}
   U_M\subset  \log \rmL_M \subset V_M
\end{equation}
The first containment follows from \Cref{lem:degenerate-quadrangle}; the second containment is because $\rmL_M$ contains all the re-scaling of the basis generating polynomial $f_M$. \cite[Theorem 3.3]{branden2010half} gives the following formula: if $M$ has $c$ connected components and $R$ is the free rank of the inner Tutte group of $M$, then
\begin{equation}
        \dim(V_M) = n - z + 1 + R = \dim(U_M) + R.
    \end{equation}
This implies the following description of $\rmL_M$ when the inner Tutte group of $M$ is torsion.

\begin{prop}\label{prop:space-of-lorentzian-trivial-inner-tutte}
    Suppose the inner Tutte group of $M$ is torsion. Then any Lorentzian polynomial $f$ supported on $M$ is a re-scaling of $f_M$. Moreover, $\rmL_M$ has the following two descriptions.
    \begin{enumerate}
        \item $\rmL_M$ is the image of $U_M$ under the coordinate-wise exponentiation map 
        \begin{equation}
            \exp: \R^{\cB(M)} \to \R^{\cB(M)}_{\geq 0},\quad \bv \mapsto \left(\exp(\bv(B))\right)_B
        \end{equation}
        \item $\rmL_M\subset \R_{\geq 0}^{\cB(M)}$ is cut out by the quadratic equations $\bc(Aij)\bc(Akl) = \bc(Ajk)\bc(Ail)$ for all degenerate quadrangles $\{Aij,Ajk,Akl,Ail\}$.
    \end{enumerate}
\end{prop}

\begin{cor}\label{cor:space-of-lorentzian-trivial-inner-tutte-2}
    If the inner Tutte group of $M$ is torsion, then for any Lorentzian polynomial $f$ supported on $M$, $\rmL[\ll_Lf]\cong \rmL[\ll_Lf_M]$ and $\rmL[f\ll_L]\cong \rmL[f_M\ll_L]$ by scaling the variables.
\end{cor}

\begin{proof}
    By \cref{prop:space-of-lorentzian-trivial-inner-tutte}, $f=f_M(a_1w_1,...,a_nw_n)$ for some $a_i>0,i=1,...,n$. Consider the map
    \begin{equation}
        \rmL[\ll_Lf_M] \to \rmL[\ll_Lf], \quad g\mapsto g(a_1w_1,...,a_nw_n).
    \end{equation}
    This map is well-defined, since the map $f_M+w_{n+1}g\mapsto f_M(a_1w_1,...,a_nw_n)+w_{n+1}g(a_1w_1,...,a_nw_n)$ preserves the Lorentzian property. It has an inverse $g\mapsto g(a_1^{-1}w_1,...,a_n^{-1}w_n)$, so $\rmL[\ll_Lf_M]\cong \rmL[\ll_Lf]$. The statement for $\rmL[f\ll_L]$ has the same proof.
\end{proof}

\section{Structural theorems on the relative Dressian $\Dr^1(\mu)$}\label{sec:dressian-structure}

In this section we study the relative Dressians and their discrete counterparts:
\begin{equation}
    \begin{split}
        \Dr^1(\mu) & :=\big\{\text{ codimension-1 tropical linear subspaces of }\Trop\mu\big\}, \\
        \Dr^1(M) & :=\big\{\text{ codimension-1 tropical linear subspaces of }\Trop M\big\}, \\
        \Qt^1(M) & : =\{\text{ elementary matroid quotients of }M\}, \\
        \widehat{\Qt}^1(M) &:=\Qt^1(M)\cup\{M\}.
    \end{split}
\end{equation}
We prove structural results including
\begin{itemize}
    \item \Cref{thm:dressian-cut-out-by-three-term}: $\Dr^1(\mu)$ is cut out by \textit{three-term} incidence relations;
    \item \Cref{thm:dressian-is-order-complex}: $\Dr^1(M)$ is the order complex of $\widehat{\Qt}^1(M)$;
    \item \Cref{thm:spanned-by-join-irreducibles}: $\Dr^1(M)$ is a tropical polytope generated by the trivial valuations on the join-irreducibles of $\widehat{\Qt}^1(M)$;
\end{itemize}
With these results, we give a geometric characterization of adjoints of ordinary matroids in \Cref{prop:characterization-of-adjoint}, which become foundational in our generalization of adjoints to valuated matroids in \Cref{sec:valuated-adjoint}.

Now consider
\begin{equation}
    \begin{split}
        \Dr(\mu) & :=\big\{\text{ all tropical linear subspaces of }\Trop\mu\big\}, \\
        \Dr(M) & :=\big\{\text{ all tropical linear subspaces of }\Trop M\big\},\\
        \Qt(M) & : =\{\text{ all matroid quotients of }M\}.
    \end{split}
\end{equation}
These sets are partially ordered by the quotient relation: $\Dr(\mu)$ and $\Dr(M)$ are ordered by valuated matroid quotient and $\Qt(M)$ is ordered by ordinary matroid quotient. There is a commutative diagram. 
\begin{center} \vspace{5pt}
    \begin{tikzcd}
    \Dr^1(\mu) \ar[d,hookrightarrow] \ar[rr, "\text{underlying}","\text{matroid}"'] & &  \Qt^1(M)   \ar[rr,"\text{trivial}", "\text{valuation}"'] \ar[d,hookrightarrow] & & \Dr^1(M) \ar[d,hookrightarrow] \\
    \Dr(\mu) \ar[rr, "\text{underlying}","\text{matroid}"'] & & \Qt(M) \ar[rr,"\text{trivial}", "\text{valuation}"'] & & \Dr(M) 
\end{tikzcd}
\vspace{5pt} \end{center} 
The horizontal maps in the second row are order-preserving. 
This diagram has a guiding role for the remainder of the paper. This section concerns the relation between $\Dr^1(\mu)$ and $\Qt^1(M)$: some essential information of $\Dr^1(\mu)$ is controlled by the lattice structure of $\widehat{\Qt}^1(M)$. \Cref{sec:incidence-problems} concerns the inclusions $\Dr^1(\mu)\into \Dr(\mu)$ and $\Qt^1(M)\into \Qt(M)$: some incidence properties of $\Trop\mu$ can be extracted from the structure of $\Dr^1(\mu)$.

\subsection{Defining equations of $\Dr^1(\mu)$}
The main goal of this subsection is to prove \Cref{thm:dressian-cut-out-by-three-term}, which reduces the number of defining equations of $\Dr^1(\mu)$ to the following three distinguished classes of incidence relations. Each class is labeled by at most three $(d-1)$-sets. Any $A\in \nk{n}{d-1}$ spans a (possibly degenerate) hyperplane in $M$, and the three classes are divided according to the relative position of the hyperplanes.

\begin{itemize}
    \item \textbf{Degenerate hyperplanes:} If $A\in \nk{n}{d-1}$ is dependent in $M$, then $\theta(A)=\infty$. This follows from the following computation. Let $C\subset A$ be a circuit of $M$. Let $i\in C$ be any element and $B$ a basis of $M$ containing $C-i$. Consider the incidence relation indexed by $(A-i, B+i)$:
    \begin{equation}
        \min_{j\in B\backslash A}\left\{ \theta(A) + \mu(B),\quad  \theta(A-i+j)+\mu(B+i-j) \right\}.
    \end{equation}
    Since $C\subset B+i-j$ for all $j\in B\backslash A$, $\mu(B+i-j)=\infty$. Then $\theta(A)=\infty$ because $\mu(B)<\infty$.
    \item \textbf{Parallel hyperplanes:} If $A,B$ are independent, $|A\backslash B|=1$, and $\cl_M(A)=\cl_M(B)$, then we claim that
    \begin{equation}\label{eqn:paralle-hyperplane-relation}
        \theta(A)+\mu(Bk)=\theta(B)+\mu(Ak)
    \end{equation}
for any $k$ such that $Ak$ (thus $Bk$) is a basis of $M$. Moreover, the difference $\mu(Ak)-\mu(Bk)$ is independent of the choice of $k$. This can be seen from the following: suppose $A=Di$ and $B=Dj$. Take any $k\notin \cl_M(A)$. Then $Dik$ is a basis of $M$. The incidence relation indexed by $(D,Dijk)$ says $\min\{\theta(A)+\mu(Bk),\theta(B)+\mu(Ak)\}$ vanishes tropically. If for some $l\neq k$, $Dil$ is another basis of $M$, then we have the Pl\"{u}cker relation of $\mu$ indexed by $(Di,Djkl)$
    \begin{equation}
        \min\big\{\quad \mu(Dil)+\mu(Djk),\quad  \mu(Dik)+\mu(Djl),\quad \mu(Dij)+\mu(Dkl)\quad 
        \big\}.
    \end{equation}
    Since $\mu(Dij)=\infty$, we get $\mu(Dil)+\mu(Djk) = \mu(Dik)+\mu(Djl)$, so $\mu(Dil)-\mu(Djl) = \mu(Dik)-\mu(Djk)$.
    \item \textbf{Concurrent hyperplanes:} Suppose that $D\in\nk{n}{d-2}$, $i,j,k\notin D$, that $Di,Dj,Dk$ are all independent. Then 
    \begin{equation}
        \min\big\{\quad \theta(Di)+\mu(Djk),\quad \theta(Dj)+\mu(Dik),\quad \theta(Dk)+\mu(Dij)\quad\big\}
    \end{equation}
    vanishes tropically. This follows from the incidence relation indexed by $(D,Dijk)$. The hyperplanes $\cl_M(Di), \cl_M(Dj),\cl_M(Dk)$ are called concurrent because they meet at a coline $\cl_M(D)$. Note that if $\cl_M(Di)=\cl_M(Dj)$, then $\mu(Dij)=\infty$, in which case this relation degenerates to the parallel hyperplane relations.
\end{itemize}

The above incidence relations involve at most three terms. Therefore, we call them the \textit{three-term incidence relations}.

\begin{lemma}\label{lem:support-lemma}
    Let $\theta\colon \nk{n}{d-1}\to \Rbar$ be any function that satisfy the incidence relations \eqref{eqn:incidence-plucker} defining $\Dr^1(\mu)$, then $\supp(\theta)$ is the set of bases of a matroid.
\end{lemma}

\begin{proof}
    Let $A,B$ be in $\supp(\theta)$. Then $A,B$ are independent in $M$. We may assume $|A\backslash B|\geq 2$, for otherwise there is nothing to prove.
    \begin{itemize}
        \item Case I: $\cl_M(A)=\cl_M(B)$. For any $i\in A\backslash B$, there is $j\in B\backslash A$ such that $A-i+j$ is a basis of $\cl_M(A)$. Since any two bases can be obtained by a sequence of basis exchanges, and by the parallel hyperplane relations, two bases differing by one element must be both in $\supp(\theta)$ or both not in $\supp(\theta)$, we conclude that $A-i+j\in \supp(\theta)$. 
        \item Case II: $\cl_M(A)\neq \cl_M(B)$. Then for some $j\in B\backslash A$, $Aj$ is a basis of $M$. Take $i\in A\backslash B$ and write $A=Di$
        \begin{itemize}
            \item If there is some $k\in B\backslash A$ such that $Djk$ is a basis of $M$, then the incidence relation indexed by $(D,Dijk)$ is supported at both $A$ and $Dk$, possibly at $Dj$, so at least one of $A-i+k$ and $A-i+j$ is in $\supp(\theta)$.
            \item If for all $k\in B\backslash A$ and $k\neq j$, $Djk$ is not a basis of $M$, then we deduce that $k\in \cl_M(Dj)$ for all $k\in B\backslash A$. Since $A\cap B\subset D$, we know that $\cl_M(B)\subset \cl_M(Dj)$. Since they have the same rank, $\cl_M(B)=\cl_M(Dj)$. By the previous case, $Dj\in\supp(\theta)$.         \end{itemize}
        
    \end{itemize}    Since for any $i\in A\backslash B$, there is some $j\in B\backslash A$ such that $A-i+j$ is in the support of $\theta$, we conclude that $\supp(\theta)$ is the set of bases of some matroid.
\end{proof}

\begin{thm}\label{thm:linear-imlies-quadratic}
 Let $\mu$ be a valuated matroid. Then the incidence relations suffice to cut out $\Dr^1(\mu)$.
\end{thm}

\begin{proof}
    Let $\theta\colon \nk{n}{d-1}\to\Rbar$ satisfy all the incidence relations. By \cite[Theorem 5.2.25]{murota2010matrices} and \Cref{lem:support-lemma}, it suffices to show that $\theta$ satisfied all the three-term Pl\"{u}cker relations. Namely, for any $S$ with $|S|=d-3$ and distinct indices $i,j,k,l\notin S$, the following vanishes tropically,    \begin{equation}\label{eq:plucker-to-prove}
        \min\big\{\;\theta(Sij)+\theta(Skl),\quad \theta(Sik)+\theta(Sjl),\quad \theta(Sil)+\theta(Sjk)\;\big\}.
    \end{equation}
If the minimum above is $\infty$, there is nothing to show, so WLOG we may assume $\theta(Sik)+\theta(Sjl)<\infty$. There are two major cases.
    \begin{itemize}
        \item Case I: $\rk_M(Sijkl)=d-1$. In this case, $\mu(Sijk)=...=\mu(Sjkl)=\infty$. Choose $m\in [n]$ such that $\rk_M(Sijklm)=d$. Then $Sikm$ and $Sjlm$ are both bases of $M$, meaning $\mu(Sikm)<\infty$ and $\mu(Sjlm)<\infty$. The incidence relations indexed by $(Si,Sijkm),(Si,Siklm),(Sj,Sjklm),$ and $(Sl,Sjklm)$ give us
        \begin{equation}\label{eq:useful-incidence}
            \begin{split}
                \theta(Sik)+\mu(Sijm)=\theta(Sij)+\mu(Sikm),\\
                 \theta(Sik)+\mu(Silm)=\theta(Sil)+\mu(Sikm), \\
                \theta(Sjk) + \mu(Sjlm)=\theta(Sjl) + \mu(Sjkm),  \\
                \theta(Skl) + \mu(Sjlm)=\theta(Sjl) + \mu(Sklm).
            \end{split}
        \end{equation}
        On the other hand, the Pl\"{u}cker relation of $\mu$ indexed by $(Smi,Smjkl)$ implies the tropical vanishing of
        \begin{equation}
            \min\big\{\; \mu(Sijm) + \mu(Sklm), \quad \mu(Sikm) + \mu(Sjlm),\quad \mu(Silm) + \mu(Sjkm) \;\big\}.
        \end{equation}
        Since $\theta(Sik)+\theta(Sjl)<\infty$, the following vanishes tropically
        \begin{equation}
            \min\begin{Bmatrix}
                \mu(Sijm) + \mu(Sklm)+\theta(Sik)+\theta(Sjl), \\
                \mu(Sikm) + \mu(Sjlm)+\theta(Sik)+\theta(Sjl),\\
                \mu(Silm) + \mu(Sjkm)+\theta(Sik)+\theta(Sjl)
            \end{Bmatrix}.
        \end{equation}
        
        Applying \eqref{eq:useful-incidence}, we get the tropical vanishing of
        \begin{equation}
            \min\begin{Bmatrix}
                \mu(Sikm) + \mu(Sjlm)+\theta(Sij)+\theta(Skl),\\
                \mu(Sikm) + \mu(Sjlm)+\theta(Sik)+\theta(Sjl),\\ \mu(Sikm) + \mu(Sjlm)+\theta(Sil)+\theta(Sjk)
                \end{Bmatrix}. 
        \end{equation}
        Since $\mu(Sikm) + \mu(Sjlm)<\infty$, subtracting $\mu(Sikm) + \mu(Sjlm)$ from the above gives \eqref{eq:plucker-to-prove}.
        \item Case II: $\rk_M(Sijkl)=d$. Since both $Sik$ and $Sjl$ can be extended to a basis of $M$, at least two of the four values $\mu(Sijk),\mu(Sijl),\mu(Sikl),\mu(Sjkl)$ are finite. Consider the incidence relations indexed by $(i,ijkl),(j,ijkl),(k,ijkl)$, and $(l,ijkl)$: the tropical vanishing of the following
        \begin{equation}\label{eqn:incidence-relations}
        \begin{split}
        \min\big\{\;\theta(Sij)+\mu(Sikl),\quad\theta(Sik)+\mu(Sijl),\quad\theta(Sil)+\mu(Sijk)\;\big\}, \\
        \min\big\{\;\theta(Sij)+\mu(Sjkl),\quad\theta(Sjk)+\mu(Sijl),\quad\theta(Sjl)+\mu(Sijk)\;\big\}, \\
        \min\big\{\;\theta(Sik)+\mu(Sjkl),\quad\theta(Sjk)+\mu(Sikl),\quad\theta(Skl)+\mu(Sijk)\;\big\},\\
        \min\big\{\;\theta(Sil)+\mu(Sjkl),\quad\theta(Sjl)+\mu(Sikl),\quad\theta(Skl)+\mu(Sijl)\;\big\}.
        \end{split}
    \end{equation}
    \begin{itemize}
        \item Case II(a): all the $\mu$-values above are finite. This case is essentially proved in \cite[Theorem 3.10]{joswig2023generalized}.
        \item Case II(b): exactly one of the $\mu$-values is infinite. We may assume $\mu(Sijk)=\infty$. Let $\eta\colon\nk{n}{d-1}\to \R$ be the affine function given by
        \begin{equation}
            A\mapsto \be_A\cdot \big(\mu(Sjkl)\be_i+\mu(Sikl)\be_j+\mu(Sijl)\be_k\big)-\mu(Sjkl)-\mu(Sikl)-\mu(Sijl). 
        \end{equation}
        Let $\hat{\mu}=\mu + \eta$ and $\hat{\theta}=\theta+\eta$. Note that $\hat{\mu}(Sijl)=\hat{\mu}(Sikl)=\hat{\mu}(Sjkl)=0$. Since $\mu\onto \theta$ if and only if $\hat{\mu}\onto \hat{\theta}$, we may thus assume $\mu(Sijl)=\mu(Sikl)=\mu(Sjkl)=0$. The incidence relations \eqref{eqn:incidence-relations} become the tropical vanishing of
        \begin{equation}
        \begin{split}
            \min\{\theta(Sij),\; \theta(Sik)\}, \quad & \min\{\theta(Sij),\;\theta(Sjk)\},\\
            \min\{\theta(Sik), \;\theta(Sjk)\}, \quad & \min\{\theta(Sil),\;\theta(Sjl),\;\theta(Skl)\}.
        \end{split}
    \end{equation}
    It is clear that they imply \eqref{eq:plucker-to-prove}.

        \item Case II(c): two of the $\mu$-values are infinite. By relabeling, we may assume $\mu(Sikl)=\mu(Sjkl)=\infty$. This means $Skl$ cannot be extended to a basis of $Sijkl$, so $Skl$ is dependent in $M$ and we must have $\theta(Skl)=\infty$. The incidence relations \eqref{eqn:incidence-relations} become the tropical vanishing of
        \begin{equation}
            \min\{\theta(Sik)+\mu(Sijl),\;\theta(Sil)+\mu(Sijk)\}, \quad
        \min\{\theta(Sjk)+\mu(Sijl),\;\theta(Sjl)+\mu(Sijk)\}
        \end{equation}
        It is easy to check that \eqref{eq:plucker-to-prove} is satisfied.
    \end{itemize}
    \end{itemize}
\end{proof}

Now we improve \Cref{thm:linear-imlies-quadratic} to \Cref{thm:dressian-cut-out-by-three-term}.

\begin{thm}\label{thm:dressian-cut-out-by-three-term}
    Let $\mu$ be a valuated matroid. Then $\Dr^1(\mu)$ is cut out by the three-term incidence relations.
\end{thm}

\begin{proof}
We show that any incidence relation is implied by the three-term incidence relations. Take the incidence relation indexed by $(A,D)$ for some $A\in \nk{n}{d-2}$ and $D\in\nk{n}{d+1}$. WLOG, we assume the relation is the tropical vanishing of 
    \begin{equation}\label{eq:arbitrary-incidence}
        \min\{\theta(A1)+\mu(D-1),...,\theta(Am)+\mu(D-m)\},
    \end{equation}
    where all the $\mu$ values are finite. We may also assume that all the $\theta$ values are finite; otherwise we can simply remove those that are infinity from the relation. Suppose for a contradiction that all the three-term incidence relations are satisfied, but \eqref{eq:arbitrary-incidence} does not vanish tropically: say, $\theta(A1)+\mu(D-1)$ is the unique minimum. Since $A1$ is independent and $D-1$ is a basis of $M$, there is some $k_1\in D-1$ such that $A1k_1$ is a basis of $M$. Using the Pl\"{u}cker relation for $\mu$ indexed by $(Ak_1,D)$, we get some $k_2\in D-1$ such that 
    \begin{equation}
        \mu(A1k_1)+\mu(D-1) \geq \mu(Ak_1k_2) + \mu(D-k_2) \Rightarrow \mu(A1k_1) - \mu(Ak_1k_2) \geq \mu(D-k_2) - \mu(D-1)
    \end{equation}
    This implies
    \begin{equation}
        \theta(Ak_2)-\theta(A1) + \mu(A1k_1) - \mu(Ak_1k_2)  \geq \theta(Ak_2)-\theta(A1) + \mu(D-k_2) - \mu(D-1) > 0.
    \end{equation}
    The subtractions make sense, since the values subtracted are all finite. By assumption, the three-term incidence relation indexed by $(A,A1k_1k_2)$ is satisfied, so 
    \begin{equation}
        \theta(Ak_2) + \mu(A1k_1) > \theta(A1)+ \mu(Ak_1k_2)  = \theta(Ak_1)+ \mu(A1k_2). 
    \end{equation}
   Thus $\theta(Ak_1)+\mu(A1k_2)<\infty$. We then use the Pl\"{u}cker relation  indexed by $(Ak_2,D)$ and so forth. Repeat this process. After obtaining $k_i$, using the Pl\"{u}cker relation indexed by $(Ak_i,D)$, we get $k_{i+1}$ such that
    \begin{equation}
        \theta(Ak_{i+1}) + \mu(A1k_i) > \theta(A1)+ \mu(Ak_ik_{i+1})  = \theta(Ak_i)+ \mu(A1k_{i+1}), 
    \end{equation}
    where $\theta(Ak_i)+ \mu(A1k_{i+1})<\infty$.
    Since there are only finitely many indices, for some $s$ and $l\leq s$ we have $k_l=k_{s+1}$. Now we have
    \begin{equation}
        \sum_{i=l}^{s} (\theta(Ak_{i+1}) + \mu(A1k_i)) > \sum_{i=l}^{s} (\theta(Ak_i) + \mu(A1k_{i+1})),
    \end{equation}
    but these two quantities are finite and equal. We get a contradiction.
\end{proof}

Imposing incidence relations only adds tropical linear equations. Hence, \Cref{thm:dressian-cut-out-by-three-term} implies the following stronger result. Let $\omega_1,...,\omega_s$ and $\nu_1,...,\nu_t$ be valuated matroids on $[n]$. Put 
\begin{equation}\label{eqn:incidence-variety}
    \Dr^1(\mu;\omega_1,...,\omega_s;\nu_1,...,\nu_t) = \left\{\Trop\theta\in \Dr^1(\mu) \;\Big|\; \bigcup_{i=1}^{t}\Trop\nu_i\subset \Trop\theta \subset \bigcap_{j=1}^s\Trop\omega_j\right\}.
\end{equation}

\begin{cor}\label{cor:first-incidence-dressian-linear}
    The space $\Dr^1(\mu;\omega_1,...,\omega_s;\nu_1,...,\nu_t)$ as in \eqref{eqn:incidence-variety} is a linear tropical prevariety.
\end{cor}
\Cref{cor:union-basis-families} can be likewise strengthened. We note the following sequence of this strengthening.
\begin{cor}\label{cor:exists-freest}
    For any matroids $L_1,...,L_k$ and $N_1,...,N_l$ on $[n]$, if the set 
    \begin{equation}
    \begin{aligned}
        \Qt^1(M;L_1,...,L_k;N_1,...,N_l):=\big\{Q\in \Qt^1(M)\mid & L_i\onto Q \text{ for all }i=1,...,k\text{ and } \\
        & Q\onto N_j\text{ for all }j=1,...,l \big \}.
    \end{aligned}
\end{equation}
is nonempty, then it has a freest element.
\end{cor}

\subsection{The simplified space $\Dr_s^1(\mu)$}

We describe a tropical prevariety $\Dr_s^1(\mu)$ inside $\PT{\cLL^1(M)}$ that carries the same information as $\Dr^1(\mu)$ does. This simplified space allows cleaner generalization of matroid adjoints to valuated matroids, while the space $\Dr^1(\mu)$ is more convenient in computations.

By the degenerate hyperplane relations, $\Dr^1(\mu)$ is contained in the coordinate subspace 
\begin{equation}
    \left\{[\bx] \in \PT{\nk{n}{d-1}} \mid \bx(A) = \infty\text{ if }\rk_M(A)<d-1 \right\}.
\end{equation}
Now the parallel hyperplane relations say if $A_1,A_2\in \nk{n}{d-1}$ span the same hyperplane, then for any $[\theta]\in\Dr^1(\mu)$, the value of $\theta(A_1)$ determines the value of $\theta(A_2)$. Therefore, $\theta$ is uniquely determined once we know its value on \textit{one} basis for each hyperplane of $M$. Suppose $\{D_H\}_{H\in \cLL^1(M)}$ is a choice of bases of hyperplanes of $M$, one for each hyperplane. Put
\begin{equation}
    \Dr_s^1(\mu) : =\left\{[\theta']\in \PT{\cLL^1(M)}\mid \text{ for some }[\theta]\in\Dr^1(\mu), \theta'(H)=\theta(D_H)\text{ for all }H\in \cLL^1(M)\right\}.
\end{equation}
Then $\Dr_s^1(\mu)$ and $\Dr^1(\mu)$ are isomorphic as polyhedral complexes. Moreover, $\Dr_s^1(\mu)$ is a prevariety of the same dimension as $\Dr^1(\mu)$ cut out by equations labeled by \textit{concurrent triples} defined below. 

\begin{definition}\label{def:concurrent-triples}
    Given three distinct hyperplanes $H_1,H_2,H_3\in \cLL^1(M)$, we say $\{H_1,H_2,H_3\}$ is a \textit{concurrent triple} if they meet at a coline, i.e., a corank-2 flat.
\end{definition}
The coefficients of the equations can be computed from the formulas given by the parallel hyperplane relations \eqref{eqn:paralle-hyperplane-relation}. Of course, $\Dr_s^1(\mu)$ depends on the choice of hyperplane bases $\{D_H\}$. However, the specific choice will not matter to us.

\subsection{The relation between $\Dr^1(M)$ and $\Qt^1(M)$}\label{subsec:dressian-of-tropM}

In this subsection we focus on the special case where $\mu$ is the trivial valuation on $M$. The central results are \Cref{thm:dressian-is-order-complex}, which describes $\Dr_s^1(\mu)$ as an order complex, and \Cref{prop:characterization-of-adjoint}, a new characterization of matroid adjoints in terms of the geometry of $\Dr^1(\mu)$. 

By \Cref{thm:dressian-cut-out-by-three-term}, $\Dr_s^1(M)$ is cut out by three-term incidence relations labeled by all the concurrent triples: for each concurrent triple $\{H_1,H_2,H_3\}$,
\begin{equation}\label{eqn:defining-equations-dressian-tropM}
    \min\{\;\theta(H_1),\;\theta(H_2),\; \theta(H_3)\;\}
\end{equation}
vanishes tropically. Consider 
\begin{equation}
    \Dr_s^1(M)^\circ = \Dr_s^1(M)\cap T_{\cLL^1(M)},
\end{equation}
which is a rational polyhedral fan. For each collection of hyperplanes $\cHH\subset \cLL^1(M)$, put 
\begin{equation}
    \be_{\cHH}=\sum_{H\in \cHH} \be_{\{H\}}\in \R^{\cLL^1(M)}.
\end{equation}
We warn the reader that $\be_{\{H\}}$ is the standard basis vector in the $H$-direction in $\R^{\cLL^1(M)}$, instead of the vector $\be_H$ as in \eqref{eqn:indicator-vec}. Given $\cHH_1,...,\cHH_k\subset \cLL^1(M)$, put 
\begin{equation}
    \cone(\be_{\cHH_1},...,\be_{\cHH_k}) = \big\{a_1\be_{\cHH_1}+\cdots + a_k\be_{\cHH_k}\mid a_1,...,a_k\geq 0\big\}.
\end{equation}

\begin{thm}\label{thm:dressian-is-order-complex}
    The collection of cones $\cone(\be_{\cHH_1},...,\be_{\cHH_k})+\R\mathbf{1}$, where $\emptyset\subsetneq \cHH_1\subsetneq \cdots \subsetneq \cHH_k \subsetneq \cLL^1(M)$ runs over all chains of linear subclasses of $M$, forms a simplicial fan. The support of this fan equals $\Dr_s^1(M)^\circ$.
\end{thm}

\begin{proof}
    We first show that $\cone(\be_{\cHH_1},...,\be_{\cHH_k})\subset \Dr_s^1(M)^\circ$. Take an arbitrary element $\bw=a_1\be_{\cHH_1}+\cdots +a_k\be_{\cHH_k}$ from this cone. Let $\{H_1,H_2,H_3\}$ be any concurrent triple. By the definition of linear subclasses, up to relabeling $H_1,H_2$ and $H_3$, the following cases exhaust the possibilities regarding the memberships of $H_1,H_2$ and $H_3$ in the $\cHH_i$'s:
    \begin{itemize}
        \item None of $H_1,H_2$ or $H_3$ is in $\cHH_k$;
        \item $H_1\in \cHH_i$ for some $i$ and $H_2,H_3\notin \cHH_k$;
        \item $H_1\in \cHH_i$, $H_2,H_3\notin \cHH_{j-1}$, and $H_2,H_3\in \cHH_j$ for $1\leq i\leq j\leq k$.
    \end{itemize}
    It is easy to check that in any case, $\bw$ is contained in the tropical hyperplane labeled by $\{H_1,H_2,H_3\}$.

    On the other hand, pick an arbitrary $\bw\in \Dr_s^1(M)^\circ$. We may assume that all coordinates of $\bw$ are nonnegative and at least one coordinate is positive. Then there is a unique way to write $\bw$ as 
    \begin{equation}
        \bw = a_1 \be_{\cG_1} + \cdots + a_k \be_{\cG_k}
    \end{equation}
    for some subsets $\emptyset \subsetneq \cG_1 \subsetneq \cdots \subsetneq \cG_k=\cLL^1(M)$ and $a_1,...,a_k>0$. Let $\{H_1,H_2,H_3\}$ be any concurrent triple. Suppose $\cG_i$ contains $H_1$ and $H_2$. We need to show $H_3\in \cG_i$.
    \begin{itemize}
        \item If $\bw(H_1)>\bw(H_2)$, then $\bw(H_3)=\bw(H_2)$, so $H_3\in\cG_i$;
        \item If $\bw(H_1)=\bw(H_2)$, then $\bw(H_3)\geq \bw(H_1)$, meaning that $H_3\in \cG_j$ for some $j\leq i$. Therefore, $H_3\in \cG_i$.
    \end{itemize}
    Hence, $\cG_i$ is a linear subclass. Moreover, we just showed that each element $\bw\in \Dr_s^1(M)^\circ$ lies in the relative interior of a unique cone defined above. Therefore, the collection of all those cones is a simplicial fan whose support equals $\Dr_s^1(M)^\circ$.
\end{proof}

By \Cref{thm:dressian-is-order-complex}, the dimension of $\Dr_s^1(M)$ is the length of the longest maximal chain of $\Qt^1(M)$. From this one sees that $\Dr_s^1(M)$ is not always pure-dimensional, as suggested by \Cref{ex:elementary-quotient-lattice}. It may have cones whose dimensions drop, as is the case for the V\'{a}mos matroid (\Cref{ex:vamos}). We also have the following dimension estimate.

\begin{cor}\label{cor:dim-of-dr1}
    The polyhedral complex $\Dr_s^1(M)$ always has cones of dimension $d-1$.
\end{cor}
\begin{proof}
    There is an order-preserving embedding $\cLL(M)\into \widehat{\Qt}^1(M)$, where each flat $F$ is sent to the principal truncation of $M$ by $F$ \cite{white1986theory}. We show that this map sends maximal chains to maximal chains.
    
    Recall that for the principal truncation of $M$ by $F$, the corresponding linear subclass is the set all the hyperplane containing $F$, and the corresponding modular cut is the upper interval of $F$. Suppose $F,G\in \cLL(M)$ and $F$ covers $G$. If $\cF$ is any modular cut strictly containing the upper interval of $F$ and meanwhile contained in the upper interval of $G$, then $\cF$ contains a hyperplane $H\in \cHH_G\backslash \cHH_F$. Since $F \wedge H = G$, we have
    \begin{equation}
        \rk_M(F) + \rk_M(H) = \rk_M(F\wedge H) + \rk_M(F\vee H),
    \end{equation}
    so $G\in \cF$, meaning $\cF$ is the upper interval of $G$. This implies that the principal truncation by $F$ covers the principal truncation by $G$ in $\widehat{\Qt}^1(M)$, so the maximal chains in $\cLL(M)$ are sent to maximal chains of $\widehat{\Qt}^1(M)$. The conclusion then follows from \Cref{thm:dressian-is-order-complex}.
\end{proof}

\begin{ex}\label{ex:elementary-quotient-lattice}
    The maximal chains of $\Qt^1(M)$ (and thus the maximal chains of $\widehat{\Qt}^1(M)$) may not have the same length. This is already the case for $M=U_{3,6}$, as shown in \Cref{fig:elementary-quotient-lattice}.

    \begin{figure}[H]
    \centering
    \includegraphics[width=3.5in]{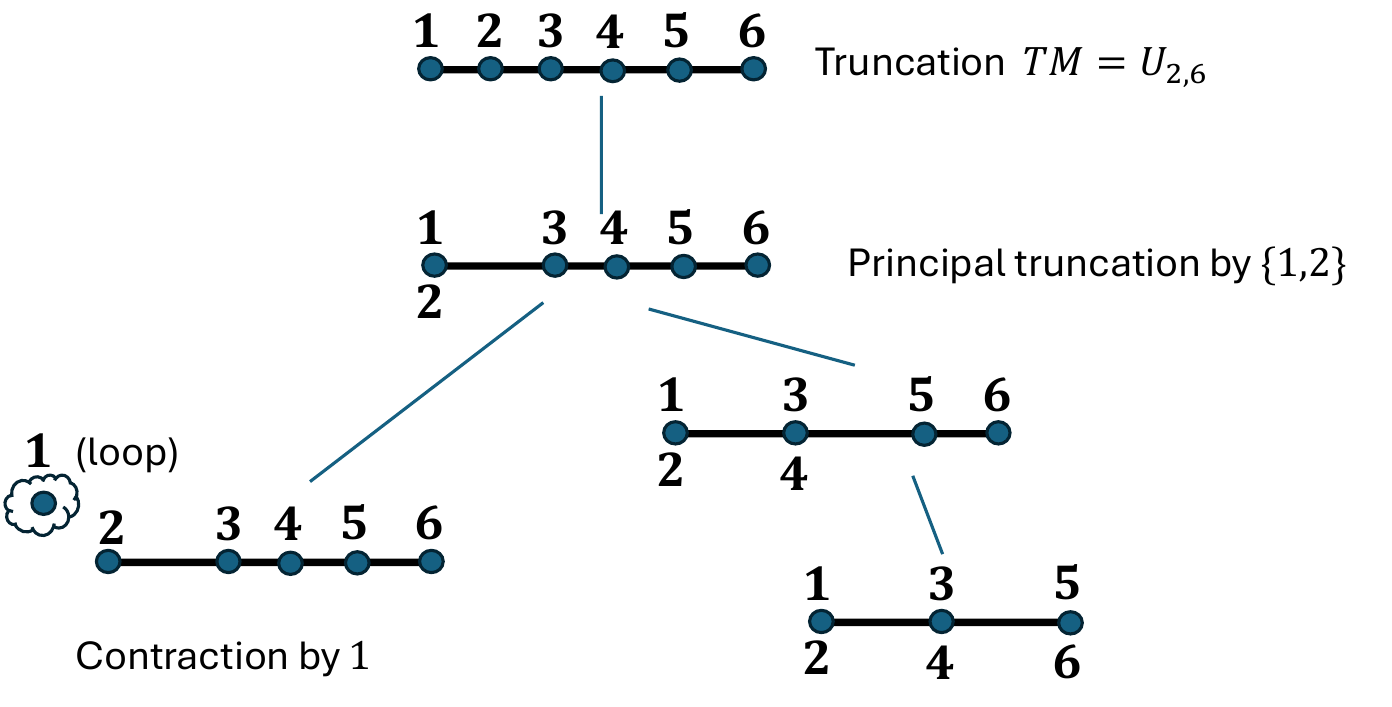}
    \caption{Two maximal chains of $\Qt^1(U_{3,6})$.}
    \label{fig:elementary-quotient-lattice}
\end{figure}
\end{ex}

The next statement provides a set of minimal generators of $\Dr_s^1(M)$. In particular, it shows directly that $\Dr_s^1(M)$ is a tropical polytope in the sense of \cite{gaubert2011minimal}. 

\begin{thm}\label{thm:spanned-by-join-irreducibles}
    Identify $\Qt^1(M)$ as a subset of $\Dr_s^1(M)$ consisting of the trivial valuations on the elementary quotients of $M$. Then $\Dr_s^1(M)$ is spanned tropically by the join-irreducible elements of $\Qt^1(M)$.
\end{thm}

\begin{proof}
    Let $\cJ$ be the subset of join-irreducible elements of $\Qt^1(M)$. It suffices to show that the tropical span of $\cJ$ contains $\Dr^1_s(M)^\circ$, because the tropical span of $\cJ$ is closed, and $\Dr^1_s(M)$ is the closure of $\Dr^1_s(M)^\circ$. The trivial valuation on the elementary quotient of $M$ given by the linear subclass $\cHH$, regarded as a function $\cLL^1(M)\to \Rbar$, is simply the indicator function $\delta_{\cHH^c}$ of $\cHH^c=\cLL^1(M)\backslash \cHH$, i.e.,
    \begin{equation}
        \delta_{\cHH^c}(H) = \begin{cases}
            0, & \text{if }H\notin \cHH \\
            \infty, & \text{if }H\in \cHH
        \end{cases}.
    \end{equation}
    
    Take any element $[\theta]\in \Dr^1_s(M)^\circ$. By \Cref{thm:dressian-is-order-complex}, up to adding a constant, every $[\theta]\in \Dr^1_s(M)^\circ$ takes the form
    \begin{equation}
        \theta = \sum_{i=1}^k a_i\be_{\cHH_i}
    \end{equation}
    where $a_i>0$ for all $i$ and $\emptyset\subsetneq \cHH_1\subsetneq \cdots \subsetneq \cHH_k$ is an ascending chain of linear subclasses. We rewrite $\theta$ as 
    \begin{equation}
        \theta = \min\left\{ \sum_{i=1}^k a_i+\delta_{\emptyset^c} , \sum_{i=1}^{k-1}a_i+\delta_{\cHH_1^c}, \cdots,  \sum_{i=1}^{k-j}a_i+\delta_{\cHH_j^c}, \cdots , \delta_{\cHH_k^c}\right\}.
    \end{equation}
This shows $\Dr_s^1(M)^\circ$ is in the tropical span of the image of $\Qt^1(M)$ inside $\Dr_s^1(M)$. It remains to show that the tropical span of $\cJ$ contains the image of $\Qt^1(M)$. This is true because the join operation in $\Qt^1(M)$ agrees with tropical sum. By the definition of join-irreducible elements, any element in $\Qt^1(M)$ is a join of elements in $\cJ$. This finishes the proof.
\end{proof}

\begin{rmk}
    The join-irreducible elements of $\Qt^1(M)$ are not necessarily the atoms of $\widehat{\Qt}^1(M)$. For instance, the linear subclass $\{12,34\}$ of $U_{3,6}$ is join-irreducible, but it is not an atom of $\widehat{\Qt}^1(U_{3,6})$.
\end{rmk}

Recall that an \textit{adjoint} $W$ of $M$ is a matroid of rank $d$ on $\cLL^1(M)$ such that $\cLL(M)$ embeds in the order dual of $\cLL(W)$. We now relate adjoints with the geometry of $\Dr_s^1(M)$.

\begin{prop}\label{prop:characterization-of-adjoint}
    Let $M$ and $W$ be simple matroids of the same rank. The following are equivalent:
    \begin{enumerate}
        \item the ground set of $W$ is $\cLL^1(M)$ and each concurrent triple of $M$ is a circuit of $W$;
       
        \item there is an order-reversing embedding $\cLL(W)\hookrightarrow \widehat{\Qt}^1(M)$ that maps $\cLL_1(W)$ onto $\cLL^1(M)$;
        \item the groundset of $W$ is $\cLL^1(M)$ and $\Trop W\subset \Dr_s^1(M)$. 
        \item $W$ is an adjoint of $M$.
    \end{enumerate}
\end{prop}

\begin{proof}
$(2)\Rightarrow(4)$ is proved in \cite{bachem1986extension} and $(4)\Rightarrow(2)$ is Cheung's embedding theorem \cite{cheung1974adjoints}. We prove $(1)\Rightarrow(2)\Rightarrow(3)\Rightarrow(1)$. Identify $\Qt^1(M)$ with the lattice of linear subclasses. By the assumption of (1), $\cLL_1(W)=\cLL^1(M)$. Consider the map 
        \begin{equation}
            \iota\colon\cLL(W) \to \widehat{\Qt}^1(M), F\mapsto \text{ the linear subclass generated by }\{H\in \cLL_1(W)\mid H\leq_W F\}.
        \end{equation}
        From the assumption that each concurrent triple is a circuit of $W$, we know $\{H\in \cLL_1(W)\mid H\leq_W F\}$ is already a linear subclass of $M$. If $F_1\leq_W F_2$, then $\iota(F_1)$ is freer than $\iota(F_2)$, as the corresponding linear subclass of $\iota(F_1)$ is smaller. Injectivity of $\iota$ is clear. Therefore, $\iota$ is an order-reversing embedding. This shows $(1) \Rightarrow (2)$. The implication $(2) \Rightarrow (3)$ follows from \Cref{thm:dressian-is-order-complex} and the fact that $\Trop W$ is the order complex of $\cLL(W)$. Given the hypotheses of (3), the concurrent triples of $M$ are dependent in $W$. Since $W$ is simple, all its circuits have size at least 3, so all the concurrent triples must be circuits of $W$. This shows $(3)\Rightarrow (1)$. 
\end{proof}

It is well known that all rank-3 matroids have adjoints. For obvious geometric reasons, all realizable matroids have adjoints.
\begin{prop}\label{prop:existence-of-ordinary-adjoint}
    If $M$ is realized by a point configuration in a projective space, then the matroid $W$ given by the hyperplane arrangement of all the hyperplanes spanned by the points is an adjoint of $M$.
\end{prop}
Dually, if $M$ is realized by a hyperplane arrangement, then the adjoint arising this way is the \textit{matroid of lines} discussed in \cite{jell2022moduli}. For further properties of adjoints, we refer to \cite[Proposition 2.3]{fu2024adjoints}. The following ones will be useful to us.

\begin{prop}\label{prop:properties-adjoint}
    Let $W$ be an adjoint of $M$.
    \begin{enumerate}
        \item If $B$ is a basis of $M$, then \begin{equation}
            \{\cl_M(B-i)\mid i\in B\}
        \end{equation}
        is a basis of $W$.
        \item Let $F\in \cLL(W)$. Then 
        \begin{equation}
            \{G\in \cLL(M)\mid G\leq_W F\}
        \end{equation}
        is a modular cut of $M$.
        \item If $M$ is a modular matroid, i.e., $\cLL(M)$ is modular, then $W$ is unique and we have isomorphisms of lattices
        \begin{equation}
           \cLL(M) \cong \widehat{\Qt}^1(M) \cong  \cLL(W)^\text{op}.
        \end{equation}
       In particular, $\Dr_s^1(M)\cong \Trop M$.    \end{enumerate}
\end{prop}

\begin{rmk}\label{rmk:unsimplified-adjoint}
    If $M$ has an adjoint $W$, we can define a matroid $\td{W}$ on $\nk{n}{d-1}$ as follows: $(B_1,...,B_d)$ is a basis if and only if 
\begin{center} \vspace{5pt}
    $\rk_M(B_i)=d-1$ for $i=1,...,d$ \quad and \quad $\{\cl_M(B_1),...,\cl_M(B_d)\}$ is a basis of $W$.
\vspace{5pt} \end{center} 
Then the simplification of $\td{W}$ is $W$. The parallel classes of atoms of $\td{W}$ are in bijection to $\cLL^1(M)$, and $B\in\nk{n}{d-1}$ is a loop in $\td{W}$ if and only if $\rk_M(B)<d-1$. Note that $\Trop \td{W}\subset \Dr^1(M)$. The following fact will be useful:
\begin{center} \vspace{5pt}
$\{B_1,...,B_d\}$ is dependent in $\td{W}$ if \; $\rk_M\left(\bigcap_{i=1}^d \cl_M(B_i)\right)\neq 0$.  
\vspace{5pt} \end{center} 
\end{rmk}

\section{Adjoints of valuated matroids}\label{sec:valuated-adjoint}
In this section, we develop the notion of adjoints of valuated matroids. On one hand, this notion is based on \Cref{prop:characterization-of-adjoint}, the geometric characterization of adjoints of ordinary matroids. On the other, it can be understood by drawing analogy from the embedding
\begin{equation}
    \Gamma: \Gr(d,E) \hookrightarrow \Gr(d,\wedge^{d-1}E),\quad V\mapsto \wedge^{d-1}V,
\end{equation}
which is an instance of \textit{geometric plethysms} which generalize Pl\"{u}cker embeddings. Relevant properties of this embedding are proved in \Cref{subsec:cofactor-grassman}. These properties say that the Pl\"{u}cker coordinates of $\wedge^{d-1} V \in \Gr(d,\wedge^{d-1}E)$ can be expressed in terms of the Pl\"{u}cker coordinates of $V\in \Gr(d,E)$.

The technical core of this section is \Cref{lem:Sigma-basis-valuation}, which include three formulas relating the values of $\mu$ to the values of its adjoint $\Sigma$. We call them the \textit{cofactor formulas}. They are the tropical analogs of the properties of the map $\Gamma$. 

In \Cref{subsec:applications-of-adjoint}, we apply what we know about adjoints to proving \Cref{thm:projective-plane}, which includes \Cref{main:G} as a special case.

\subsection{Two formulations}\label{subsec:valuated-adjoint}

We give two formulations of adjoints of valuated matroids, in reference to the two models $\Dr_s^1(\mu)$ and $\Dr^1(\mu)$, respectively.

\begin{subdefinition}{thm}\label{def:valuated-adjoint}
    \begin{definition}\label{def:valuated-adjoint-A}
        A valuated matroid $\Sigma_s\colon\cLL^1(M)\to \Rbar$ with a simple underlying matroid is an \textit{adjoint} of $\mu$, if
\begin{center} \vspace{5pt}
    $\Trop\Sigma_s\subset \Dr^1_s(\mu)$ and $\Sigma_s$ has rank $d$.
\vspace{5pt} \end{center}  
    \end{definition}
Similar to how we define the space $\Dr_s^1(\mu)$, any subset $P$ in $\Dr^1(\mu)$ can be simplified to a subset $P_s\subset \Dr_s^1(\mu)\subset \PT{\cLL^1(M)}$. In particular, for any valuated matroid $\Sigma$ such that $\Trop\Sigma\subset \Dr^1(\mu)$, there is an induced simplification $\Sigma_s$ such that $\Trop\Sigma_s\subset\Dr_s^1(\mu)$. Using this, we give an alternative formulation.
\begin{definition}\label{def:valuated-adjoint-B}
        A valuated matroid $\Sigma\colon \binom{\nk{n}{d-1}}{d}\to\Rbar$ is an \textit{adjoint} of $\mu$ if 
\begin{center} \vspace{5pt}
    $\Trop\Sigma\subset \Dr^1(\mu)$, $\Sigma$ has rank $d$, and its simplification $\Sigma_s\subset \Dr_s^1(\mu)$ has a simple underlying matroid.
\vspace{5pt} \end{center} 
    \end{definition}
\end{subdefinition}

It is easy to verify that the two formulations are consistent: $\Sigma$ is an adjoint in the sense of \Cref{def:valuated-adjoint-B} if and only if the induced simplification $\Sigma_s$ is an adjoint in the sense of \Cref{def:valuated-adjoint-A}.
We state some equivalent definitions of valuated adjoints below in reference to the simplified space $\Dr^1_s(\mu)$.

\begin{prop}\label{prop:valuated-adjoint-simplified}
Let $\Sigma_s\colon\cLL^1(M)\to \Rbar$ be a valuated matroid with a simple underlying matroid $W$. The following are equivalent.
\begin{enumerate}
    \item $\Sigma_s$ is an adjoint of $\mu$ in the sense of \Cref{def:valuated-adjoint-A};
    \item $\Trop\Sigma_s\subset \Dr_s^1(\mu)$ and $W$ is an adjoint of $M$;
    \item $\Sigma_s$ has rank $d$, and each defining equation of $\Dr^1_s(\mu)$ is given by a  valuated circuit of $\Sigma_s$.
\end{enumerate}
\end{prop}

\begin{proof}
    We have
    \begin{equation}
        \rec\Sigma_s=\Trop W \subset \rec \Dr^1_s(\mu)\subset  \Dr^1_s(M).
    \end{equation}
    The second containment follows from \eqref{eqn:recession-space-intersection}. Hence, (1) implies (2) by \Cref{prop:characterization-of-adjoint} (3). If $W$ is an adjoint of $M$, then the concurrent triples are circuits of $W$. Hence, if in addition $\Sigma_s\subset \Dr^1_s(\mu)$, the defining equations of $\Dr^1_s(\mu)$ have to be valuated circuits of $\Sigma_s$, proving $(2)\Rightarrow(3)$. If the defining equations of $\Dr_s^1(\mu)$ are valuated circuits of $\Sigma_s$, then $\Sigma_s\subset \Dr^1_s(\mu)$. Since $\Sigma_s$ has rank $d$, it is an adjoint of $\mu$ by definition.
\end{proof}

We restate \Cref{prop:valuated-adjoint-simplified} in reference to the space $\Dr^1(\mu)$. 

\begin{prop}\label{prop:valuated-adjoint-unsimplified}
Let $\Sigma\colon\binom{\nk{n}{d-1}}{d}\to \Rbar$ be a valuated matroid with underlying matroid $\td{W}$. The following are equivalent.
\begin{enumerate}
    \item $\Sigma$ is an adjoint of $\mu$ in the sense of \Cref{def:valuated-adjoint-B};
    \item $\Trop\Sigma\subset \Dr^1(\mu)$ and the induced simplification of $\td{W}$ to a matroid on $\cLL^1(M)$ is an adjoint of $M$;
    \item $\Sigma$ has rank $d$ and each of the three-term incidence relations defining $\Dr^1(\mu)$ is given by a valuated circuit of $\Sigma$. 
\end{enumerate}
\end{prop}

\subsection{The cofactor formulas for valuated matroids}\label{subsec:cofactor-valuated}

Adjoints of valuated matroids may not exist and may not be unique if they do exist. However, \Cref{lem:Sigma-basis-valuation} prescribes \textit{some} values of $\Sigma$. We call these formulas the \textit{cofactor formulas}, because they are the tropical analogs of the relation between the determinant of a matrix with the determinant of its cofactor matrix. For a full elaboration, see \Cref{subsec:cofactor-grassman}.

By \Cref{prop:properties-adjoint} (1), for any basis $D$ of $\mu$, the subset
\begin{equation}
    \tD =\{D-i\mid i\in D\} \in \binom{\nk{n}{d-1}}{d}
\end{equation}
is a basis of $\Sigma$. we call $\tD$ the basis \textit{induced by} $D$. For $B\in \nk{n}{d-1}$, put 
\begin{center} \vspace{5pt}
$\tD+B:=\tD\cup\{B\}$ if $B\notin \tD$ and $\tD-B:=\tD\backslash\{B\}$ if $B\in \tD$.    
\vspace{5pt} \end{center} 
Besides these induced bases, one finds other bases of $\Sigma$ using basic circuits. If $(\tD,A)$ is a basic circuit of $\Sigma$, where $\tD$ is a basis of $\Sigma$ and the $(d-1)$-set $A$ is not in $\tD$, then for any $B\in (\tD,A)$, $\tD-B+A$ is another basis of $\Sigma$. We call this process an \textit{elementary basis exchange via the circuit} $(\tD,A)$. Moreover, if $\bC_{(B,v)}$ is a valuated basic circuit, then
\begin{equation}\label{eq:basis-val-from-circuit-val}
    \bC_{(\tD,A)}(B)-\bC_{(\tD,A)}(A) = \Sigma(\tD-B+A) - \Sigma(\tD)
\end{equation}
measures how much the basis valuation changes when we change to a new basis.

To avoid stacks of notations, we use the following \textit{temporary} notations for the proof of \Cref{lem:Sigma-basis-valuation}: if $B\subset [n]$ and $j\in B$, we write $B^j$ for $B-j$. The statements in \Cref{lem:Sigma-basis-valuation} look complicated, but the content is simple. We illustrate the statements in \Cref{ex:adjoint-valuation-lemma} using the case $d=3$.

\begin{lemma}[The cofactor formulas]\label{lem:Sigma-basis-valuation} 
    Let $\Sigma$ be an adjoint of $\mu$. Let $\tD_1,\tD_2$ be the bases of $\Sigma$ induced by the bases $D_1$ and $D_2$ of $\mu$, respectively. Then $\Sigma$ satisfies the following.
    \begin{enumerate}

    \item  If $i\in D_1\cap D_2$ and $B$ is any independent $(d-1)$-set of $M$ such that $\tD_1 - D_1^i + B$ is a basis of $\Sigma$, then
    \begin{equation}
        \Sigma(\tD_2 -D_2^i + B) - \Sigma(\tD_1 - D_1^i + B) = (d-2)(\mu(D_2)-\mu(D_1))
    \end{equation}    
        \item $\Sigma(\tD_2)-\Sigma(\tD_1)=(d-1)(\mu(D_2)-\mu(D_1))$.
        \item Let $\tD$ be a basis of $\Sigma$ induced by $D$. If $i\in D$, $B$ is some independent $(d-1)$-set of $M$, and $D'=B+i$, then
        \begin{equation}\label{eqn:valuation-of-sigma-3}
            \Sigma(\tD - D^i + B) - \Sigma(\tD) = \mu(D') - \mu(D)
        \end{equation}
    \end{enumerate}
\end{lemma}

\begin{ex}\label{ex:adjoint-valuation-lemma}
\Cref{subfig:valuation-1,subfig:valuation-2,subfig:valuation-3} illustrate the three cases in \Cref{lem:Sigma-basis-valuation} using point configurations in the plane. In this case, the ground set of $\Sigma$ is all subsets of $[n]$ of size 2, which correspond to lines. The basis $\tD$ of $\Sigma$ induced by the basis $D=\{1,2,3\}$ of $\mu$ is $\{12,13,23\}$, and $\tD^1=\{12,13\}$, etc. Using the same color coding as in the pictures, the three cases in \Cref{lem:Sigma-basis-valuation} mean
\begin{center} \vspace{5pt}
    \Cref{subfig:valuation-1}: $\Sigma({\color{red}12},{\color{red}13},67)-\Sigma({\color{blue}14},{\color{blue}15},67)=\mu(123)-\mu(145)$; \\
    \Cref{subfig:valuation-2}: $\Sigma(12,13,23)-\Sigma({\color{red}45},{\color{red}46},{\color{red}56})=2\big(\mu(123)-\mu({\color{red}456})\big)$; \\
    \Cref{subfig:valuation-3}: $\Sigma(12,13,{\color{red}45})-\Sigma(12,13,{\color{blue}23})=\mu(145)-\mu(123)$.
\vspace{5pt} \end{center} 
\begin{figure}
    \centering
    \begin{tabular}{ccc}
        \subfloat[]{\label{subfig:valuation-1}
        \includegraphics[width=1.2in]{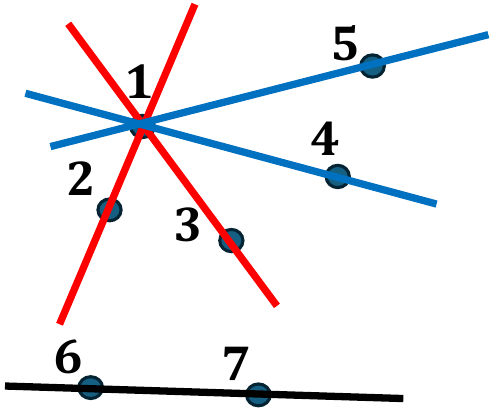}
        } \hspace{0.2in} & \subfloat[]{\label{subfig:valuation-2}
        \includegraphics[width=1.5in]{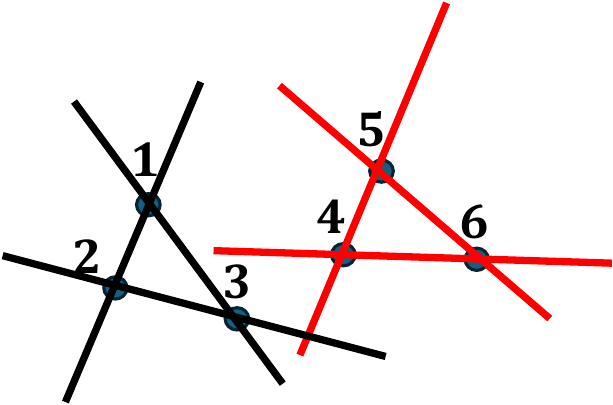}
        }\hspace{0.2in} &  \subfloat[]{\label{subfig:valuation-3}
        \includegraphics[width=1.2in]{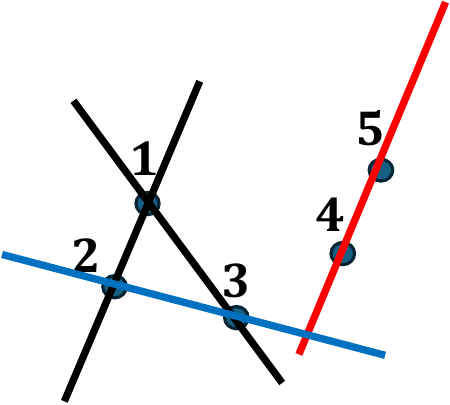}
        }
    \end{tabular}
    
    \caption{The three cases in \Cref{lem:Sigma-basis-valuation} when $d=3$.}
    \label{fig:three-cases}
\end{figure}    
\end{ex}

\begin{proof}
    For the first statement: we first prove the case when $|D_1\cap D_2|=d-1$. WLOG, assume that
    \begin{equation}
        D_1 = A \sqcup \{a\} \text{ and } D_2 = A \sqcup \{b\},
    \end{equation}
    where $A=[d-1]$, $a,b\notin A$ and $a\neq b$. Assume also that $i=1$. We analyze how to get $\tD_2-D_2^1+B$ from $\tD_1-D_1^1+B$ using elementary basis exchange. Explicitly, 
    \begin{center} \vspace{5pt}
        $\tD_1-D_1^1+B = \left\{B,D_1^2,D_1^3,...,D_1^{d-1},A\right\}$ \quad and \quad $\tD_2-D_2^1+B = \left\{B,D_2^2,D_2^3,...,D_2^{d-1},A\right\}$.
    \vspace{5pt} \end{center} 
    Note that for each $k=2,...,d-1$, the hyperplanes $\cl_M(D_1^k)$, $\cl_M(D_2^k)$ and $\cl_M(A)$ are meet at the coline $\cl_M(A^k)$. By \Cref{prop:valuated-adjoint-unsimplified} (3), $\{D_1^k,D_2^k,A\}$ is a circuit of $\Sigma$. We have the following sequence of elementary basis exchanges:  
    \begin{align*}
        \tD_1-D_1^1+B = \left\{B,D_1^2,D_1^3,...,D_1^{d-1},A\right\} & \xrightarrow{+D_2^2 - D_1^2} \left\{B,D_2^2,D_1^3,...,D_1^{d-1},A\right\} \\
        & \xrightarrow{+D_2^3 - D_1^3} \left\{B,D_2^2,D_2^3,...,D_1^{d-1},A\right\} \\
        & \xrightarrow{\quad\cdots\quad} \cdots \\
        & \xrightarrow{+D_2^{d-1} - D_1^{d-1}} \left\{B,D_2^2,D_2^3,...,D_2^{d-1},A\right\} =\tD_2 - D_2^1 + B
    \end{align*}
    Each circuit $\{D_1^k,D_2^k,A\}$ corresponds to the incidence relation for $\Trop\theta\subset\Trop\mu$ indexed by $(A^k,A^kabk)$: 
    \begin{equation}
        \min\Big\{\theta(A^ka)+\mu(A^kbk),\quad \theta(A^kb)+\mu(A^kak),\quad \theta(A^kk)+\mu(A^kab)\Big\},
    \end{equation}
     which corresponds to a valuated circuit $\bC_k\in \Rbar^{\nk{n}{d-1}}$ of $\Sigma$, where
    \begin{equation}
        \bC_k(B) = \begin{cases}
            \mu(A^kbk),&\quad \text{ if }B = A^ka \\
            \mu(A^kak),&\quad \text{ if }B = A^kb \\
            \mu(A^kab),&\quad \text{ if }B = A^kk \\
            \infty, & \quad \text{ else.}
        \end{cases}
    \end{equation}
    Note that $\mu(A^kbk)=\mu(D_2)$ and $\mu(A^kak)=\mu(D_1)$ are both finite. Take the sum of all the changes of the basis valuation given by \eqref{eq:basis-val-from-circuit-val}:
    \begin{align*}
        \Sigma(\tD_2 - D_2^1+B) - \Sigma(\tD_1 - D_1^1+B) & = \sum_{k=2}^{d-1}\big(\bC_k(D_1^{k}) - \bC_k(D_2^{k})\big) \\
        & = \sum_{k=2}^{d-1}\big(\bC_k(A^{k}a) - \bC_k(A^{k}b)\big) \\
        & = (d-2)(\mu(D_2)-\mu(D_1)).
    \end{align*}
    If $D_1$ and $D_2$ are arbitrary bases with a common element $i$, then we can take a sequence of elementary basis exchanges to get $D_2$ from $D_1$. Applying the above calculation to each consecutive step gives the result.

    For the second statement: again, by taking a sequence of elementary basis exchanges, it suffices to prove the case where $D_1$ and $D_2$ differ by one element. Still assume that
    \begin{equation}
        D_1 = A \sqcup \{a\} \text{ and } D_2 = A \sqcup \{b\},
    \end{equation}
    where $A=[d-1]$, $a,b\notin A$ and $a\neq b$. Note that $\cl_M(D_1^1),\cl_M(D_2^1)$, and $\cl_M(A)$ are concurrent hyperplanes, so we have the elementary basis exchange
    \begin{equation}
        \tD_2 - D_2^1 + D_1^1 =\left\{D_1^1,D_2^2,D_2^3,...,D_2^{d-1},A\right\} \xrightarrow{+D_2^1 - D_1^1} \tD_2.
    \end{equation}
    The corresponding incidence relation is indexed by $(A^1,A^1ab1)$:

        \begin{equation}
        \min\Big\{\theta(A^1a)+\mu(A^1b1),\quad\theta(A^1b)+\mu(A^1a1),\quad\theta(A^11)+\mu(A^1ab)\Big\},
    \end{equation}
    and the corresponding valuated circuit $\bC$ is
    \begin{equation}
        \bC(B) = \begin{cases}
            \mu(A^1a1),&\quad \text{ if }B = A^1b \\
            \mu(A^1b1),&\quad \text{ if }B = A^1a \\
            \mu(A^1ab),&\quad \text{ if }B = A^11 \\
            \infty, & \quad \text{ else.}
        \end{cases}
    \end{equation}
     We get  
    \begin{equation}
    \begin{split}
        \Sigma(\tD_2) - \Sigma(\tD_1) = \; &  \Sigma(\tD_2) -\Sigma(\tD_2 - D_2^1 + D_1^1) + \Sigma(\tD_2 - D_2^1 + D_1^1) - \Sigma(\tD_1) \\
        = \; & \bC(D_1^1) - \bC(D_2^1) + \Sigma(\tD_2 - D_2^1 + D_1^1) - \Sigma(\tD_1-D_1^1 + D_1^1) \\
        = \; & \mu(A^1b1) - \mu(A^1a1) + (d-2)(\mu(D_2)-\mu(D_1))\\
        = \; &  \mu(D_2)-\mu(D_1) + (d-2)(\mu(D_2)-\mu(D_1)) \\
         = \; &(d-1)(\mu(D_2)-\mu(D_1)).
    \end{split}
    \end{equation}
    For the third statement, if $i\notin \cl_M(B)$, then \eqref{eqn:valuation-of-sigma-3} is obtained by combining the first two: 
    \begin{equation}
    \begin{split}
        \Sigma(\tD - D^i +B)-\Sigma(\tD) =  \; & \Sigma(\tD - D^i +B) -\Sigma(\tD'-B + B) + \Sigma(\tD') - \Sigma(\tD) \\ = \; & (d-2)(\mu(D)-\mu(D')) + (d-1)(\mu(D')-\mu(D)) \\
        =\; & \mu(D')-\mu(D).
    \end{split}
    \end{equation}
    If $i\in \cl_M(B)$, then
 since 
 \begin{center} \vspace{5pt}
    $i\in \cl_M(A)$ for all $A\in \tD-D^i+B$,
 \vspace{5pt} \end{center} 
we know $\tD-D^i+B$ is not a basis of $\Sigma$ by \Cref{rmk:unsimplified-adjoint}. Note that $D=B+i$ is not a basis of $\mu$, either, so both sides of \eqref{eqn:valuation-of-sigma-3} are $\infty$. This completes the proof.
\end{proof}

\subsection{Valuated matroids with adjoints}
We provide classes of valuated matroids with adjoints.

\begin{prop}\label{prop:projective-geometry-valuated-adjoint}
If the underlying matroid $M$ of $\mu$ is a direct sum of finite projective spaces, then $\mu$ has a unique adjoint $\Sigma_s\colon \cLL^1(M) \to \Rbar$. Moreover, there is an isomorphism $\mu\cong \Sigma_s$.
\end{prop}
\begin{proof}
    A direct sum of finite projective spaces is rigid, so we may assume that $\mu$ is the trivial valuation on $M$. By \Cref{thm:dressian-is-order-complex}, $\Dr_s^1(M)\cong \Trop M$. Both the existence and uniqueness follows from \Cref{def:valuated-adjoint-A}.
\end{proof}

Another source of valuated adjoints is \textit{tropicalization}. Let $\Trop\mu$ be a \textit{tropicalized} linear space. Namely, there is a $d$-dimensional vector space $V\subset\K^n$ such that $\Trop \PP V = \Trop\mu$. Since $W\subset V\Rightarrow \Trop \PP W\subset \Trop \PP V$, we have the following diagram
\begin{equation}\label{eqn:tropicalization-diagram}
    \begin{tikzcd}
        \Gr(d-1,V)\ar[r,hookrightarrow]\ar[d,"\Trop"] & \Gr(d-1,n) \ar[d,"\Trop"] \\
      \Dr^1(\mu) \ar[r,hookrightarrow] & \Dr(d-1,n)
    \end{tikzcd}
\end{equation}
As is discussed in \Cref{subsec:history}, $\Gr(d-1,V)$ lives in the Pl\"{u}cker embedding of $\Gr(d-1,n)$ as a $(d-1)$-dimensional linear subspace, so $\Trop \Gr(d-1,V)$ is a tropical linear space of projective dimension $d-1$ inside $\Dr^1(\mu)$. We show that the valuated matroid associated with $\Trop \Gr(d-1,V)$ is an adjoint of $\mu$. The case when the underlying matroid of $\mu$ is $U_{n-1,n}$ is special: $\Dr^1(\mu)$ is a tropical linear space.

\begin{prop}\label{prop:valuated-adjoint-tropicalization}
    Let $\mu$ be the valuated matroid representing the tropicalized linear space $\Trop V$.
    \begin{enumerate}
        \item The valuated matroid associated with $\Trop \Gr(d-1,V)$ as in \eqref{eqn:tropicalization-diagram} is an adjoint of $\mu$.
        \item If $V$ is a generic hyperplane in $\K^n$, then $\Trop \Gr(n-2,V)=\Dr^1(\mu)$ is a translation of $\Trop M_{K_n}$, where $M_{K_n}$ is the graphic matroid of the complete graph $K_n$.
    \end{enumerate}
\end{prop}

\begin{proof}
    By \cite[Theorem 8]{jell2022moduli}, the matroid $\td{W}$ realized by the linear subspace $\Gr(d-1,V)\subset\mathbb{P}^{\nk{n}{d-1}}$ is the matroid of lines of the hyperplane arrangement given by $V$. By \Cref{prop:existence-of-ordinary-adjoint}, the simplification of $\td{W}$ to a matroid on $\cLL^1(M)$ is an adjoint of $M$. The first statement follows from \Cref{prop:valuated-adjoint-unsimplified} (2).

    If $V$ is a generic hyperplane in $\K^n$, then $\Trop V$ is a translation of $\Trop U_{n-1,n}$, so $\Dr^1(\mu)$ is a translation of $\Dr^1(U_{n-1,n})$. Label the vertices of $K_n$ with $[n]$. Identify the hyperplane $[n]\backslash\{i,j\}$ of $U_{n-1,n}$ with the edge $ij$ of $K_n$. Then a direct computation shows that the concurrent hyperplane relations of $\Dr^1(U_{n-1,n})$ are exactly given by all the 3-cycles of $K_n$. These are all the defining equations of $\Dr^1(U_{n-1,n})$. By \cite[Theorem 3.1]{yu2006representing}, these circuits form a tropical basis for $\Trop M_{K_n}$. Hence, $\Dr^1(U_{n-1,n})=\Trop M_{K_n}$. Since $M_{K_n}$ has rank $n-1$, $\Trop M_{K_n}$ cannot contain a different tropical linear space of the same dimension, so $\Trop \Gr(n-2,V) = \Dr^1(\mu)$.
\end{proof}

It is known that all rank-3 matroids have adjoints. Whether or not all rank-3 valuated matroids have adjoints will be addressed in a subsequent paper.

\subsection{Applications}\label{subsec:applications-of-adjoint}

We now apply the above results to proving \Cref{thm:projective-plane}, from which \Cref{main:G} follows.

\begin{thm}\label{thm:projective-plane}
Let $q$ be a prime power and $M$ be a finite projective plane over $\mathbb{F}_q$. The subset 
\begin{equation}
    \mathscr{V} := \left\{ [\theta]\in \Dr(2,q^3-1) \mid \mu\onto\theta \text{ for some basis valuation $\mu$ on }M \right\}\subset \Dr(2,q^3-1)
\end{equation}
is a subcomplex of codimension at least 2 inside $\Dr(2,q^3-1)$. 
\end{thm}
\begin{proof}
    Since $M$ is rigid, $\mathscr{V}$ is simply $\Dr^1(M)$ plus the $(q^3-2)$-dimensional lineality space coming from translation. By \Cref{prop:properties-adjoint} (3), $\Dr^1(M)$ has dimension 2, so $\dim\mathscr{V}=q^3$. On the other hand, $\Dr(2,q^3-1)$ equals the tropical Grassmannian, so $\dim \Dr(2,q^3-1) = 2(q^3-3) = 2q^3-6$. Now we have 
    \begin{equation}
        \dim\Dr(2,q^3-1)- \dim\mathscr{V} = q^3-6 \geq 2.
    \end{equation}
\end{proof}

 The uniform matroid $U_{2,n}$ is an elementary quotient of any simple rank-3 matroid on $[n]$. In particular, $U_{2,q^3-1}$ is an elementary quotient of the projective plane over $\mathbb{F}_q$. By \Cref{thm:projective-plane}, if $\Trop\theta$ is a generic tropical line, then no tropical plane $\Trop\mu$ contains $\Trop\theta$ if the underlying matroid of $\mu$ is the tropical projective plane over $\mathbb{F}_q$. If we only assume $M$ is simple and rigid, we get a slightly weaker statement.

\begin{thm}\label{thm:rigid-tropical-plane}
    Let $M$ be a simple rigid matroid on $[n]$ of rank 3, where $n\geq 7$. Then the subset
    \begin{equation}
    \mathscr{V} := \left\{ \theta\in \Dr(2,n) \mid \theta\subset \mu \text{ for some basis valuation $\mu$ on }M \right\}\subset \Dr(2,n)
\end{equation}
is strictly contained in $\Dr(2,n)$. In other words, there are tropical lines $\theta$ such that no tropical plane containing $\theta$ has underlying matroid $M$.
\end{thm}
\begin{proof}
    By \Cref{cor:dim-of-dr1}, $\Dr^1(M)$ has cones of dimension 2. By the rigidity of $M$, after adding a lineaity space of dimension $n-1$ coming from translation, $\mathscr{V}$ has a cone of dimension at most $2+n-1 = n+1$, but $\Dr(2,n)$ has pure dimension $2(n-2)-1=2n-5$. The conclusion follows since $2n-5 > n+1$ when $n\geq 7$.
\end{proof}

\section{Tropical linear incidence geometry}\label{sec:incidence-problems}

In this section, we study the tropical analogs of \ref{itm:IG-intersection}, \ref{itm:IG-interpolation} and \ref{itm:IG-flag} listed in the introduction. Our choice of the three questions comes from three poset-theoretic properties $\Qt(M)$ and $\Dr(\mu)$ might have, according to Baker and Bowler's theory of matroids over hyperfields \cite{baker2019matroids}. In that theory, a matroid can be thought of as a `linear subspace' over a \textit{hyperfield}: ordinary matroids are `linear subspaces' over the \textit{Krasner hyperfield}; valuated matroids are `linear subspaces' over the \textit{tropical hyperfield}. Conjecturally, the posets $\Qt(M)$ and $\Dr(\mu)$ might resemble the lattice of flats of a matroid. It is well-known that $\Qt(M)$ may not be a lattice, which is already the case when $M=U_{3,3}$. See \Cref{fig:U33-not-lattice}.
\begin{figure}[H]
    \centering
    \includegraphics[width=0.5\linewidth]{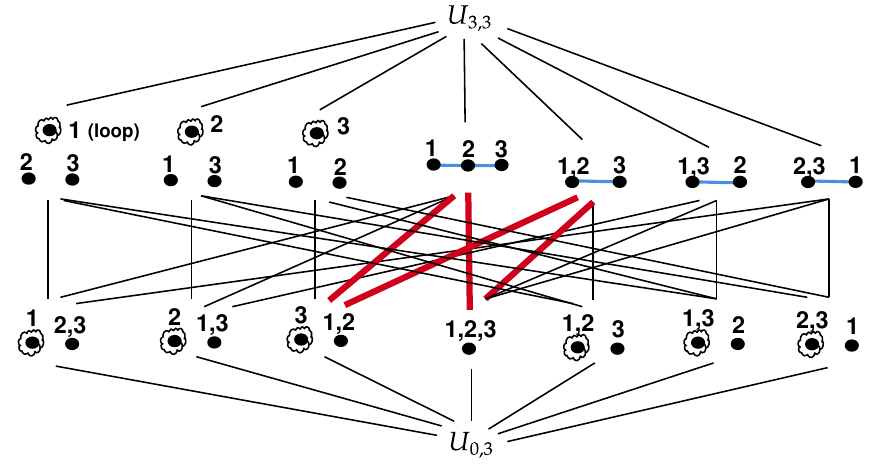}
    \caption{$\Qt(U_{3,3})$ is not a lattice. The highlighted subposet violates the lattice condition.}
    \label{fig:U33-not-lattice}
\end{figure}
Still, the following key properties of $\cLL(M)$ can be defined for general posets, so one can ask whether or not they hold for $\Qt(M)$.
\begin{itemize}
    \item \cite{Haskins_1971} A poset $\cP$ is \textit{upper semimodular}\footnote{For graded lattices such as $\cLL(M)$, the upper semimodularity is equivalent to the \textit{submodularity} of the rank function. Therefore, both the words `upper semimodular' and `submodular' are used to describe the poset $\cLL(M)$.} if 
    \begin{center} \vspace{5pt}
        whenever $F_1,F_2\in \cP$ cover $G$, there is some $H\in \cP$ covering both $F_1$ and $F_2$.   \vspace{5pt} \end{center} 
        \item A poset $\cP$ is \textit{graded} if all maximal chains of $\cP$ have the same length.
    
\end{itemize}
It is known that $\Qt(M)$ is a graded poset, due to the existence of Higgs factorization. For the quest for semimodularity of $\Qt(M)$, we consider the special case $\Qt(U_{n,n})$. By matroid duality, we have (see \Cref{def:matroid-perspectivity}),
\begin{center} \vspace{5pt}
    $\Qt(U_{n,n})$ is upper semimodular $\Leftrightarrow$ if $M_1$ and $M_2$ are coperspective, then they are perspective. 
\vspace{5pt} \end{center} 
We show in \Cref{ex:counter-to-submodular} and \Cref{prop:intersection-fail-at-every-d-larger-than-4} that 
\begin{center} \vspace{5pt}
    $\Qt(U_{n,n})$ is \textit{not} upper semimodular for all $n\geq 8$, 
\vspace{5pt} \end{center} 
while in \Cref{cor:rank-3-modular} we show
\begin{center} \vspace{5pt}
     $\Qt(M)$ is a modular poset, meaning both $\Qt(M)$ and its order dual are upper semimodular, for all $M$ of rank at most 3.
\vspace{5pt} \end{center} 

The lattice of flats $\cLL(M)$  also has the following property:
\begin{itemize}
    \item Any $d-1$ points of $\cLL(M)$ are contained in a hyperplane of $M$.
\end{itemize}
We ask the analogous question for $\Qt(M)$:
\begin{center} \vspace{5pt}
    Given $d-1$ rank-1 quotients $Q_1,...,Q_{d-1}$ of $M$, is there an elementary quotient $N$ of $M$ such that $Q_1,...,Q_{d-1}$ are quotients of $N$?
\vspace{5pt} \end{center} 
We show in \Cref{thm:adjoint-discrete-interpolation} that the answer depends on $M$. 

The tropical analogs of \ref{itm:IG-intersection}, \ref{itm:IG-interpolation} and \ref{itm:IG-flag} come from asking the above questions for the poset $\Dr(\mu)$. Consider the following valuated version of matroid perspectivity.
\begin{definition}\label{def:valuated-perspectivity}
    Valuated matroids $\theta_1,...,\theta_k$ are \textit{coperspective} if they are common elementary quotients of some valuated matroid $\mu$; they are \textit{perspective} if they have a common elementary quotient. Correspondingly, tropical linear spaces $\Trop\theta_1,...,\Trop\theta_k$ are \textit{coperspective} (\textit{perspective}, respectively) if $\theta_1,...,\theta_k$ are coperspective (perspective, respectively).
\end{definition}

We state the tropical analogs of \ref{itm:IG-intersection}, \ref{itm:IG-interpolation}, and \ref{itm:IG-flag} as follows.

\begin{itemize}
    \item \textbf{Existence of a common subspace of expected dimension:} if $\Trop\theta_1$ and $\Trop\theta_2$ are codimension-1 subspaces of $\Trop\mu$, are $\Trop\theta_1$ and $\Trop\theta_2$ perspective?

    We show in \Cref{main:C} that this is true when $d=3$, but we construct counter-examples for every $d\geq 4$ (\Cref{ex:counter-to-submodular} and \Cref{prop:intersection-fail-at-every-d-larger-than-4}).
    \item \textbf{Linear interpolation of points:} is every set of $d-1$ points $[\bw_1],...,[\bw_{d-1}]\in \Trop\mu$ contained in a codimension-1 subspace $\Trop\theta\subset \Trop\mu$?

    No general answer was known except for the case $d=n$, where the subspace $\theta$ is constructed by \textit{stable sum}. We show in \Cref{main:D} and \Cref{main:E} that the answer depends on the structure of $\Dr^1(\mu)$.
    \item \textbf{Completion of partial flags:} can every partial flag $\Trop\mu_k\subset \Trop\mu_0$, where $\Trop\mu_k$ is codimension-$k$ in $\Trop\mu_0$, be completed to a full flag $\Trop\mu_k \subsetneq \Trop\mu_{k-1}\subsetneq \cdots\subsetneq \Trop\mu_0$?

    When the flag starts with $\emptyset$, 
     successive hyperplane sections of $\Trop\mu$ by \textit{arbitrary} hyperplanes complete the flag $\emptyset\subset \Trop\mu$. We show in \Cref{prop:complete-towards-point} that when $\Trop\mu_k$ is a point, any flag can be completed by \textit{well-chosen} successive hyperplane sections. 
\end{itemize}

For convenience of reference, we make the following definitions.
\begin{definition}
    A subset $S\subset \Trop\mu$ has the \textit{linear interpolation property} if every set of $d-1$ points in $S$ is contained in a codimension-1 subspace $\Trop\theta\subset \Trop\mu$.
\end{definition}

\begin{definition}
     A matroid $M$ has the \textit{discrete interpolation property} if for every $d-1$ rank-1 quotients of $M$ there is an elementary quotient $N$ of $M$ such that $Q_1,...,Q_{d-1}$ are quotients of $N$.
\end{definition}
We study the continuous and the discrete versions together.

\subsection{Linear interpolation of points}

We discuss the discrete interpolation problem first. For a rank-1 matroid $Q$, $\cLL(Q)=\{F,[n]\}$ for some $F\subsetneq [n]$, so a rank-1 quotient $Q$ of $M$ is given by a proper flat $F\in \cLL(M)$. The discrete interpolation property is equivalent to the following statement:
\begin{center} \vspace{5pt}
    For any $d-1$ proper flats $F_1,...,F_{d-1}\in \cLL(M)$, there is some elementary quotient $N$ of $M$ such that $F_1,...,F_{d-1}\in \cLL(N)$.
\vspace{5pt} \end{center} 

Recall that $M$ has the \textit{Levi intersection property} if 
every set of $d-1$ hyperplanes are contained in a nontrivial linear subclass (see \Cref{subsec:prelim-matroid-quotient}). The following theorem places the discrete interpolation property in-between the existence of adjoints and the Levi intersection property.

\begin{thm}\label{thm:adjoint-discrete-interpolation}
     Let $M$ be a matroid. If $M$ has an adjoint, then $M$ has the discrete interpolation property. If $M$ has the discrete interpolation property, then $M$ has the Levi intersection property.
 \end{thm}

 \begin{proof}
    Let $W$ be an adjoint of $M$. Let $\cA\subset \cLL(M)$ consist of $d-1$ proper flats. We define an order ideal $\cF\subset \cLL(W)$ recursively. All the poset operations below are within $\cLL(W)$:
    \begin{itemize}
        \item Let $\cA_0=\{F\in \cA\mid \rk_W(F)=1\}$, namely, the subset of $\cA$ consisting of atoms of $W$. Set $\cF_0$ to be the lower interval of $\cl_W(\cA_0)$. Note that $\rk_W(\cl_W(\cA_0))\leq |\cA_0|$
        \item If $\cF_i$ is the lower interval of $F_i$, and there is some $G\in \cA\backslash\cF_i$ such that $G$ covers some element in $\cF_i$ (which must be $F_i\wedge G$), set $\cF_{i+1}$ to be the lower interval of $F_{ i+1}=F_i\vee G$. Since $F$ covers $F\wedge G_i$, $F_{i+1}$ covers $F_i$ by submodularity. This means $\rk_W(F_{i+1})=\rk_W(F_i)+1$; if no such $G\in \cA$ exists, terminate this process;
        \item Let $\cF=\cF_m$ if the process terminates at the $m$-th step.
    \end{itemize}
    The above process will terminate in at most $|\cA|=d-1$ steps. By construction, $\cF$ is a lower interval of some flat $F_m\in \cLL(W)$ and $\rk_W(F_m)\leq d-1$. Consider $\cF'=\{F\in \cF\mid F\in \cLL(M)\}$, which is a nontrivial modular cut of $M$ by \Cref{prop:properties-adjoint} (2). If a flat in $\cA$ does not belong to $\cF'$, then it is not covered by any element in $\cF'$. In other words, $\cA$ does not meet the collar of $\cF$, so the elementary quotient of $M$ by $\cF'$ contains $\cA$. This proves the first implication.

    Take $d-1$ hyperplanes $H_1,...,H_{d-1}\in \cLL^1(M)$. Since $M$ has the discrete interpolation property, there is an elementary quotient $N\in \Qt^1(M)$ such that $H_1,...,H_{d-1}\in \cLL(N)$. The nontrivial linear subclass $\cLL^1(N)\cap \cLL^1(M)$ contains $H_1,...,H_{d-1}$. Hence, $M$ has the Levi intersection property.
\end{proof}

Given any collection of points $[\bw_1],...,[\bw_k]\in \PT{[n]}$, put 
\begin{equation}
    \Dr^1(\mu;\bw_1,...,\bw_k) := \left\{\;[\theta]\in \Dr^1(\mu)\mid [\bw_1],...,[\bw_k]\in \Trop\theta\;\right\}.
\end{equation}
 Then $S\subset\Trop\mu$ has the linear interpolation property if and only if
\begin{center} \vspace{5pt}
$\Dr^1(\mu;\bw_1,...,\bw_{d-1})\neq\emptyset$ whenever $[\bw_1],...,[\bw_{d-1}]\in S$.
\vspace{5pt} \end{center} 
We now prove \Cref{main:E} and \Cref{main:D}, which provide negative and positive cases of the linear interpolation property, respectively.

   \begin{repthm}{main:E}
Let $M$ be a matroid. If $M$ does not have the Levi intersection property, then $\Trop M\cap T_{[n]}$ does not have the linear interpolation property. 
\end{repthm}

\begin{proof}
    Suppose $H_1,...,H_{d-1}$ are the hyperplanes where the Levi intersection property fails. Namely, the only linear subclass containing $\{H_1,...,H_{d-1}\}$ is $\cLL^1(M)$. Pick $c>0$ and consider the points $[\bw_1],...,[\bw_{d-1}]$ where
    \begin{equation}
        {\bw_i}(k)=\begin{cases}
            c, & \text{ if }k\in H_i \\
            0, & \text{ else }.
            \end{cases}
    \end{equation}
    Each $\bw_i$ is a tropical convex combination of the valuated cocircuit associated with the hyperplane $H_i$ and the point $(c,...,c)$, so $[\bw_i]\in \Trop M$.
    
    Suppose for a contradiction that $\Trop\theta$ is a tropical linear subspace of rank $d-1$ containing all the points. WLOG, we may assume that $\min_B\{\theta(B)\}=0$. Let $M^{\bw_i}$ be the initial matroid of $M$ at ${\bw_i}$ and $\theta^{\bw_i}$ be the initial matroid of $\theta$ at $\bw_i$. A simple calculation shows that $M^{\bw_i}$ is the direct sum of (the restriction of $M$ on) $H_i$ and the unique loopless rank-1 matroid on $[n]\backslash H_i$. Taking initial matroids preserves quotients, so $M^{\bw_i}\onto \theta^{\bw_i}$. By assumption, $[\bw_i]\in \Trop\theta$, so $\theta^{\bw_i}$ is loopless. Hence, $\theta^{\bw_i}$ has a basis $B_i=D_i + a_i$, where $D_i\subset H_i$ is of size $d-2$ and $a_i\in [n]\backslash H_i$ is any element.
    
    By the parallel hyperplane relations, if $A_1,A_2$ are two bases of any hyperplane $H$, then $\theta(A_1)=\theta(A_2)$. In this case, put $\theta_H=\theta(A_1)$.  A basis of $H_i$ may or may not be a basis of $\theta^{\bw_i}$. By definition, $B_i$ is a minimizer of 
    \begin{equation}
        \min_{B}\{\theta(B)- {\bw_i} \cdot \be_B\},
    \end{equation}
    so we have $\theta_{H_i} -(d-1)c \geq \theta(B_i) - (d-2)c$. In other words,
\begin{center} \vspace{5pt}
$\theta_{H_i}\geq \theta(B_i)+c \geq c$ for each $i$.    
\vspace{5pt} \end{center} 
 We show that this implies $\theta_H\geq c$ for all hyperplanes $H$, thus contradicting to the assumption that $\theta$ can take the value 0. Let $\cHH_0=\{H_1,...,H_{d-1}\}$. Set
    \begin{equation}
        \cHH_{k} = \cHH_{k-1} \cup \Big\{H\in \cLL^1(M)\mid (H,H_i,H_j)\text{ is a concurrent triple for some }H_i,H_j\in \cHH_{k-1} \Big\}
    \end{equation}
    By construction, the sequence $\cHH_k$ will not stabilize until $\cHH_k$ is a linear subclass. Moreover, from the defining equations \eqref{eqn:defining-equations-dressian-tropM} of $\Dr_s^1(M)$, we see that if $(H,H',H'')$ is a concurrent triple, then \begin{equation}
        \theta_H\geq c,\theta_{H'}\geq c \Rightarrow \theta_{H''}\geq c.
    \end{equation} 
    This means if $\theta_H\geq c$ for all $H\in \cHH_k$, then $\theta_H\geq c$ for all $H\in \cHH_{k+1}$. Since the smallest linear subclass containing $\cHH_0$ is the set of all hyperplanes, we see that $\theta_H\geq c$ for all $H\in \cLL^1(M)$.
\end{proof}

\begin{ex}\label{ex:vamos}
    The V\'{a}mos matroid $V_8$, whose affine diagram is given in \Cref{fig:vamos}, does not have the Levi intersection property \cite{cheung1974adjoints}. The only linear subclass containing $\{1234,1256,3456\}$ is the trivial one.
\end{ex}

\begin{figure}
    \centering
    \includegraphics[width=1.5in]{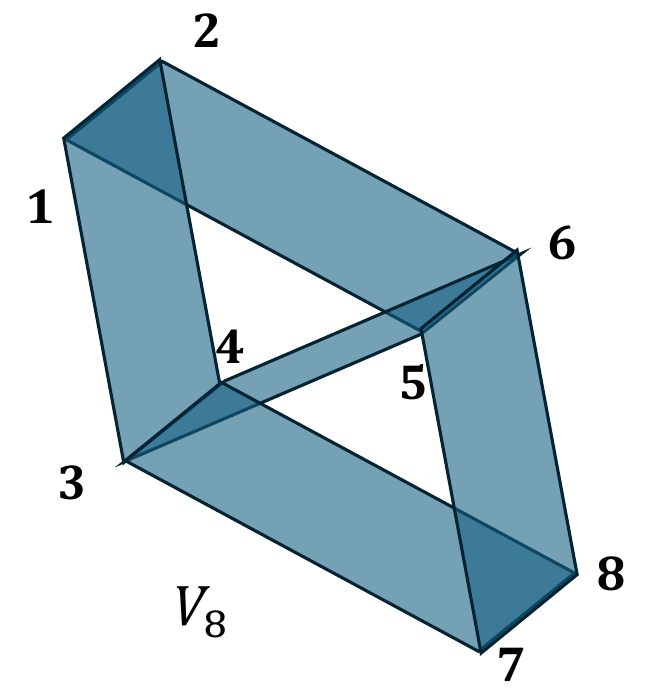}
    \caption{The V\'{a}mos matroid does not have the Levi intersection property}
    \label{fig:vamos}
\end{figure}

We now proceed to prove \Cref{main:D}. This requires careful examination of the space $\Dr^1(\mu;\bw)$ when $\mu$ has an adjoint $\Sigma$. Using \Cref{lem:Sigma-basis-valuation}, we show that $\Dr^1(\mu;\bw)$ contains a hyperplane section of $\Sigma$ when $[\bw]\in \Trop\mu\cap T_{[n]}$. \Cref{main:D} then follows from an intersection argument. 

To complete the argument, we need the following fact on how valuated circuits change under truncation. For any $\bw\in\Rbar^n\backslash\{(\infty,...,\infty)\}$, let $\nu$ be the corank-1 valuated matroid given by
\begin{equation}\label{eqn:corank-1-valuation}
    \nu([n]-i) = \bw(i).
\end{equation}
Then the \textit{truncation} of $\mu$ \textit{by} $\nu$, denoted $\Tr_{\nu}\mu$, is a valuated matroid of rank $d-1$ given by
\begin{equation}\label{eqn:valuation-under-truncation}
    \Tr_{\nu}\mu(B) = \min_{i\notin B}\{\mu(B+i)+\bw(i)\}.
\end{equation}
Let $\mu'=\Tr_{\nu}\mu$. Geometrically, $\Trop\mu'$ is the stable intersection $\Trop\mu\stcap \Trop\nu$. If $F\subset [n]$ is the support of $\bw$, then the underlying matroid of $\Tr_\nu\mu$ is the principal truncation of $M$ by $\cl_M(F)$. In particular, if $\bw\in \R^n$, then $\Trop\nu$ is a standard tropical hyperplane centered at $[-\bw]\in T_{[n]}$, and the underlying matroid of $\Tr_\nu \mu$ is the truncation $\Tr M$ of $M$. 

The circuits of the truncation $\Tr M$ of a matroid $M$ consists of the circuits of $M$ of size no larger than $d$, plus the bases of $M$. This can be refined for valuated matroids as follows. Suppose $\Trop\nu$ is a tropical hyperplane centered at $[\bw]\in T_{[n]}$. The valuated circuit $\bC_{D}$ of $\Tr_\nu \mu$ associated with a basis $D$ of $\mu$ is given by
\begin{equation}\label{eq:truncated-circuit}
    \bC_D(i) = \Tr_\nu\mu(D-i) =  \min_{j\notin D-i}\{\mu(D-i+j)-\bw(j)\}.
\end{equation}
We summarize the above as the following lemma.

\begin{lemma}\label{lem:valuation-of-truncation}
    Let $\nu$ be the corank-1 valuated matroid given by \eqref{eqn:corank-1-valuation} for some point $\bw\in \R^n$. Then the valuated circuits of $\mu'=\Tr_\nu\mu$ consist of the valuated circuits of $\mu$ whose underlying circuits have size no large than $d$, plus the circuits $\bC_D$ described by \eqref{eq:truncated-circuit}.
\end{lemma}

The next lemma rephrases \cite[Corollary 4.2.11]{maclagan2021introduction}. We will need it as a criterion for a tropical linear space to contain a given point.

\begin{lemma}\label{lem:point-membership}
    Let $[\bw]\in T_{[n]}$. Then $[\bw]\in \Trop\mu$ if and only if the minimizers of 
    \begin{equation}
        \min_{B}\{\mu(B) - \bw\cdot\be_B\}
    \end{equation}
    cover $[n]$.
\end{lemma} 

For any point $\bw\in \R^n$, let $\bw^{\wedge k}$ be a point in $\R^{\nk{n}{k}}$ whose coordinates, indexed by $\nk{n}{k}$, are given by
\begin{equation}
    (\bw^{\wedge k})(I) = \bw\cdot \be_I = \sum_{i\in I}\bw(i).
\end{equation}
This operation descends to $T_{[n]}$. We get a point $[\bw^{\wedge k}]\in T_{\nk{n}{k}}$.

\begin{lemma}\label{lem:adjoint-and-linear-interpolation}
    Let $[\bw]\in \Trop\mu\cap T_{[n]}$ and $[\bV]=[\bw^{\wedge (d-1)}]$. Then $H_\bV\stcap \Trop\Sigma\subset \Dr^1(\mu;\bw)$.
\end{lemma}

\begin{proof}
    Let $[\theta]\in H_\bV\stcap \Trop\Sigma$. We show that $[\bw]\in \Trop\theta$ by showing that all the incidence relations are satisfied. The nontrivial incidence relations for this containment have the following form: for each $d$-set $D$ whose rank in $M$ is at least $d-1$, the following vanishes tropically, 
    \begin{equation}
        \min_{i\in D}\{\theta(D-i)+\bw(i)\}.
    \end{equation}
    This is the same as the tropical vanishing of
    \begin{equation}\label{eqn:vanishing-to-prove}
        \min_{B\subset D}\{\theta(B)-\bV(B)\}
    \end{equation}
    where the minimum ranges over $(d-1)$-sets $B\subset D$. There are two cases depending on whether or not $D$ is a basis of $\mu$.
    \begin{itemize}
            \item Case I: $D$ is a basis of $\mu$. By \Cref{lem:valuation-of-truncation}, $\theta$ satisfies the tropical vanishing of
    \begin{equation}\label{eqn:vanishing-at-hand}
        \min_{B\subset D}\{\theta(B)+\bC_{\tD}(B)\},
    \end{equation}
    where $\bC_{\tD}(B)=\min_{|B'|=d-1}\{\Sigma(\tD - B + B')-\bV(B')\}$, $\tD$ is the basis of $\Sigma$ induced by $D$, and $B'$ ranges over all $(d-1)$-sets. If we can show that $\bC_{\tD}(B)+\bV(B)$ does not depend on $B$, then we know that \eqref{eqn:vanishing-at-hand} is exactly the relation \eqref{eqn:vanishing-to-prove}. Fix $B\subset D$ and let $i=D\backslash B$. We compute
        \begin{align*}
            \bC_{\tD}(B)-\Sigma(\tD)+\bV(B) & = \min_{|B'|=d-1}\{\Sigma(\tD - B + B')-\Sigma(\tD) -\bV(B') + \bV(B)\} \\
            &  \xlongequal{\Cref{lem:Sigma-basis-valuation}(3)}  \min_{|B'|=d-1}\{\mu(B'i) - \mu(D) -\bV(B') + \bV(B)\} \\
           & = \min_{|B'|=d-1}\{\mu(B'i) - \mu(D) -\bw\cdot \be_{B'i} + \bw\cdot\be_{Bi}\} \\
           & = \min_{|B'|=d-1}\{\mu(B'i) -\bw\cdot \be_{B'i} \} -(\mu(D) -\bw\cdot\be_D).
        \end{align*}
    
    By \Cref{lem:point-membership}, the minimizers of $\min_{|D'|=d}\{\mu(D')-\bw\cdot \be_{D'}\}$ cover $[n]$, so 
    \begin{equation}
       \min_{|B'|=d-1}\{\mu(B'i) -\bw\cdot\be_{B'i} \} =\min_{|D'|=d}\{\mu(D')-\bw\cdot\be_{D'}\},
    \end{equation}
    which is independent of $i$. Therefore, for fixed $D$, $C_{\tD}(B)-\Sigma(\tD)+\bV(B)$, and thus $C_{\tD}(B)+\bV(B)$, is independent of $B$.

    \item Case II: $D$ is not a basis of $\mu$. Since we assume that $\mu$ is loopless, $D$ contains at least two bases of $\cl_M(D)$. If $B_1,B_2\subset D$ are distinct bases of $\cl_M(D)$, then $|B_1\backslash B_2|=1$ and by the parallel hyperplane relations, $\theta(B_1)-\theta(B_2)=\mu(B_1k)-\mu(B_2k)$ is independent of $k$ when $Dk$ spans $[n]$. Pick any such $k$. We have the incidence relation for $[\bw]\in \Trop\mu$ indexed by $(\emptyset,Dk)$:
    \begin{equation}\label{eq:point-mu-incidence}
         \min_{k\notin B}\{\mu(Bk)-\bw\cdot\be_{Bk}\}.
    \end{equation}
    Since $\theta(B_1)-\theta(B_2)=\mu(B_1k)-\mu(B_2k)$, \eqref{eq:point-mu-incidence} is equivalent to the tropical vanishing of
    \begin{equation}
        \min_{k\notin B}\{\theta(B)-\bw\cdot \be_{Bk}\}
    \end{equation}
    and thus equivalent to the tropical vanishing of $\min_{B\subset D}\{\theta(B)-\bV(B)\}$.
    \end{itemize}
    Therefore, all the incidence relations for $[\bw]\in \Trop\theta$ are satisfied.
\end{proof}

\begin{repthm}{main:D}
    Let $\mu$ be valuated matroid on $[n]$. If $\mu$ has an adjoint, then $\Trop\mu$ has the linear interpolation property.
\end{repthm}

\begin{proof}
    First, we consider the case $[\bw_1],...,[\bw_{d-1}]\in \Trop\mu\cap T_{[n]}$. Let $\Sigma$ be an adjoint of $\mu$. Let $H_{{\bV_i}}\subset \PT{\nk{n}{d-1}}$ be the tropical hyperplane centered at $[\bV_i]=[{\bw_i}^{\wedge(d-1)}]$. We have
\begin{equation}
    \begin{split}
        H_{\bV_1}\stcap H_{\bV_2} \stcap \cdots \stcap H_{\bV_{d-1}}\stcap \Trop \Sigma & \subset \bigcap_{i=1}^{d-1} (H_{{\bV_i}} \stcap \Trop\Sigma) \\
        & \subset \bigcap_{i=1}^{d-1}\Dr^1(\mu;{\bw_i}) \\
        & =\Dr^1(\mu;\bw_1,...,\bw_{d-1}).
    \end{split}
\end{equation}
    Since $\Sigma$ has rank $d$, the above stable intersection is nonempty, so $\Dr^1(\mu;\bw_1,...,\bw_{d-1})$ is nonempty.

    If $[\bw_1],...,[\bw_{d-1}]$ are arbitrary points in $\Trop M$, then each point $[\bw_i]$ is the limit of a sequence $\{\{[\bw_i^k]\}\}$ where $[\bw_i^k]\in \Trop M$ for all $k$. For each $k$, we get from stable intersection a unique tropical linear subspace $\Trop\theta_k$ containing $\{[\bw_1^k],...,[\bw_{d-1}^k]\}$. Since $\Dr^1(\mu)$ is compact, $[\theta_k]$ limits to some $[\theta]$. Since the incidence relations are closed conditions, $\Trop\theta$ contains $\{[\bw_1],...,[\bw_{d-1}]\}$. 
\end{proof}

\subsection{Completion of partial flags}
Recall the setup of the question: $\mu_0\onto \mu_k$ is a quotient of valuated matroids with a flag of underlying
matroids $M_0\onto M_k$, where $\rk M_0=d$ and $\rk M_k = d-k$. We ask if there are valuated matroids $\mu_1,...,\mu_{k-1}$ such that $\mu_0\onto \mu_1\onto\cdots\onto \mu_k$ where each consecutive quotient is an elementary quotient.

We first consider two variants of the question motivated by quotients of ordinary matroids. If $M_0\onto M_k$ is a quotient of corank $k$, then there is a matroid $\ol{M}$ on $[n]\sqcup I$ where $|I|=k$, such that $M_0 = \ol{M}\backslash I$ and $M_k=\ol{M}/I$. The existence of such an $\ol{M}$ implies the existence of the Higgs factorization. We give a short proof of this classical fact using Lorentzian polynomials.

\begin{prop}\label{prop:extension-implies-Higgs}
    Let $\ol{M}$ be a matroid on $[n]\sqcup I$ such that $I$ is independent in $I$. Let $M_0=\ol{M}\backslash I$ and $M_k=\ol{M}/I$. Then $M_0\onto M_k$ and it has a Higgs factorization.
\end{prop}

\begin{proof}
    Suppose $|I|=k$. The basis generating polynomial of $\ol{M}$ can be written as 
    \begin{equation}
        f_{\ol{M}}(w_1,...,w_n,z_1,...,z_k) = f_{M_0} + \sum^k_{i=1}z_if_i + \sum_{\ba\in \nk{k}{2} }z^\ba f_\ba + \cdots + z_1z_2...z_k f_{M_k}. 
    \end{equation}
    Setting $z_1=...=z_k=w_0$, by \cite[Theorem 2.10]{branden2020lorentzian} we get a new Lorentzian polynomial in the following form
    \begin{equation}
        f(w_0,w_1,...,w_n) = f_{M_0} + w_0f_1 + w_0^2f_2 + \cdots + w_0^k f_{M_k}.
    \end{equation}
    By \Cref{prop:interval-preserves-lorentzian}, $\supp f_{M_0}\onto \supp f_1 \onto \cdots \onto \supp f_{M_k}$ is a Higgs factorization.
\end{proof}

The above discussion suggests the following questions.

\begin{qst}\label{qst:completion-to-higgs}
    Can we complete the flag $\mu_0\onto \mu_k$ such that the flag of underlying matroids $M_0\onto \cdots\onto M_k$ is a Higgs factorization?
\end{qst}

\begin{qst}\label{qst:completion-by-extension}
    Is there a valuated matroid $\ol{\mu}$ of rank $d$ on $[n]\sqcup I$, where $|I|=k$, such that $\mu_0=\ol{\mu}\backslash I$ and $\mu_k=\ol{\mu}/I$?
\end{qst}

\cite[Theorem 5.1.2]{brandt2021tropical} gives an affirmative answer to \Cref{qst:completion-by-extension} in the case of elementary quotients ($k=1$), which is a special case of \Cref{lem:elementary-quotient-equiv-description}. In the matroid case, \Cref{qst:completion-to-higgs} and \Cref{qst:completion-by-extension} imply each other. We show that \Cref{qst:completion-by-extension} implies \Cref{qst:completion-to-higgs} for valuated matroids.

\begin{prop}\label{prop:extension-implies-higgs}
    Let $\ol{\mu}$ be a valuated matroid such on $[n]\sqcup I$ and $I$ is independent in $\ol{\mu}$. Let $\mu_0$ be the deletion $\ol{\mu}\backslash I$ and $\mu_2$ be the contraction $\ol{\mu}/I$. Then there is some $\mu_1$ such that $\mu_0\onto \mu_1\onto\mu_2$ and the underlying flag matroid $M_0\onto M_1\onto M_2$ is a Higgs factorization.
\end{prop}

\begin{proof}
    Suppose $|I|=k$. The basis generating polynomial of $\ol{\mu}$ can be written as 
    \begin{equation}
        f_q^{\ol{\mu}}(w_1,...,w_n,z_1,...,z_k) = f_q^{\mu_0} + \sum^k_{i=1}z_if_q^{\theta_i} + \sum_{\ba\in \nk{k}{2} }z^\ba f_q^{\theta_\ba} + \cdots + z_1z_2...z_k f_q^{\mu_k}, 
    \end{equation}
    where the $\theta_i, \theta_\ba$, etc. are valuated matroids on $[n]$ obtained by deleting and contracting certain elements in $I$. Setting $z_1=...=z_k=w_0$, we get a new Lorentzian polynomial in the following form
    \begin{equation}
        f_q(w_0,w_1,...,w_n) = f_q^{\mu_0} + w_0\sum^k_{i=1}f_q^{\theta_i} + w_0^2\sum_{\ba\in \nk{k}{2} } f_q^{\theta_\ba} + \cdots + w_0^k f_q^{\mu_k}
    \end{equation}
    Since $f_q$ is Lorentzian for any $0<q\leq 1$, we can regard $f_q$ as a Lorentzian polynomial over $\K$. We then get a factorization 
    \[\mu_0\onto \min_i\{\theta_i\}\onto\min_{\ba\in \nk{k}{2}}\{\theta_\ba\}\onto \cdots \onto \mu_k \] by tropicalizing $f$. \Cref{prop:extension-implies-Higgs} shows that it is a Higgs factorization.
\end{proof}

Now we show that if the flag starts with a point (or, dually, ends with a hyperplane), then both \Cref{qst:completion-to-higgs} and \Cref{qst:completion-by-extension} have affirmative answers. \Cref{main:F} is then a corollary. 
\begin{prop}\label{prop:complete-towards-point}
     Let $\mu_0\onto\mu_{d-1}$ be an quotient of valuated matroids of corank $d-1$ where $\mu_{d-1}$ has rank 1. Let $M_0$ and $M_{d-1}$ be the underlying matroids of $\mu_0$ and $\mu_{d-1}$, respectively.
     \begin{enumerate}
         \item There are valuated matroids $\mu_1,...,\mu_{d-2}$ such that $\mu_0\onto\mu_1\onto\cdots \onto \mu_{d-2}\onto \mu_{d-1}$ and the flag of underlying matroids is the Higgs factorization of $M_0\onto M_{d-1}$.
         \item There is a valuated matroid $\ol{\mu}$ on $[n]\sqcup I$ of rank $d$, where $|I|=d-1$, such that $\ol{\mu}\backslash I = \mu_0$ and $\ol{\mu}/I=\mu_{d-1}$.
     \end{enumerate}
 \end{prop}

 \begin{proof}
 We will construct the valuated matroids $\mu_1,...,\mu_{d-2}$ as successive truncations of $\mu_0$. Suppose $\cLL(M_{d-1})=\{F,[n]\}$. Then $M_{d-1}=U_{0,F}\oplus U_{1,[n]\backslash F}$. By translation, we may assume that $\mu_{d-1}$ is the trivial valuation on $M_{d-1}$. There are three cases depending on $\rk_{M_0}(F)$.
 \begin{itemize}
     \item Case I: $\rk_{M_0}(F)=0$. Since we assume $M$ is simple, this means $F=\emptyset$, so $M_{d-1}=U_{1,n}$. Let $\nu$ be the trivial valuation on $U_{n-1,n}$ and $\mu_1=\Tr_{\nu}\mu_0$. By \eqref{eqn:valuation-under-truncation}, $\mu_1$ is given by
     \begin{equation}
    \mu_1(B) =\min_{B\subset D} \{\mu(D)\}.
 \end{equation}
 By \Cref{lem:point-membership}, the minimizers of $\min_{|D|=d}\{\mu(D)\}$ cover $[n]$. This means the minimizers of
 \begin{equation}
     \min_{|B|=d-1}\{\mu_1(B)\}
 \end{equation} 
 also cover $[n]$, so $\mu_1\onto\mu_{d-1}$. By induction, we can construct $\mu_2,...,\mu_{d-1}$ where $\mu_{k+1}=\Tr_{\nu}\mu_k$ and conclude that 
 \begin{equation}
     \mu_0\onto \mu_1\onto\cdots \mu_{d-2}\onto \mu_{d-1}.
 \end{equation}
 \item Case II: $\rk_{M_0}(F)=d-1$. Let $M_F=U_{|F|-1,F}\oplus U_{n-|F|,[n]\backslash F}$. Let $\nu$ be the trivial valuation on $M_F$ and $\mu_1=\Tr_{\nu}\mu_0$. Then
 \begin{equation}
     \nu([n]-i) = \begin{cases}
         0,& \text{ if }i\in F \\
         \infty, & \text{ if }i\notin F.
     \end{cases}
 \end{equation}
 By \eqref{eqn:valuation-under-truncation},
 \begin{equation}
     \mu_1(B) = \min_{i\in F}\{\mu_0(B+i)\}.
 \end{equation}
 Let $D$ be any $d$-set. We show that the incidence relation for $\mu_1\onto\mu_{d-1}$ indexed by $(\emptyset, D)$ is satisfied, i.e., 
 \begin{equation}\label{eqn:proof-temp}
     \min_{i\in D}\{\mu_1(D-i)+\mu_{d-1}(i)\} = \min_{i\in D\backslash F}\{\mu_1(D-i)\}
 \end{equation}
 vanishes tropically. Suppose $i_1$ is a minimizer of \eqref{eqn:proof-temp} and $\mu_1(D-i_1)=\mu_0(D-i_1+j_1)$ for some $j_1\in F$. The incidence relation for $\mu_0\onto \mu_{d-1}$ indexed by $(\emptyset,D+j_1)$ implies that there is some $i_2\in D\backslash F$, $i_2\neq i_1$, such that $\mu_0(D-i_2+j_1)\leq \mu_0(D-i_1+j_1)$. Now we have 
 \begin{equation}
     \mu_1(D-i_1)\geq \mu_0(D-i_2+j_1) \geq \mu_1(D-i_2),
 \end{equation}
so $\mu_1(D-i_1)=\mu_1(D-i_2)$ and \eqref{eqn:proof-temp} vanishes tropically. Let $M_1$ be the underlying matroid of $\mu_1$. Then $M_1$ is the principal truncation of $M_0$ by $F$, so $\rk_{M_1}(F)=d-2$. Repeating this process, we can construct $\mu_2,...,\mu_{d-1}$ where $\mu_{k+1}=\Tr_{\nu}\mu_k$ that complete the flag.
  \item Case III: $0<\rk_{M_0}(F)=s<d-1$. Let $\nu$ be the trivial valuation on $U_{n-1,n}$ and $\nu'$ be the trivial valuation on $U_{|F|-1,F}\oplus U_{n-|F|,[n]\backslash F}$. Suppose we have constructed $\mu_i$ with underlying matroid $M_i$. Put
  \begin{equation}
      \mu_{i+1} = \begin{cases}
          \Tr_{\nu}\mu_i,& \text{ if }\rk_{M_i}(F) < d-i-1 \\
          \Tr_{\nu'}\mu_i,& \text{ if }\rk_{M_i}(F)=d-i-1.
      \end{cases}
  \end{equation}
  It is then routine to verify that $\mu_0\onto \mu_1\onto\cdots \onto \mu_{d-1}$, which we have done in the previous two cases.
 \end{itemize}
In the first case, the underlying matroids are successive truncations of $M_0$; in the second case, the underlying matroids are successive principal truncations of $M_0$ by $F$; in the third case, the underlying matroids are successive truncations of $M_0$ until $F$ becomes a hyperplane, followed by successive principal truncations by $F$. In all three cases, the flag of underlying matroids is a Higgs factorization of $M_0\onto M_{d-1}$. 

It remains to construct the matroid $\ol{\mu}$. The key observation is that the completed flag is constructed by principal truncations. Each principal truncation $\mu_i\onto \mu_{i+1}$ comes from a principal extension $\ol{\mu_i}$, in the sense that $\mu_i$ is a deletion of $\ol{\mu_i}$ and $\mu_{i+1}$ is a contraction of $\ol{\mu_i}$.

Let $\mu_{d-2,1}$ be the principal extension on the ground set $[n+1]$ corresponding to $\mu_{d-2}\onto \mu_{d-1}$, such that $\mu_{d-2,1}/\{n+1\} =\mu_{d-1}$ and $\mu_{d-2,1}\backslash\{n+1\}=\mu_{d-2}$. Since principal extensions are compatible with each other, we can simultaneously extend $\mu_{d-3},...,\mu_0$ and get valuated matroids $\mu_{d-3,1},...,\mu_{0,1}$ on $[n+1]$, such that $\mu_{i,1}\backslash \{n+1\} = \mu_i$ for all $i=0,...,d-2$, and 
\begin{equation}
    \mu_{0,1} \onto \mu_{1,1} \onto \cdots \onto \mu_{d-2,1}.
\end{equation}
Moreover, each quotient remains a principal truncation. This reduces the steps of the flag. By induction, we get the desired valuated matroid $\ol{\mu}$.
\end{proof}

We rephrase the above proof in terms of tropical linear spaces. If the point $\Trop\mu_{d-1}$ is in $T_{[n]}$, then we complete the flag $\Trop\mu_{d-1}\subset\Trop\mu_0$ by taking successive sections by the hyperplane centered at $\Trop\mu_{d-1}$. The other two cases are dealt with likewise by choosing appropriate hyperplanes. Principal extensions correspond to tropical modifications by a tropical linear function, and we construct the space $\Trop\ol{\mu}$ by a sequence of such tropical modifications. 

\subsection{When is $\Qt(M)$ upper semimodular?}\label{subsec:sQ-modularity}
Recall that for the poset $\Qt(U_{n,n})$, upper semimodularity is equivalent to modularity, so one disproves the upper semimodularity of $\Qt(U_{n,n})$ by producing a matroid $M$ with two elementary quotients $Q_1$ and $Q_2$ such that $Q_1$ and $Q_2$ do not have a common elementary quotient. The following is such an example.
 \begin{ex}\label{ex:counter-to-submodular} Let $V_8^-$ be a rank-4 matroid on $[8]$ whose 4-element circuits are $1234,3456,1256$, and $3478$. The notation suggests that it is a relaxation of the V\'{a}mos matroid $V_8$. The upward closure of $\{127,568,3478\}$ in $\cLL(V_8^-)$ and the upward closure of $\{12,34,56\}$ are both modular cuts. Let $Q_1$ and $Q_2$ be the elementary quotients of $V_8^-$ given by these modular cuts, respectively. The affine diagrams of $V_8^-,Q_1$ and $Q_2$ are given in \Cref{fig:counter-ex}. The common flats of $Q_1$ and $Q_2$ are $\{\emptyset,7,8,127,568,3478,[8]\}$, which do not contain the lattice of flats of any rank-2 matroid on $[8]$.

    \begin{figure}[H]
        \centering
        \includegraphics[width=4in]{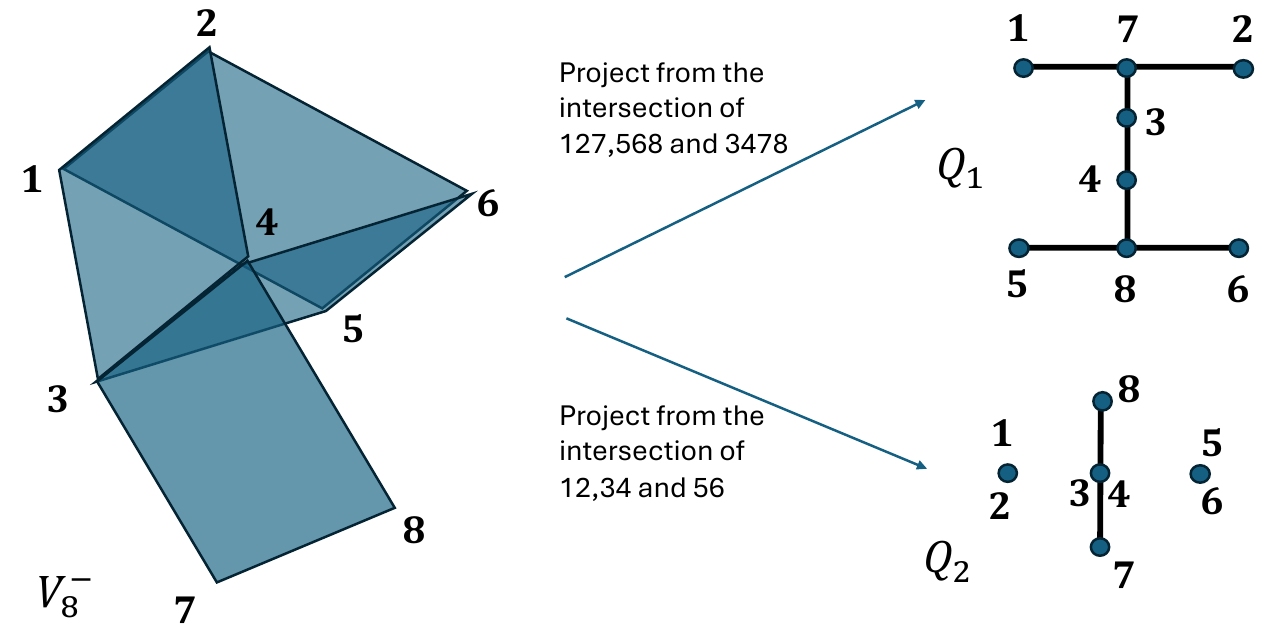}
        \caption{A realizable matroid with two elementary quotients not having expected intersection.}
        \label{fig:counter-ex}
    \end{figure}
 \end{ex}

Together with the following structural result about $\widehat{\Qt}^1(M)$, \Cref{ex:counter-to-submodular} implies that the upper semimodularity of $\Qt(U_{n,n})$ fails for every $n\geq 8$; moreover, the modularity of $\Qt(M)$, and thus the tropical analog of \ref{itm:IG-intersection}, fail for every $d\geq 4$.

\begin{prop}\label{prop:extension-lattice-direct-sum}
    Let $M$ be a matroid on $E$ and $N$ a matroid on $E'$. There is a canonical isomorphism of lattices
    \begin{equation}
        \kappa: \widehat{\Qt}^1(M)\times \widehat{\Qt}^1(N) \xrightarrow{\sim} \widehat{\Qt}^1(M\oplus N),\quad (Q_1,Q_2) \mapsto Q, 
    \end{equation}
    with inverse
    \begin{equation}
        \kappa^{-1}: \widehat{\Qt}^1(M\oplus N) \xrightarrow{\sim} \widehat{\Qt}^1(M)\times \widehat{\Qt}^1(N) ,\quad Q \mapsto (Q/E',Q/E).
    \end{equation}
\end{prop}

\begin{proof}
To define this map $\kappa$, we identify the elementary quotient lattices with lattices of linear subclasses, and put
\begin{equation}
        \kappa: \widehat{\Qt}^1(M)\times \widehat{\Qt}^1(N) \xrightarrow{\sim} \widehat{\Qt}^1(M\oplus N),\quad (\cHH,\cG) \mapsto \cHH\times \{E'\}\cup \{E\}\times \cG, 
    \end{equation}
We first show that $\cHH\times \{E'\}\cup \{E\}\times \cG$ is a linear subclass of $M\oplus N$, so $\kappa$ is well-defined. Take any two hyperplanes $L_1,L_2$ in $\cHH\times \{E'\}\cup  \{E\}\times\cG$. Up to relabeling, there are two possible cases.
    \begin{itemize}
        \item Case I: $L_1 = H_1 \sqcup E'$ and $L_2= H_2\sqcup E'$ for some $H_1,H_2\in \cHH$. Then any hyperplane covering $L_1\cap L_2$ is in $\cHH\times \{E'\}$, because $\cHH$ is a linear subclass of $M$.
        \item Case II: $L_1 = H\sqcup E'$ for some $H\in \cLL^1(M)$ and $L_2=E\sqcup G$ for some $G\in \cLL^1(N)$. Then $L_1\cap L_2 = H\sqcup G$. The only hyperplanes covering $L_1\cap L_2$ are $L_1$ and $L_2$.
    \end{itemize}
    Therefore, $\cHH\times \{E'\}\cup \{E\}\times\cG$ is a linear subclass. Clearly, $\kappa$ is injective and order-preserving. We show that it is surjective. If $\mathcal{L}$ is a linear subclass of $M\oplus N$, then $\mathcal{L}=\cHH\times \{E'\}\cup \{E\}\times\cG$ for some subset $\cHH\subset \cLL^1(M)$ and $\cG\subset \cLL^1(N)$. If $\{H_1,H_2,H_3\}\subset\cLL^1(M)$ is a concurrent triple and $H_1,H_2\in \cHH$ then $\{H_1\sqcup E',H_2\sqcup E',H_3\sqcup E'\}\subset\cLL^1(M\oplus N)$ is a concurrent triple in $\mathcal{L}$. Since $\mathcal{L}$ is a linear subclass, $H_3\sqcup E'\in \mathcal{L}$. Hence, $H_3\in \cHH$. For the same reason, $\cG$ is a linear subclass of $N$. Surjectivity is proved.

    If $Q$ is a quotient given by the linear subclass $\cHH\times \{E'\}\cup \{E\}\times \cG$, the contraction $Q/E'$ is the elementary quotient of $M$ given by the linear subclass $\cHH$. The full statement follows.
\end{proof}

\begin{prop}\label{prop:intersection-fail-at-every-d-larger-than-4}
    Let $M_1$ and $M_2$ be matroids on $E$ of the same positive rank and $N$ a matroid on $E'$. Then $M_1$ and $M_2$ have a common elementary quotient if and only if $M_1\oplus N$ and $M_2\oplus N$ have a common elementary quotient.
\end{prop}

\begin{proof}
    By \Cref{prop:extension-lattice-direct-sum},
    \begin{equation}
        \widehat{\Qt}^1(M_1\oplus N)\cap \widehat{\Qt}^1(M_2\oplus N) = \left(\widehat{\Qt}^1(M_1)\cap \widehat{\Qt}^1(M_2)\right) \times \widehat{\Qt}^1(N).
    \end{equation}
    Therefore, $\widehat{\Qt}^1(M_1)\cap \widehat{\Qt}^1(M_2)\neq\emptyset$ if and only if $\widehat{\Qt}^1(M_1\oplus N)\cap \widehat{\Qt}^1(M_2\oplus N)\neq\emptyset$.
\end{proof}

On the other hand, we show that $\Qt(M)$ is modular if $M$ has rank at most three. 
\begin{lemma}\label{lem:rank-3-matroid-construction}
    Let $\cS_1,\cS_2\subset 2^{[n]}\backslash\{\emptyset\}$ satisfy the following four conditions:
    \begin{enumerate}
        \item Elements in $\cS_1$ are pairwise disjoint; elements in $\cS_2$ are incomparable.
        \item For any $A\in \cS_1$ and $B\in\cS_2$, either $A\subsetneq B$, or $A\cap B = \emptyset$.
        \item If $A\in \cS_1$, $B_1$ and $B_2$ are distinct elements in $\cS_2$, and $A\subset B_1\cap B_2$, then $A=B_1\cap B_2$. 
        \item If $C\neq \emptyset$ and $C=\cap_{i=1}^k B_i$ for some $B_1,...,B_k\in \cS_2$, then $C=B_i\cap B_j$ for any $1\leq i<j\leq k$.
        \end{enumerate}
        Then there is a loopless rank-3 matroid $M$ such that $\cS_1\cup \cS_2\subset \cLL(M)$.
\end{lemma}

\begin{proof}
    Let $\cS_3$ be the set of non-empty pairwise intersections of the elements of $\cS_2$. By Condition (2) and (3), any element in $\cS_1$ and any element in $\cS_3$ are either equal or disjoint. By Condition (4), if $C\in \cS_3$ and $C\subset B_i\cap B_j$ for some $B_i,B_j\in \cS_2$,  then $C=B_i\cap B_j$. Note that elements in $\cS_3$ are pairwise disjoint. Moreover, an element of $\cS_3$ and an element of $\cS_1$ are either equal or disjoint.
    
    Note that we can impose the equivalence relation on $[n]$ defined by $\cS_1$ and $\cS_3$ without changing the validity of the four conditions. Therefore, we may assume that $\cS_1$ and $\cS_3$ consist of singletons. Consider the poset $\cP\subset 2^{[n]}$ under inclusion consisting of the following elements:
    \begin{itemize}
        \item $[n],\emptyset$, and elements in $[n]$;
        \item elements in $\cS_2$;
        \item elements in $\nk{n}{2}$ not contained in any element of $\cS_2$.
    \end{itemize}
    There are exactly two kinds of maximal chains in $\cP$. One kind is $\emptyset \subsetneq \{i\}\subsetneq S\subsetneq [n]$ for some $i\in n$, and for some $S\in \cS_2$ which is not a singleton; the other is $\emptyset \subsetneq \{i\}\subsetneq \{i,j\}\subsetneq [n]$ for some $\{i,j\}\subset [n]$ which is not contained in any element of $\cS_2$. Hence, $\cP$ is graded.

    We check that $\cP$ is a lattice.
    \begin{itemize}
        \item Closure under intersection: let $L_1$ and $L_2$ be two distinct rank-2 elements. If $L_1$ and $L_2$ both have size 2, then $|L_1\cap L_2|\leq 1$; if $L_1\in \cS_2$ and $L_2\notin \cS_2$, then by construction, $L_1$ is not contained in $L_2$, so $|L_1\cap L_2|\leq 1$; if $L_1,L_2\in \cS_2$, then $L_1\cap L_2\in \cS_3$, which by assumption is in $\cP$. This checks the closure under intersection.
        \item Well-definedness of join: since the intersection of any two rank-2 elements is either empty or a singleton, we see that for any two distinct elements $i,j\in [n]$, there is a unique rank-2 element containing $\{i,j\}$. This shows the join operation is well-defined.
        \end{itemize}
Submodularity and atomisticity are clear, so $\cP$ is a geometric lattice.
\end{proof}

Recall that the lattice of flats of a rank-2 matroid consists of the set of loops $A$, the ground set $[n]$, and subsets $A\sqcup B_1,...,A\sqcup B_k$, where $k\geq 2$, the $B_i$'s are nonempty and they partition $[n]\backslash A$. See \Cref{fig:cases-rank-2-perspectivity} for an illustration.

\begin{thm}\label{thm:rank-2-perspectivity}
Two rank 2 matroids $M_1,M_2$ are perspective if and only if they are coperspective.
\end{thm}

\begin{proof}
Suppose $M_1$ and $M_2$ are elementary quotients of $L$.  We may assume that $L$ is simple. If either of $M_1$ and $M_2$ is a principal truncation, then they are perspective by Cheung's compatible extension theorem \cite{cheung1974compatibility}. Assume neither of $M_1$ or $M_2$ is a principal truncation. Then necessarily they are both simple so $U_{1,n}$ is their common quotient. This proves the `if' direction.

\begin{figure}
    \centering
    \begin{adjustbox}{width=\columnwidth,center}
    \begin{tabular}{c|c}
       \begin{tikzcd}[row sep = tiny, column sep = 0.2em]
    & {[n]} & \\
        B_1\ar[ur, dash]  & \cdots & B_k\ar[ul, dash] \\
        & \ar[ul, dash] \ar[ur, dash] \boxed{\emptyset} &  \\
        & \cLL(M_1) &
    \end{tikzcd} \begin{tikzcd}[row sep = tiny, column sep = 0.2em]
    & {[n]} & \\
        D_1\ar[ur, dash]  & \cdots & D_l\ar[ul, dash] \\
        & \ar[ul, dash] \ar[ur, dash] \boxed{\emptyset} & \\
        & \cLL(M_2) & 
    \end{tikzcd} & \begin{tikzcd}[row sep = tiny, column sep = 0.2em]
    & {[n]} & \\
        \boxed{B_1}\ar[ur, dash]  & \cdots & B_k\ar[ul, dash] \\
        & \ar[ul, dash] \ar[ur, dash] \emptyset &  \\
        & \cLL(M_1) &
    \end{tikzcd} \begin{tikzcd}[row sep = tiny, column sep = 0.2em]
    & {[n]} & \\
        \boxed{C\sqcup D_1}\ar[ur, dash]  & \cdots & C\sqcup D_l\ar[ul, dash] \\
        & \ar[ul, dash] \ar[ur, dash] C & \\
        & \cLL(M_2) & 
    \end{tikzcd} \\ 
  Case I  &  Case II(a) \\\hline
    \begin{tikzcd}[row sep = tiny, column sep = 0.2em]
    & {[n]} & \\
        \boxed{A\sqcup B_1}\ar[ur, dash]  & \cdots & A\sqcup B_k\ar[ul, dash] \\
        & \ar[ul, dash] \ar[ur, dash] A &  \\
        & \cLL(M_1) &
    \end{tikzcd} \begin{tikzcd}[row sep = tiny, column sep = 0.2em]
    & {[n]} & \\
        \boxed{C\sqcup D_1}\ar[ur, dash]  & \cdots & C\sqcup D_l\ar[ul, dash] \\
        & \ar[ul, dash] \ar[ur, dash] C & \\
        & \cLL(M_2) & 
    \end{tikzcd} & \begin{tikzcd}[row sep = tiny, column sep = 0.2em]
    & {[n]} & \\
        \boxed{A\sqcup B_1}\ar[ur, dash]  & \cdots & A\sqcup B_k\ar[ul, dash] \\
        & \ar[ul, dash] \ar[ur, dash] A &  \\
        & \cLL(M_1) &
    \end{tikzcd} \begin{tikzcd}[row sep = tiny, column sep = 0.2em]
    & {[n]} & \\
        C\sqcup D_1\ar[ur, dash]  & \cdots & C\sqcup D_l\ar[ul, dash] \\
        & \ar[ul, dash] \ar[ur, dash] \boxed{C} & \\
        & \cLL(M_2) & 
    \end{tikzcd} \\
    Case II(b) & Case III
    \end{tabular}
    \end{adjustbox}
    \caption{The cases in the proof of \Cref{thm:rank-2-perspectivity}. The common proper flat of $M_1$ and $M_2$ is boxed.}
    \label{fig:cases-rank-2-perspectivity}
\end{figure}
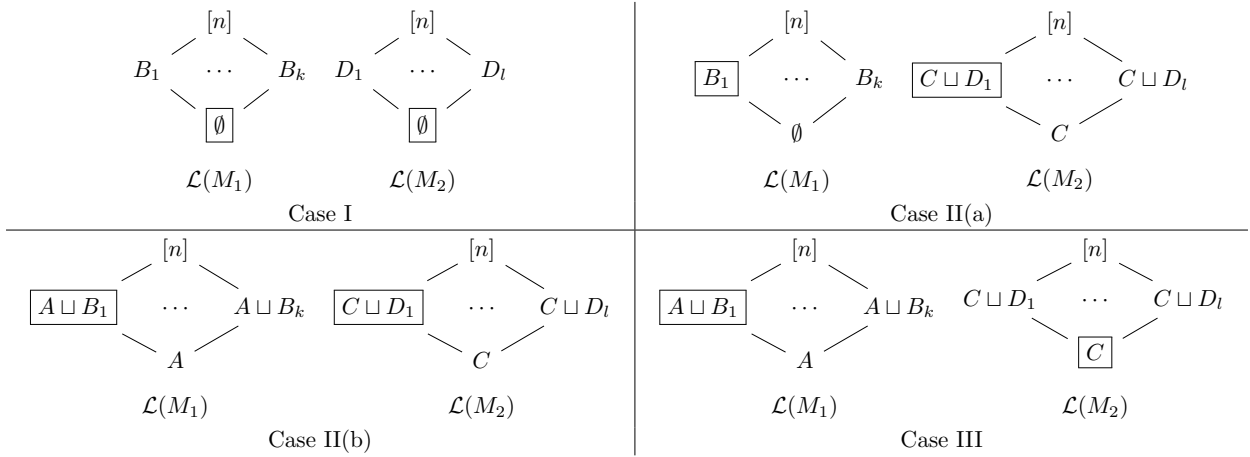 

Suppose $N$ is a common elementary quotient of $M_1$ and $M_2$. We need to construct a rank-3 matroid $M$ such that $\cLL(M)$ contains $\cLL(M_1)\cup\cLL(M_2)$. Since the rank of $N$ is 1, this simply means $M_1$ and $M_2$ share a proper flat. Let $M_1$ be given by $A\sqcup B_1,...,A\sqcup B_l$ and $M_2$ be given by $C \sqcup D_1,...,C\sqcup D_k$. Since all the subsets contain $A\cap C$, we may assume that $A\cap C=\emptyset$. There are three cases depending on which proper flat they share. In each case we construct collections $\cS_1$ and $\cS_2$. It is routine to verify that $\cS_1$ and $\cS_2$ satisfy the four conditions as in \Cref{lem:rank-3-matroid-construction} and that that $\cLL(M_1)\cup \cLL(M_2)\subset \cS_1\cup\cS_2\cup \{[n]\}$, so the conclusion follows from \Cref{lem:rank-3-matroid-construction}.
    \begin{itemize}
        \item Case I: $A=C=\emptyset$. Since we assume $A\cap C=\emptyset$, this implies $A=C=\emptyset$. Put
        \begin{align}
            & \cS_1=\{B_i\mid B_i\subset D_j\text{ for some }j\}\cup\{D_i\mid D_i\subset B_j\text{ for some }j\},\\
            & \cS_2 = \{B_1,...,B_k\}\cup \{D_1,...,D_l\}\backslash \cS_1.
        \end{align}        
        \item Case II: $A\neq C$ and $A\sqcup B_1 = C\sqcup D_1$. There are two subcases:
        \begin{itemize}
            \item Case II(a): $A=\emptyset$ and $C\neq \emptyset$. Put 
            \begin{equation}
                \begin{split}
                    \cS_1 =\{B_i\mid  B_i\subset D_j\text{ for some }j\}\cup\{C\},\quad 
                    \cS_2 = \{B_2,...,B_k,C\sqcup D_1,...,C\sqcup D_l\}\backslash \cS_1. 
                \end{split}
            \end{equation}
            \item Case II(b): Both $A$ and $C$ are non-empty. Note that $B_i\cap D_j=\emptyset$ if $i=1$ and $j\neq 1$ or $j=1$ and $i\neq 1$. Therefore,
            \begin{equation}
                (A\sqcup B_i)\cap (C\sqcup D_j)=\begin{cases}
                   A\sqcup B_1,& \text{ if }i=j=1 \\
                   A,&\text{ if }j=1,i\neq 1 \\
                   C,&\text{ if }i=1,j\neq 1 \\
                   B_i\cap D_j,&\text{ if }i\neq 1,j\neq 1.
                \end{cases}
            \end{equation}
            We can put
            \begin{equation}
                \cS_1=\{A,C\},\quad \cS_2=\{A\sqcup B_1,...,A\sqcup B_k,C\sqcup D_2,...,C\sqcup D_l\}.
            \end{equation}
            \end{itemize}

        \item Case III: $A\sqcup B_1=C$. Since we assumed $A\cap C=\emptyset$, this means $A=\emptyset$ and $C=B_1$. Put
        \begin{equation}
            \cS_1 = \{B_i\mid B_i\subset C\sqcup D_j\text{ for some }j\},\quad \cS_2 = \{B_2,...,B_k,C\sqcup D_1,...,C\sqcup D_l\}\backslash \cS_1.
        \end{equation}
        \end{itemize}
\end{proof}

\begin{cor}\label{cor:rank-3-modular}
    Let $M$ be a matroid of rank at most 3. Then  $\Qt(M)$ is a modular poset.
\end{cor}

\begin{proof}
    
All rank-3 matroids have adjoints, so the corollary follows from \Cref{thm:rank-2-perspectivity} and \Cref{thm:adjoint-discrete-interpolation}.
\end{proof}

The remainder of this section is devoted to proving \Cref{main:C}. As a preparation, we recall some facts on how tropical linear spaces change under deletion and contraction of the valuated matroids. For each nonempty $E\subset [n]$, put
\begin{equation}
    \PT{E} = \{[\bx]\in \PT{[n]}\mid \supp(\bx)\subset E\},\quad T_{E} = \{[\bx]\in \PT{[n]}\mid \supp(\bx)=E\}.
\end{equation}
These are the coordinate subspaces and subtori. Consider the projection map:
\begin{equation}
    \pi_1\colon \PT{[n]}\backslash \{[(0,\infty,...,\infty))]\} \onto \PT{[n]-1},\quad [(\bx(1),\bx(2),...,\bx(n))]\mapsto [(\infty,\bx(2),...,\bx(n))]
\end{equation}
Then for any valuated matroid $\mu$ on $[n]$, 
\begin{center} \vspace{5pt}
    $\pi_1(\Trop \mu)=\Trop(\mu\backslash 1)$ \quad and\quad  $\Trop\mu\cap H = \Trop (\mu/1)$.
\vspace{5pt} \end{center} 

\Cref{lem:injective-outside-of-contraction} refines the above observations. In particular, it says $\pi_1$ is injective away from the preimage of $\Trop(\mu/1)$. 

\begin{lemma}\label{lem:injective-outside-of-contraction}
Let $\mu$ be a loopless valuated matroid on $[n]$, $P=\Trop\mu$ and $P_\infty = \Trop (\mu/1)$. Let $[\bw]\in T_{[n]-1}$. If $[\bw]\in \pi_1(P)$, then there is some $[\bw']\in T_{[n]}\cap P$ such that $\pi_1([\bw'])=[\bw]$. Moreover, the following are equivalent.
\begin{enumerate} 
    \item 
    $[\bw]\in P_\infty$

    \item there is some $[\bw']\in \pi_1^{-1}([\bw])$ such that for all $x \geq 0$, $[\bw'+x \be_1]\in P$ and 1 is a coloop of the initial matroid $\mu^{\bw'+x\be_1}$.
    
    \item $\pi_1^{-1}([\bw])\cap P$ contains more than one point.
    \item there is some $[\bw']\in \pi_1^{-1}([\bw])\cap P$ such that 1 is a coloop of the initial matroid $\mu^{\bw'}$.
\end{enumerate}

\end{lemma}
\begin{proof}
    WLOG, assume
    \begin{equation}
        \bw = (\infty, 0, ..., 0 ).
    \end{equation}
    If $[\bw]\in \pi_1(P)$, then the minimizers of
    \begin{equation}\label{eqn:minimum-deletion}
        \min_{1\notin B}\{\mu(B)\}
    \end{equation}
    cover $[n]-1$. Let $a$ be the minimum value of \eqref{eqn:minimum-deletion}. Consider $\bw'=x\be_1$. Then for $x\in \R$, $\min_{1\in B}\{\mu(B)-\bw'\cdot \be_B\}$ is a continuous function in $x$ with range $\R$. Hence, we can choose $x$ such that $\min_{1\in B}\{\mu(B)-\bw'\cdot \be_B\}=a$. Then the minimizers of 
    \begin{equation}
        \min_{B}\{\mu(B)\}
    \end{equation}
    cover $[n]$. We have $[\bw']\in T_{[n]}\cap P$ and $\pi_1([\bw'])=[\bw]$.

    If $[\bw]\in P_\infty$, then the minimizers of 
    \begin{equation}
        \min_{1\in B} \{\mu(B)\}
    \end{equation}
    cover $[n]$. This means for all sufficiently large $x$, the minimizers of $\min_B\{\mu(B)-x\be_1\}$ cover $[n]$ and they all contain 1. This proves $(1)\Rightarrow (2)$. $(2)\Rightarrow (3)$ is clear. Suppose $\bw_1=x_1 \be_1$ and $\bw_2=x_2\be_1$ are two distinct points such that $x_1>x_2$ and $[\bw_1],[\bw_2]\in \pi_1^{-1}([\bw])$. Note that for all $x>x_2$, the minimizers of $\min\{\mu(B)-x\be_1\cdot\be_B\}$ are exactly the minimizers of $\min\{\mu(B)-\bw_2\cdot\be_B\}$ that contain 1. Since $[\bw_1]\in P$, the minimizers of $\min\{\mu(B)-x\be_1\cdot\be_B\}$ cover $[n]$ for all $x>x_2$, so $(3)\Rightarrow (4)$. The same argument shows that under the hypothesis of $(4)$, the minimizers of $\min_{1\in B}\{\mu(B)\}$ cover $[n]$, which means $[\bw]\in P_\infty$.
\end{proof}

The whole proof to \Cref{main:C} is notation-heavy, as it involves keeping track of many pieces of geometric data. However, the idea is simple, so we give a preview first. Let $\Trop\mu$ be a tropical plane and $\Trop\theta_1$ and $\Trop\theta_2$ two tropical lines in $\Trop\mu$. If either $\theta_1$ or $\theta_2$ has a loop $i$, then $\Trop(\theta_1/i)\cap \Trop(\theta_2/i)$ is nonempty. From here to the end of this section, we assume $\theta_1$ and $\theta_2$ are loopless and fix notations
\begin{center} \vspace{5pt}
    $P=\Trop\mu$ a tropical plane;\quad $L_1=\Trop\theta_1$ and $L_2=\Trop\theta_2$ two tropical lines in $P$;\\
    $P_\infty=\Trop(\mu/1)$;\quad $[\ba_\infty] = \Trop(\theta_1/1)$;\quad $[\bb_\infty] = \Trop(\theta_2/1)$;
\vspace{5pt} \end{center} 
Under the assumption that $\mu$ is simple and $\theta_1$ and $\theta_2$ are loopless, $P_\infty$ is a tropical line in $\PT{[n]}$ and $[\ba_\infty]$ and $[\bb_\infty]$ are points on $P_\infty$. A tropical line can be topologically represented by a tree. If the flats of $M$ containing 1 are $\{1\}\sqcup F_1,...,\{1\}\sqcup F_k$, then each $F_i$ labels an endpoint of a leaf of the tree. These endpoints are the valuated cocircuits of $P$; $[\ba_\infty]$ and $[\bb_\infty]$ are points on $P_\infty$, which may be the valuated cocircuits of $P$ or in the interior of $P_\infty$. This configuration is illustrated by \Cref{fig:tree-pic} and \Cref{subfig:t-infty}.

\begin{figure}
    \centering
    \begin{subfigure}[c]{0.3\textwidth}
        \includegraphics[width=1.8in]{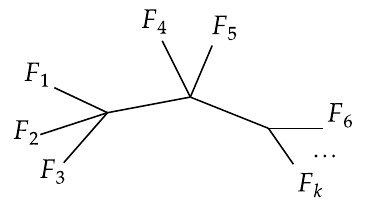} 
    \caption{The tropical line $P_\infty$.}\end{subfigure}\quad\quad 
    \begin{subfigure}[c]{0.4\textwidth}
        \centering
        \includegraphics[width=1.6in]{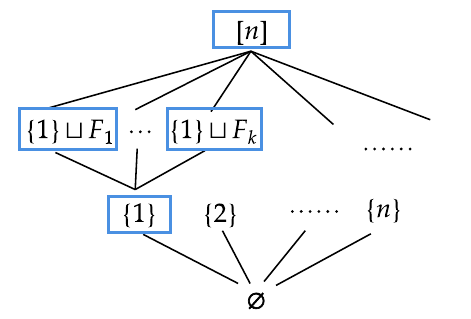}
        \caption{$\cLL(M)$. The flats containing 1 are boxed.}
    \end{subfigure}
    \caption{Representing the tropical line $P_\infty$ by a tree.}
    \label{fig:tree-pic}
\end{figure}

For each $t\in (-\infty,\infty]$, let $H_t$ be the hyperplane in $\PT{[n]}$ given by the equation
\begin{equation}
        \bx(1) = \bx(2) + t.
    \end{equation}
    The hyperplane $H_\infty$ is the coordinate hyperplane $\PT{[n]-1}$. As $t$ decreases from a sufficiently large number, $H_t$ sweeps through $L_1,L_2$ and $P$. We keep track of the stable intersections
\begin{equation}
        P_t =  P\stcap H_t,\quad [\ba_t] = L_1\stcap H_t,\quad [\bb_t] = L_2\stcap H_t,
    \end{equation}
    as well as the set-theoretic intersections\begin{equation}
        Q_t =  P\cap H_t,\quad X_t = L_1\cap H_t,\quad Y_t = L_2\cap H_t.
    \end{equation} 
    The sweeping hyperplane $H_t$ will witness the intersection of $L_1$ and $L_2$: either the two stable intersections $[\ba_t]=[\bb_t]$ for some $t$, or the set-theoretic intersections $X_t$ and $Y_t$ meet. This is most apparent when $1$ is a coloop of $\mu$. In this case, the hyperplane $H_t$ always meets $P$ transversely, so $P_t=Q_t$. The movements of $[\ba_t]$, $[\bb_t]$, $X_t$ and $Y_t$ are then confined in $P_t$. That $H_t$ will witness the intersection of $L_1$ and $L_2$ follows essentially from the balancing condition for tropical lines. To clarify this idea, we address this special case first. See \Cref{fig:sweeping-hyperplane} and \Cref{fig:proof-coloop-case} for illustrations accompanying the proof.

    \begin{figure}
        \centering
        \includegraphics[width=0.4\linewidth]{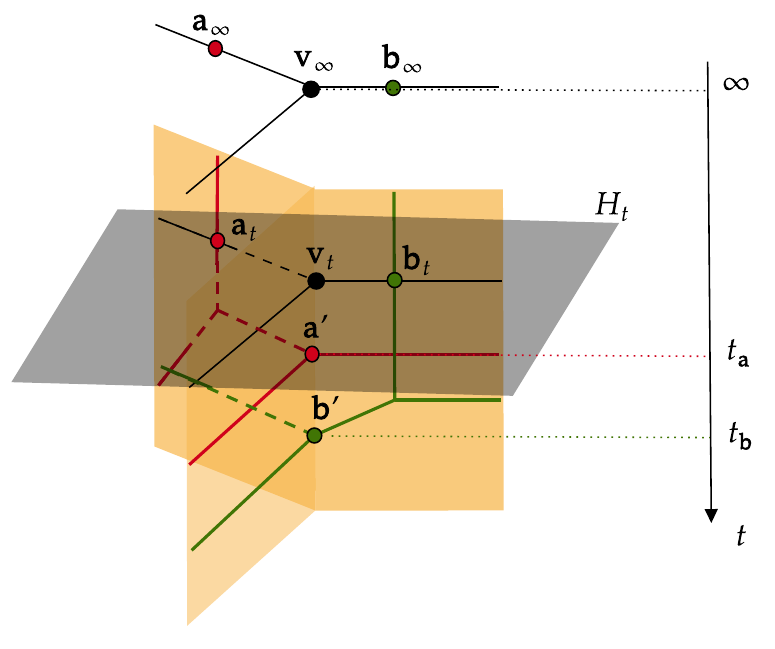}
        \caption{As $t$ decreases, the hyperplane will witness the intersection of $L_1$ and $L_2$ at some point.}
        \label{fig:sweeping-hyperplane}
    \end{figure}

\begin{prop}\label{prop:coloop-case}
    If $\mu$ has a coloop, then any two lines $L_1,L_2$ on $P=\Trop\mu$ intersect.
\end{prop}

\begin{proof}
    WLOG, suppose 1 is a coloop of $\mu$. If 1 is a coloop of both $\theta_1$ and $\theta_2$, then since we assume $\theta_1$ and $\theta_2$ are loopless, their lattice of flats are both $\{\emptyset, 1, [n]-1,[n]\}$. This means $L_1$ and $L_2$ both contain the valuated cocircuit of $\mu$ supported at $[n]-1$. This is where they intersect.

    Suppose now 1 is not a coloop of $\theta_1$. Recall that $H_t$ and $P$ intersect transversely for all $t$ and we have $P_t = Q_t$. Moreover, for each $t$, $P_t$ is simply a copy of $P_\infty$ shifted to $H_t$. Note that $\pi_1(L_1)=\pi_2(L_2)=P_\infty$. If $[\ba_\infty]=[\bb_\infty]$, we are done. Suppose $[\ba_\infty]\neq [\bb_\infty]$ and let 
    \begin{center} \vspace{5pt}
       $J_\infty:=$ the tropical convex hull of $[\ba_\infty]$ and $[\bb_\infty]$,
    \vspace{5pt} \end{center} 
     which is a tropical line segment in $P_\infty$. Let $[\bv_\infty]$ be any point in $J_\infty\cap T_{[n]-1}$ that is distinct from $[\ba_\infty]$ and $[\bb_\infty]$. Then there is $[\ba']\in L_1$ and $[\bb']\in L_2$ such that $\pi_1([\ba'])=\pi_1([\bb'])=[\bv_\infty]$. Working in the coordinate chart given by $\bx(2)=0$, we write
    \begin{equation}
        \bv_\infty = (\infty,0,\bv(3),...,\bv(n)),\quad  \ba' = (t_\ba,0,\bv(3),...,\bv(n)),\quad \bb' = (t_\bb,0,\bv(3),...,\bv(n)).
    \end{equation}
    Let $\bv_t=(t,0,\bv(3),...,\bv(n))$ and
            \begin{center} \vspace{5pt}
            $C_t:=$ the tropical convex hull of $[\ba_t]$ and $[\bv_t]$, \\
            $D_t:=$ the tropical convex hull of $[\bb_t]$ and $[\bv_t]$, \\
            $C'_t:=$ the tropical convex hull of $[\ba_t]$ and $[\ba']$, \\
            $D'_t:=$ the tropical convex hull of $[\bb_t]$ and $[\bb']$.
        \vspace{5pt} \end{center} 
        Note that
        \begin{itemize}
        \item $C'_\infty\subset L_1$;
            \item $[\ba_t]\in C'_\infty$ for all $t\geq t_\ba$;
            \item for $t_\ba\leq t_1<t_2<\infty$, $C'_{t_1}\subset  C'_{t_2}\subset C'_\infty$;
            \item $C_t$ is the projection of $C'_t$ onto $H_t$.
        \end{itemize}
          Similar observations hold for $[\bb_t],D_t$, and $D'_t$. If $t_\ba=t_\bb$, we are done. WLOG, suppose $t_\ba>t_\bb$. There are two possible situations at $t=t_\ba$, 

    \begin{itemize}
            \item $[\ba_{t_\ba}]\neq [\ba']$. Since $L_1$ contains both $[\ba_{t_\ba}]$ and $[\ba']=[\bv_{t_\ba}]$, it contains the whole tropical line segment $C_{t_\ba}$. Since $L_1$ is a tropical line, i.e., a degree-1 tropical cycle, the balancing condition and the degree-1 condition means that $L_1$ must contain $D_{t_\ba}$ as illustrated by \Cref{subfig:coloop-case-1}. In particular, $[\bb_{t_\ba}]\in  L_1$. 
            
            \item $[\ba_{t_\ba}] = [\ba']$. In this case, as $t$ decreases to $t_\bb$, either of the following happens,
            \begin{itemize}
                \item $[\ba_t]\in D_t$ for all $t$ such that $ t_\bb\leq  t \leq t_\ba$. In particular, $[\ba_t]\in D_{t_\bb}\subset L_2$. 
                \item for some $t_0 $ such that $ t_\bb <t_0\leq t_\ba$, $[\ba_{t_0}]\in D_{t_0}$ whereas $\ba_{t_0-\epsilon}\notin D_{t_0-\epsilon}$ for all $\epsilon>0$ small. In other words, $[\ba_t]$ exits $D_t$ at $t_0$. If $[\ba_{t_0}]=[\bb_{t_0}]$, we are done. Otherwise, the balancing condition at $[\ba_{t_0}]$ implies that $X_{t_0}=L_1\cap H_{t_0}$ must contain the tropical line segment connecting $[\ba_{t_0}]$ and $[\bb_{t_0}]$, as indicated by \Cref{subfig:coloop-case-2b}.
            \end{itemize}

        \end{itemize}
\end{proof}

\begin{figure}
    \centering
    \begin{subfigure}[c]{0.4\textwidth}
        \centering
        \includegraphics[width=2in]{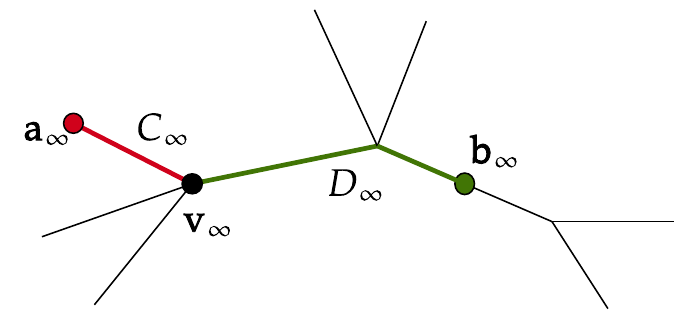}\caption{$t=\infty$}
        \label{subfig:t-infty}
    \end{subfigure}
    \quad
    \begin{subfigure}[c]{0.4\textwidth}
        \centering
        \includegraphics[width=2in]{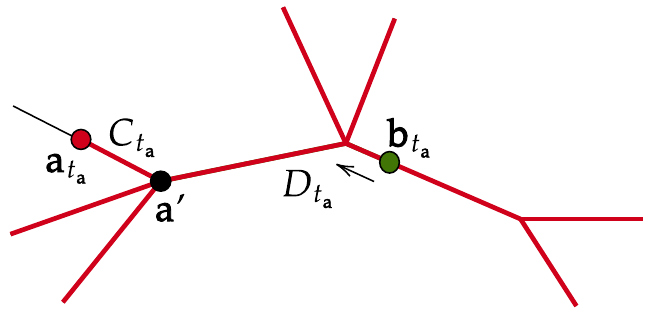}
        \caption{$t=t_\ba$ and $[\ba_{t_\ba}]\neq[\ba']$. Since $L_1$ contains $C_{t_\ba}$, $L_1$ must contain the rest of $P_{t_\ba}$ which is colored red.}\label{subfig:coloop-case-1}
    \end{subfigure}

    \begin{subfigure}[c]{\textwidth}
        \centering
        \begin{subfigure}[c]{0.4\textwidth}
        \centering
        \includegraphics[width=2in]{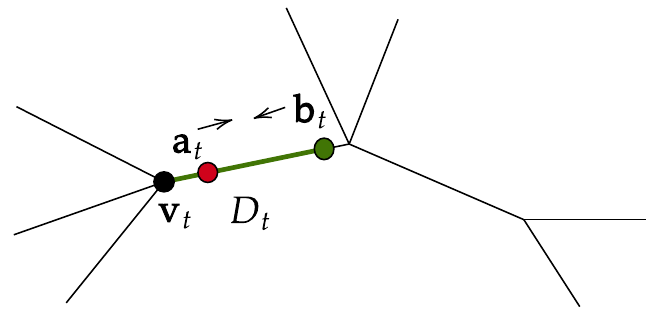}
        \caption{$[\ba_{t_\ba}]=[\ba']$ and $[\ba_t]\in D_{t}$ for all $t\in [t_\bb,t_\ba]$. Then we have $\ba\in D_{t_\bb}$.}
        \label{subfig:coloop-case-2a}
    \end{subfigure} \quad
    \begin{subfigure}[c]{0.4\textwidth}
        \centering
        \includegraphics[width=2in]{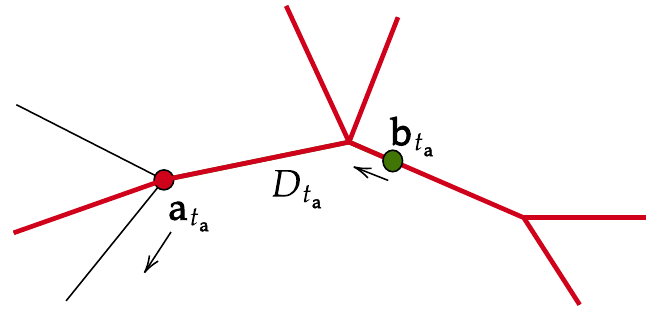}
        
        \caption{$[\ba_{t_\ba}]=[\ba']$ and $[\ba_t]$ exits $D_t$ after $t=t_0$. If $[\ba_{t_0}]\neq [\bb_{t_0}]$, $[\bb_{t_0}]$ must be in $L_1\cap H_{t_0}$, which is colored red.}
        \label{subfig:coloop-case-2b}
    \end{subfigure}
    \end{subfigure}
    
    \caption{An illustration of the dynamics of $[\ba_t]$ and $[\bb_t]$ on $P_t$ as $t$ decreases. Under the assumption that 1 is a coloop of $\mu$, $P_t$ remains unchanged as $t$ changes. The arrows indicate where the points $[\ba_t]$ and $[\bb_t]$ move.}
    \label{fig:proof-coloop-case}
\end{figure}

If $\mu$ does not have a coloop, then the set-theoretic intersection $Q_t$ may not be transverse, which means $L_1\cap H_t$ may not be contained in the tropical line $P_t$. To work around this complication, we use \Cref{lem:injective-outside-of-contraction} and induct on $n$. For the base case $n=3$, it is obvious that any two tropical lines in $\PT{[n]}$ intersect. When $n>3$, by the induction hypothesis, $\pi_1(L_1)$ and $\pi_1(L_2)$ intersect in $\pi_1(P)$. If there is some point $[\bv_\infty]\in \pi_1(L_1)\cap\pi_1(L_2)$ that is not in $P_\infty$, then by \Cref{lem:injective-outside-of-contraction}, $L_1$ and $L_2$ intersect. Therefore, we only need to show that $L_1$ and $L_2$ intersect if $\pi_1(L_1)\cap\pi_1(L_2)\subset P_\infty$. Let
\begin{equation}
    U = \pi_1^{-1}(P_\infty).
\end{equation}
Then $U$ is tropically convex, so the tropical line segments $C'_t$ and $D'_t$, as defined in the proof to \Cref{prop:coloop-case}, are contained in $U$ before $t$ reaches $t_\ba$ and $t_\bb$, respectively. Moreover, by \Cref{lem:injective-outside-of-contraction}, for every point $[\bw]$ in the interior of $U$, the initial matroid $\mu^{\bw}$ has a coloop 1, so $H_t$ intersects with $P$ transversely near $[\bw]$. This transverslity means that the portion of $L_1\cap H_t$ we care about is contained in $P_t$, and this is enough for our purpose.

\begin{repthm}{main:C}
    Two coperspective tropical lines are perspective. In other words, two tropical lines on a tropical plane always have nonempty  intersection.
\end{repthm}

\begin{proof}
    Recall the notations
     \begin{center} \vspace{5pt}
       $J_\infty:=$ the tropical convex hull of $[\ba_\infty]$ and $[\bb_\infty]$,\\
            $C_\infty:=$ the tropical convex hull of $[\ba_\infty]$ and $[\bv_\infty]$, \\
            $D_\infty:=$ the tropical convex hull of $[\bb_\infty]$ and $[\bv_\infty]$.
        \vspace{5pt} \end{center} 
      Note that the triple intersection of $C_\infty,D_\infty$ and $J_\infty$ must be non-empty, so we may assume that $[\bv_\infty]\in J_\infty$.
    
    Suppose first that there is a point $[\bv_\infty]\in \pi_1(L_1)\cap \pi_1(L_2)\cap T_{[n]-1}$. By \Cref{lem:injective-outside-of-contraction}, there is some $[\ba']\in L_1\cap T_{[n]}$ and some $[\bb']\in L_2\cap T_{[n]}$ such that $\pi_1([\ba'])=\pi_1([\bb']) = [\bv_\infty]$. Working in the coordinate chart given by $\bx(2)=0$, we write
    \begin{equation}
        \bv_\infty = (\infty,0,\bv(3),...,\bv(n)),\quad  \ba' = (t_\ba,0,\bv(3),...,\bv(n)),\quad \bb' = (t_\bb,0,\bv(3),...,\bv(n)).
    \end{equation}
    We may assume that the values $t_\ba$ and $t_\bb$ are maximal and are finite. Otherwise, 1 is a coloop of $\theta_1$ or $\theta_2$, which means 1 is a coloop of $\mu$. 
    
    If $t_\ba=t_\bb$, we are done. Suppose $t_\ba>t_\bb$. Recall the following notations from the proof of \Cref{prop:coloop-case}
            \begin{center} \vspace{5pt}
            $\bv_t=(t,0,\bv(3),...,\bv(n))$, \\
            $C_t:=$ the tropical convex hull of $[\ba_t]$ and $[\bv_t]$, \\
            $D_t:=$ the tropical convex hull of $[\bb_t]$ and $[\bv_t]$, \\
            $C'_t:=$ the tropical convex hull of $[\ba_t]$ and $[\ba']$, \\
            $D'_t:=$ the tropical convex hull of $[\bb_t]$ and $[\bb']$.
        \vspace{5pt} \end{center} 
       We still have the observations
        \begin{itemize}
        \item $C'_\infty \subset U\cap L_1$;
            \item $[\ba_t]\in C'_\infty$ for all $t\geq t_\ba$;
            \item for $t_\ba\leq t_1<t_2<\infty$, $C'_{t_1}\subset  C'_{t_2}\subset C'_\infty$;
            \item $C_t$ is the projection of $C'_t$ onto $H_t$.
        \end{itemize}
Moreover, by \Cref{lem:injective-outside-of-contraction}, if $t>t_\ba$, then for every point $[\bw]\in C_t\backslash [\ba_t]$, 1 is a coloop of the initial matroid $\mu^{\bw}$, so $H_t$ intersects with $P$ transversely near $[\bw]$. Similar observations hold for $[\bb_t],D_t$, and $D'_t$. 

There are two possible situations at $t=t_\ba$, 

    \begin{itemize}
            \item $[\ba_{t_\ba}]\neq [\ba']$. Then $C_{t_\ba}\subset L_1\cap H_{t_\ba} \subset Q_{t_\ba}$. Since $t_\ba>t_\bb$, at every $[\bw]\in D_{t_\ba}\backslash [\bb_{t_\ba}]$, $H_{t_\ba}$ intersects $P$ transversely, so $L_1$ must contain $D_{t_\ba}$. In particular, $[\bb_{t_\ba}]\in L_1$. 
            
            \item $[\ba_{t_\ba}] = [\ba']$. In this case, as $t$ decreases to $t_\bb$, either of the following happens,
            \begin{itemize}
                \item $[\ba_t]\in D_t$ for all $t$ such that $ t_\bb\leq  t \leq t_\ba$. In particular, $[\ba_{t_\bb}]\in D_{t_\bb}\subset L_2$. 
                \item for some $t_0 $ such that $ t_\bb <t_0\leq t_\ba$, $[\ba_{t_0}]\in D_{t_0}$ whereas $\ba_{t_0-\epsilon}\notin D_{t_0-\epsilon}$ for all $\epsilon>0$ small. If $[\ba_{t_0}]=[\bb_{t_0}]$, we are done. Otherwise, since for every $[\bw]\in D_{t_0}$, $H_{t_0}$ intersects with $P$ transversely, the balancing condition at $[\ba_{t_0}]$ implies that $L_1\cap H_{t_0}$ must contain the tropical line segment connecting $[\ba_{t_0}]$ and $[\bb_{t_0}]$.
            \end{itemize}
        \end{itemize}
         If $\pi_1(L_1)\cap \pi_1(L_2)$ is contained in the boundary of $P_\infty$, then a point $[\bv_\infty]\in \pi_1(L_1)\cap \pi_1(L_2)$ is a valuated cocircuit of $P$. We have $\supp (\bv_\infty) = [n]\backslash F$ for some hyperplane $F\in \cLL^1(M)$ containing 1. Moreover, either $[\ba_\infty]=[\bv_\infty]$ or $[\bb_\infty]=[\bv_\infty]$. Otherwise, as $[\bv_\infty]$ is chosen from $J_\infty$, we would have $[\bv_\infty]\in T_{[n]-1}$ if $[\bv_\infty]\notin \{[\ba_\infty],[\bb_\infty]\}$. Assume $[\ba_\infty]=[\bv_\infty]$.
    \begin{itemize}
        \item If $F=\{1,i\}$ for some $i$, 
        \begin{itemize}
            \item if $F$ is a flat of both $N_1$ and $N_2$, then $[\ba_\infty]=[\bb_\infty]=[\bv_\infty]\in L_1\cap L_2$;
            \item if $F$ is not a flat of $N_2$, then $[\bb_\infty]\neq [\ba_\infty]$. Since we assume $N_1$ is loopless, we may choose $[\bw]\in L_1\cap T_{[n]}$. Any tropical convex combination $[\bw']$ of $[\ba_\infty]$ and $[\bw]$ is in $L$. We can choose $[\bw']$ where
            \begin{equation}
                \bw' = (x,y,\ba_\infty(3),...,\ba_\infty(n))
            \end{equation}
            for some $x<\infty$ sufficiently large. Note that the minimizers of $\min_B\{\mu(B)-x\be_1-y\be_2\}$ must contain 1 for $x$ sufficiently large. This means $\pi_1([\bw'])\in P_\infty$. Moreover, for $x$ sufficiently large, $\pi_1([\bw])$ is on the leaf labeled by 2. Therefore, $\pi_1([\bw])\in \pi_1(L_2)$ and we are back in the situation where $[\bv_\infty]$ can be chosen from $T_{[n]-1}$.
        \end{itemize}

        \item If $F-1$ is not a singleton, since we assume that $\mu$ is simple, $F-1$ is not a flat of $M$. This means the only preimage of $[\bv_\infty]$ under $\pi_1$ is $[\bv_\infty]$. Hence, $L_1\cap L_2 = [\bv_\infty]$.
    \end{itemize}
\end{proof}

\section{From incidence geometry to Lorentzian polynomials}\label{sec:from-incidence-to-lorentzian}

We use the results obtained so far to study Lorentzian polynomials, leading to a proof of \Cref{main:B}. This relies on a result about interlacing polynomials. Let $f$ and $h$ be two real-rooted univariate polynomials of degree $d$ and $d-1$, respectively. We say $h$ \textit{interlaces} $f$ if in-between any two roots of $f$ there is a root of $h$. We say $h$ is a \textit{common interlacer} of a family of polynomials $\{f_1,...,f_k\}$ if $h$ interlaces $f_i$ for $i=1,...,k$.

\begin{prop}\label{prop:interlacing} \cite[Theorem 3.6]{chudnovsky2007roots} Let $f_1(x),...,f_k(x)$ be degree $d$ real-rooted univariate polynomials with positive leading coefficients. The following four statements are equivalent.
    \begin{enumerate}
        \item $c_1f_i+c_2f_j$ is real-rooted for all $1\leq i,j\leq k$ and all $c_1,c_2\geq 0$; 
        \item $f_i$ and $f_j$ have a common interlacer for all $1\leq i,j\leq k$;
        \item any convex combination of $f_1,...,f_k$ is real-rooted;
        \item $f_1,...,f_k$ have a common interlacer.
    \end{enumerate}
\end{prop}

\begin{lemma}\label{lem:stable-poly-cone}
    Let $f_1,...,f_k$ be a family of degree-$d$ stable polynomials with nonnegative coefficients such that $c_1f_i+c_2f_j$ is stable for any $c_1,c_2\geq 0$, then any convex combination of $f_1,...,f_k$ is stable.
\end{lemma}

\begin{proof}
   Let $f$ be any convex combination of $f_1,...,f_k$. The univariate polynomial $f(t\bu+\bv)$ is real-rooted for any $\bu\in \R_{\geq0}^n$ and any $\bv\in \R^n$ by applying \Cref{prop:interlacing} to $f_1(t\bu+\bv),...,f_k(t\bu+\bv)$.
\end{proof}

\begin{repthm}{main:B}
Let $\mu$ be a valuated matroid. Then for each $0<q\leq 1$, $\rmL[\ll_Lf^\mu_q]$ contains the convex cone spanned by 
\begin{equation}
    \rmL[\Dr^1(\mu);q] : = \big\{f^\theta_q\mid [\theta]\in \Dr^1(\mu)\big\} .
\end{equation}
\end{repthm}

\begin{proof}
    Let $f_q$ be any convex combination of $\{f_q^{\theta_i}\}$ for $[\theta_1],...,[\theta_k]\in \Dr^1(\mu)$. The M-convexity of $\supp (f_q^\mu+w_{n+1}f_q)$ follows from the tropical convexity of $\Dr^1(\mu)$. To show the signature condition of $f_q^\mu+w_{n+1}f_q$, the only non-obvious part is showing the signature condition for $f_q$. 

    Put $f_{ij}=c_1f_q^{\theta_i}+c_2f_q^{\theta_j}$ for $c_1,c_2\geq 0$. Take any $B\subset [n]$ such that $|B|=d-3$. We have
    \begin{equation}
        \partial^Bf_{ij} = c_1 \partial^B f_q^{\theta_i} + c_2 \partial^B f_q^{\theta_j} = c_1 f_q^{\theta_i/B} + c_2f_q^{\theta_j/B},
    \end{equation}
    if $B$ is independent in both $\theta_i$ and $\theta_j$, then $\theta_i/B$ and $\theta_j/B$ are coperspective rank-2 valuated matroids; otherwise, either $\partial^B f_q^{\theta_i}$ or $\partial^B f_q^{\theta_j}$ is zero. In the former case, by \Cref{main:C}, there is a rank-1 valuated matroid $\nu$ that is a quotient of both $\theta_i/B$ and $\theta_j/B$. This means $f_q^\nu\ll f_q^{\theta_i/B}$ and $f_q^\nu\ll f_q^{\theta_j/B}$. By \Cref{prop:convexity-stable-proper-position}, $\partial^B f_{ij}$ is stable. By \Cref{lem:stable-poly-cone}, this means $\partial^Bf_q$ is stable. Hence, $f_q$ is Lorentzian and the proof is completed.    
\end{proof}

 \Cref{main:B} specializes to basis generating polynomials of ordinary matroids.
\begin{cor}\label{cor:basis-generating-poly-cone}
    Let $M$ be a matroid. Then $\rmL[\ll_Lf_M]$ contains the convex cone spanned by $\{f_Q\mid Q\in \Qt^1(M)\}$.
\end{cor}

Put
\begin{equation}
\cone\hspace{-0.2em}\left\{\rmL[\Dr^1(\mu);q]\right\}:= \Big\{\text{ all nonnegative linear combinations of elements in }\rmL[\Dr^1(\mu);q]\Big\}.
\end{equation}
We have the containment of the three sets of Lorentzian polynomials: \begin{center} \vspace{5pt}
    $\rmL[\Dr^1(\mu);q]\subset\cone\hspace{-0.2em}\left\{\rmL[\Dr^1(\mu);q]\right\}\subset \rmL[\ll_Lf_q^\mu]$.
\vspace{5pt} \end{center} 
By \Cref{ex:L1-not-convex}, the second containment can be strict. Now we show that $\rmL[\Dr^1(\mu);q]$ are polyhedral fans defined by the exponentiation of the three-term incidence relations of $\Dr^1(\mu)$. As a consequence, $\rmL[\Dr^1(\mu);q]$ may not be convex in the usual sense, so the first containment can also be strict.

\begin{prop}\label{prop:image-of-dressian-in-lorentzian}
For each $0<q\leq 1$, the set of  polynomials $\rmL[\Dr^1(\mu);q]$, identified as tuples of coefficients $(a_B)\in \R_{\geq 0}^{\nk{n}{d-1}}$, is a polyhedral fan defined by the tropical vanishing of\footnote{Here tropical vanishing means the maximum is achieved at least twice.}
    \begin{equation}\label{eqn:multiplicative-incidence-relation}
        \max\left\{\quad q^{\mu(Djk)}a_{Di},\quad  q^{\mu(Dik)}a_{Dj},\quad q^{\mu(Dij)}a_{Dk} \quad \right\}
    \end{equation}
    for each $(d-2)$-set $D$ and $i,j,k\notin D$.
\end{prop}
\begin{proof}
    
    The base-$q$ logarithm is strictly monotone decreasing, so $\{a_{B}\}_{|B|=d-1}$ satisfies the tropical vanishing of \eqref{eqn:multiplicative-incidence-relation} if and only if $\{\log_q a_B\}_{|B|=d-1}$ satisfies the three-term incidence relations, i.e., if and only if $\theta\colon\nk{n}{d-1}\to\Rbar$ given by 
    \begin{equation}
        \theta(B) = \log_q(a_B)
    \end{equation}
    is in $\Dr^1(\mu)$.
\end{proof}
The case $\mu=\Trop M$ is special. The space $\rmL[\Dr^1(M);q]$ does not depend on $q$, because $\Dr^1_s(M)^\circ$ is a fan. We may write $\rmL[\Dr^1(M)]$ for $\rmL[\Dr^1(M);q]$. Moreover, by \Cref{prop:image-of-dressian-in-lorentzian}, up to flipping min and max, the set $\rmL[\Dr^1(M)]$ is cut out in $\Rbar_{\geq 0}^{\nk{n}{d-1}}$ by the same tropical equations that cut out $\Dr^1(M)$.

\begin{rmk}\label{rmk:extend-to-M-convex-fn}
There are \textit{polarization} and \textit{projection} operators on M-convex functions \cite{kobayashi2007operations}, analogous to those on Lorentzian polynomials, that preserve M-convexity. The polarization of an M-convex function turns an M-convex function into a valuated matroid, and it can be chosen to preserve elementary quotients. Since polarization and projection on Lorentzian polynomials preserve the Lorentzian property, one shows that \Cref{main:B} holds for arbitrary M-convex functions.
\end{rmk}

Motivated by the phenomenon for univariate real-rooted polynomials (\Cref{prop:interlacing} $(1)\Leftrightarrow(2)$), one may ask the following question: if $f_1$ and $f_2$ are Lorentzian polynomials of the same degree and all convex combinations of $f_1$ and $f_2$ are Lorentzian, is there a Lorentzian polynomial $g$ such that $g\ll_L f_1$ and $g\ll_L f_2$? The following example shows this is not true.

\begin{ex}\label{ex:converse-cone-property-false}
Recall the matroids $V_8^-$, $Q_1$ and $Q_2$ in \Cref{ex:counter-to-submodular}. By \Cref{main:B}, any convex combination of $f_{Q_1}$ and $f_{Q_2}$ is Lorentzian. However, if there were some Lorentzian polynomial $h$ such that $h\ll_L f_{Q_1}$ and $h\ll_Lf_{Q_2}$, then $\supp h$ would be an elementary quotient of both $Q_1$ and $Q_2$, which is impossible.
\end{ex}

\begin{appendices}

\section{Plethystic maps of Grassmannians}\label{subsec:cofactor-grassman}

In this appendix, we study the classical geometry behind adjoints of valuated matroids. More precisely, we study the embedding
\begin{equation}
   \Gamma: \Gr(d,E) \to \Gr(d,\wedge^{d-1}E), \quad V\mapsto \wedge^{d-1} V.
\end{equation}
In general, one can consider the embedding
\begin{equation}
    \Gr(d,E) \to \Gr\left(\binom{d}{k},\wedge^kE\right), \quad V\mapsto \wedge^k V,
\end{equation}
for each $1\leq k\leq d$. When $k=d$, this is the usual Pl\"{u}cker embedding; when $k=d-1$, this is $\Gamma$. Such constructions fall into the notion of \textit{geometric plethysms} \cite[Section 11]{fulton2013representation}. We thus call such maps \textit{plethystic maps}.
 
 The main result is \Cref{prop:relation-two-plucker-embedding}, which relates the usual Pl\"{u}cker embedding of $\Gr(d,E)$ and the Pl\"{u}cker embedding of the image of  $\Gr(d,E)$ under $\Gamma$. Recall a $d$-by-$d$ matrix $A$ and its cofactor matrix $B$ satisfy the following relation.
\begin{equation}\label{eqn:usual-cofactor}
 \det(B)=\det(A)^{d-1}.
\end{equation}
We generalize this relation to $d$-by-$n$ matrices where $d\leq n$. Let $A$ by a $d$-by-$n$ matrix where $d\leq n$. We use the following temporary notations: for each $i\in [d]$ and each $J\in \nk{n}{d-1}$, put
\begin{center} \vspace{5pt}
    $A_{i,J}:=$ the $(d-1)$-by-$(d-1)$ minor of $A$ consisting of the columns of $A$ labeled by $J$ \\ and all but the $i$-th row.
\vspace{5pt} \end{center} 
For each $I\in \nk{n}{d}$, put
\begin{center} \vspace{5pt}
    $A_I:=$ the $d$-by-$d$ minor of $A$ consisting of the columns of $A$ labeled by $I$.
\vspace{5pt} \end{center} 

\begin{definition}
    Let $A$ be any $d$-by-$n$ matrix where $d\leq n$. Let $\prec$ be any total order on $\nk{n}{d-1}$. The \textit{generalized cofactor} matrix of $A$ with respect to the total order $\prec$ is the $d$-by-$\binom{n}{d-1}$ matrix $B$ defined as follows:
    \begin{itemize}
        \item The $j$-th column of $B$ is labeled by the $j$-th element in $\nk{n}{d-1}$ in the total order $\prec$;
        \item  the $i$-th row of $B$ is labeled by $[d]-i$;
        \item The $(i,J)$-th entry of $B$ is $A_{i,J}$.  
    \end{itemize}
\end{definition}

\begin{prop}\label{prop:cofactor-identity}
    Let $B$ be the generalized cofactor matrix of $A$ with respect to the total order $\prec$. Let $\cJ=\{J_1,...,J_d\}\subset \nk{n}{d-1}$ where $J_1=\{i_1,...,i_{d-1}\}$ and $J_1\prec J_2\prec\cdots \prec J_d$. Then
    \begin{equation}\label{eqn:determinantal-identity}
\begin{split}
    B_{\cJ} & = \det \begin{bmatrix}    (-1)^{\chi_{J_2,i_1}+1}A_{J_2+i_1} & (-1)^{\chi_{J_3,i_1}+1}A_{J_3+i_1} & \cdots & (-1)^{\chi_{J_d,i_1}+1}A_{J_d+i_1} \\
    (-1)^{\chi_{J_2,i_2}+2}A_{J_2+i_2} & (-1)^{\chi_{J_3,i_2}+2}A_{J_3
+i_2} & \cdots & (-1)^{\chi_{J_d,i_2}+2}A_{J_d+i_2} \\
\vdots & \vdots & \ddots & \vdots \\
(-1)^{\chi_{J_2,i_{d-1}}+d-1}A_{J_2+i_{d-1}} & (-1)^{\chi_{J_3,i_{d-1}}+d-1}A_{J_3+i_{d-1}} & \cdots & (-1)^{\chi_{J_d,i_{d-1}}+d-1}A_{J_d+i_{d-1}}
\end{bmatrix} \\
    & = \det\begin{bmatrix}
         (-1)^{\chi_{J_k,i_l}+l}A_{J_k+i_l}
     \end{bmatrix}_{1\leq l\leq d-1,2\leq k\leq d}
\end{split}
\end{equation}
where 
\begin{equation}
    \chi_{J_k,i_l} = \text{ the cardinality of }\{j\in J_k\mid j<i_l\}.
\end{equation}
\end{prop}

\begin{proof}
    Since $B_{\cJ}$ is a continuous function of entries of $A$, we may assume that every $d$-by-$d$ minor of $A$ is nonzero. By adding repeated columns to $A$, we may also assume that the $J_k$'s are pairwise disjoint.

    Assume now that $J_1=[d-1]$. If $C$ is a $d$-by-$d$ invertible matrix, then the action of $C$ on $A$ induces the same action on the two sides of \eqref{eqn:determinantal-identity}. That is, multiplication by $(\det C)^{d-1}$. Hence, we may assume that the $d$-by-$d$ submatrix of $A$ with columns $J_1$ is an identity matrix. In this case, the equality follows from a direct computation.
\end{proof}

\begin{ex}\label{ex:cofactor-identity}
    When $n=d$, \Cref{prop:cofactor-identity} recovers the familiar Equation \eqref{eqn:usual-cofactor}. We illustrate \Cref{prop:cofactor-identity} and its proof using a 3-by-4 matrix:    
    \begin{equation}
        A = \begin{bmatrix}
            1 & 0 & 0 & a_{14} & a_{15} & a_{16} \\
             0 & 1 & 0 & a_{24} & a_{25} & a_{26}\\
              0 & 0 & 1 & a_{34} & a_{35} & a_{36}
        \end{bmatrix}
    \end{equation}
    Choose a total order $\prec$ on $\nk{n}{2}$ such that $12\prec 34\prec 56$. The submatrix with columns $\{12,34,56\}$ of the generalized cofactor matrix $B$ of $A$ is 
    \begin{equation}
        \begin{blockarray}{cccccc}
 & 12 & \cdots & 34 & \cdots & 56 \\
\begin{block}{c[ccccc]}
  23 & 0 & \cdots & \begin{vmatrix}
      0 & a_{24} \\
      1 & a_{34}
  \end{vmatrix} & \cdots & \begin{vmatrix}
      a_{25} & a_{26} \\
      a_{35} & a_{36}
  \end{vmatrix} \\
  & \\
  13 & 0 & \cdots & \begin{vmatrix}
      0 & a_{14} \\
      1 & a_{34}
  \end{vmatrix} & \cdots & \begin{vmatrix}
      a_{15} & a_{16} \\
      a_{35} & a_{36}
  \end{vmatrix} \\
  & \\
  12 & 1 & \cdots & 0 & \cdots & \begin{vmatrix}
      a_{15} & a_{16} \\
      a_{25} & a_{26}
  \end{vmatrix} \\
\end{block} 
\end{blockarray}.
\end{equation}
Indeed, one has
\begin{equation}
    B_{12,34,56} = \det \begin{bmatrix}
        \begin{vmatrix}
      0 & a_{24} \\
      1 & a_{34}
  \end{vmatrix} & \begin{vmatrix}
      a_{25} & a_{26} \\
      a_{35} & a_{36}
  \end{vmatrix} \\ & \\
  \begin{vmatrix}
      0 & a_{14} \\
      1 & a_{34}
  \end{vmatrix} & \begin{vmatrix}
      a_{15} & a_{16} \\
      a_{35} & a_{36}
  \end{vmatrix}
    \end{bmatrix} = \det\begin{bmatrix}
        A_{134} & A_{156} \\
        -A_{234} & -A_{256}
    \end{bmatrix}.
\end{equation}
\end{ex}

Let $V$ be a $d$-dimensional vector subspace of an $n$-dimensional vector space $E$. Fix an ordered basis $\{e_1,...,e_n\}$ of $E$. Choose a total order $\prec$ on the induced bases $\{e_{i_1}\wedge \cdots \wedge e_{i_{d-1}}\}_{i_1<\cdots <i_{d-1}}$ of $\wedge^{d-1}E$. The homogeneous coordinates of $\PP(\wedge^d(\wedge^{d-1} E))$ are then labeled by $\cJ=\{J_1,J_2,...,J_d\}\subset\nk{n}{d-1}$, where $J_1\prec J_2\prec\cdots\prec J_d$, and $J_1=\{i_1,...,i_{d-1}\}$ where $i_1<\cdots < i_{d-1}$.

We construct embeddings 
 \begin{equation}
     \Phi_\prec\colon \mathbb{P}(\wedge^dE) \hookrightarrow \mathbb{P}(\wedge^d(\wedge^{d-1}E))
 \end{equation}
such that the diagram
\begin{equation}\label{diagram:cofactor}
      \begin{tikzcd}
     \Gr(d,E) \ar[r,hook,"\Gamma"] \ar[d,hook,"\text{Pl\"{u}cker}"] & \Gr(d,\wedge^{d-1} E) \ar[d,hook,"\text{Pl\"{u}cker}"] \\
     \mathbb{P}(\bigwedge^dE) \ar[r,hook,"\Phi_\prec"] & \mathbb{P}(\bigwedge^d(\wedge^{d-1}E))
 \end{tikzcd}
 \end{equation}
 commutes. For each point $[\bx]\in \PP(\wedge^dE)$, define the $\cJ$-th homogeneous coordinate of $\Phi_\prec([\bx])$ to be 
 \begin{equation}
     \det\begin{bmatrix}
         (-1)^{\chi_{J_k,i_l}+l}\bx(J_k+i_l)
     \end{bmatrix}_{1\leq l\leq d-1,2\leq k\leq d},
 \end{equation}
where we set \begin{equation}
    \bx(I) = 0
\end{equation}
if the multi-set $I$ contains repeated elements. Namely, $\Phi_\prec$ is defined using the same formula as in \eqref{eqn:determinantal-identity}. 
If $\cJ$ is the set of $d$ subsets of size $d-1$ of some $d$-set $A$, then the $\cJ$-th coordinate of $\Phi_\prec([\bx])$ is $\bx(A)^{d-1}$ modulo a sign. Hence, $\Phi_\prec$ is an embedding.

\begin{prop}\label{prop:relation-two-plucker-embedding}
    For any choice of the total order $\prec$, the diagram \eqref{diagram:cofactor} commutes. 
\end{prop}

\begin{proof}
    This follows from the coordinate description of the Pl\"{u}cker embedding and \Cref{prop:cofactor-identity}.
\end{proof}
\Cref{prop:relation-two-plucker-embedding} relates the Pl\"{u}cker coordinate of a projective linear subspace $\PP V$ with the Pl\"{u}cker coordinate of the projective linear subspace of all hyperplanes in $\PP V$. \Cref{lem:Sigma-basis-valuation} is a tropical analog of this relation.

\end{appendices}

\printbibliography

\end{document}